\documentclass[a4paper]{amsart}
\usepackage[dvipsnames]{xcolor}
\usepackage{amsmath,amssymb,amsthm,mathrsfs,mathtools}
\usepackage{graphicx}
\usepackage{enumerate,enumitem} 
\usepackage{hyperref}
 \hypersetup{
     citecolor=CadetBlue,
     colorlinks=true,
     linkcolor=Periwinkle,
     filecolor=Periwinkle,      
     urlcolor=black,
     }
\usepackage{fancyhdr}
\usepackage{overpic}
\usepackage{orcidlink}
\usepackage{faktor}
\usepackage{xfrac}

\DeclareRobustCommand{\SkipTocEntry}[5]{}
\newtheorem{theorem}{Theorem}[section]
\newtheorem{proposition}[theorem]{Proposition}
\newtheorem{lemma}[theorem]{Lemma}
\newtheorem*{lemma*}{Lemma}
\newtheorem*{claim*}{Claim}
\newtheorem*{definition*}{Definition}
\newtheorem{cor}[theorem]{Corollary}
\theoremstyle{definition}

\newtheorem{claim}[theorem]{Claim}
\newtheorem{definition}[theorem]{Definition}
\newtheorem{remark}[theorem]{Remark}

\newtheorem*{theorem**}{Theorem\theoremnum}

\newenvironment{theorem*}[1][]{%
  \if\relax\detokenize{#1}\relax
    \begin{theorem}
  \else
    \renewcommand{\thetheorem}{#1}
    \begin{theorem}%
    \addtocounter{theorem}{-1}
  \fi
}{%
  \end{theorem}
}

\newenvironment{cor*}[1][]{%
  \if\relax\detokenize{#1}\relax
    \begin{cor}
  \else
    \renewcommand{\thetheorem}{#1}
    \begin{cor}%
    \addtocounter{theorem}{-1}
  \fi
}{%
  \end{cor}
}

\newtheorem*{proposition**}{Proposition\theoremnum}

\newenvironment{proposition*}[1][]{%
  \if\relax\detokenize{#1}\relax
    \begin{proposition}%
  \else
    \renewcommand{\thetheorem}{#1}
    \begin{proposition}%
    \addtocounter{theorem}{-1}
  \fi
}{%
  \end{proposition}
}

\newcommand{\addQEDstyle}[2]{\AtBeginEnvironment{#1}{\pushQED{\qed}\renewcommand{\qedsymbol}{#2}}\AtEndEnvironment{#1}{\popQED}}

\addQEDstyle{example}{$\triangle$}

\renewcommand{\mod}{\;\mathrm{mod}\,}
\DeclareMathOperator{\std}{std}

\DeclareMathOperator{\ima}{Im}

\DeclareMathOperator{\h}{\mathfrak{h}}

\DeclareMathOperator{\Mod}{Mod}

\DeclareMathOperator{\Schw}{Schw}

\DeclareMathOperator{\sech}{sech}

\DeclareMathOperator{\arccosh}{arccosh}
\DeclareMathOperator{\arcsinh}{arcsinh}

\DeclareMathOperator{\area}{area}

\DeclareMathOperator{\Isom}{Isom}

\DeclareMathOperator{\sys}{sys}
\DeclareMathOperator{\gr}{gr}
\DeclareMathOperator{\Grr}{Gr}
\DeclareMathOperator{\dwp}{d_{WP}}
\DeclareMathOperator{\dhex}{d_{\mathscr{H}}}
\DeclareMathOperator{\dflip}{d_{\mathscr{F}}}
\DeclareMathOperator{\WP}{WP}

\DeclareMathOperator{\dteich}{d_{T}}
\DeclareMathOperator{\Th}{Th}
\DeclareMathOperator{\dhyp}{d_{\mathbb{H}}}
\DeclareMathOperator{\stret}{stretch}
\DeclareMathOperator{\shear}{shear}
\DeclareMathOperator{\teich}{Teich}
\DeclareMathOperator{\inj}{inj}
\DeclareMathOperator{\PSL}{PSL}
\DeclareMathOperator{\FN}{FN}
\DeclareMathOperator{\Hex}{Hex}

\DeclareMathOperator{\Twist}{Tw}
\DeclareMathOperator{\spine}{spine}
\DeclareMathOperator{\interior}{int}

\newcommand{\dteichpath}{\widehat{\mathrm{d}}^{\;\epsilon}_{\mathrm{T}}}

\newcommand{\C}{\mathbb{C}}
\newcommand{\CP}{\mathbb{CP}}

\newcommand{\N}{\mathbb{N}}
\newcommand{\R}{\mathbb{R}}
\newcommand{\Hh}{\mathbb{H}}
\newcommand{\Z}{\mathbb{Z}}

\newcommand{\Sph}{\mathbb{S}}
\newcommand{\D}{\mathbb{D}}
\newcommand{\As}{\mathcal{A}}

\newcommand{\Cs}{\mathcal{C}}

\newcommand{\Fs}{\mathscr{F}}

\newcommand{\Hs}{\mathscr{H}}

\newcommand{\Ps}{\mathscr{P}}

\newcommand{\restr}[1]{|_{#1}}

\newcommand{\Zero}{\mathbf{0}}
\newcommand{\wabstract}{w}
\newcommand{\wstd}{w^{\std}}
\newcommand{\wtrunc}{w^{T}}

\newcommand{\coloneq}{\mathrel{\resizebox{\widthof{$\mathord{=}$}}{\height}{ $\!\!\resizebox{1.2\width}{0.8\height}{\raisebox{0.23ex}{$\mathop{:}$}}\!\!=\!\!$ }}}

\newenvironment{proofclaim}{%
  \begin{proof}[Proof of the claim]
  \renewcommand{\qedsymbol}{$\diamond$}%
}{%
  \end{proof}
}

\title{Hexagon decompositions and the Weil--Petersson metric}
\author{Marie Abadie \orcidlink{0000-0002-8951-4643}}
\email{marie.abadie@uni.lu}
\date{\today}

\begin{document}

\begin{abstract}
    We give an explicit quasi-isometry from the flip graph of triangulations $\Fs_{g,n}$ to the $\epsilon$-thick part of Teichmüller space equipped with the Teichmüller metric, by mapping each ideal triangulation to the hyperbolic surface whose shearing coordinates along that triangulation are all equal to zero. 
    Extending this construction, we obtain a quasi-isometry $Q$ from the hexagon graph $\Hs_{g,n}$ to the augmented Teichmüller space $\overline{\teich}_{g,n}$ equipped with the Weil--Petersson metric, and we estimate its width, that is, the Hausdorff
    distance between $Q(\Hs_{g,n})$ and $\overline{\teich}_{g,n}$.

    By bounding the Weil--Petersson distance along grafting rays, and by using an analogue for ideal triangulations of Wolpert's upper bound on the Weil--Petersson distance to a stratum, we show that the width is at most of order $\sqrt{(g+n)\log(g+n)}$. 
    We also provide an explicit projection from any point in $\overline{\teich}_{g,n}$ to the thick part of boundary strata, maintaining simultaneous control over the Weil--Petersson distance and the shearing coordinates.
\end{abstract}
  
\maketitle

\tableofcontents 

\section{Introduction}
We consider two metrics on the Teichmüller space $\teich_{g,n}$ of an oriented surface $\Sigma$ of genus $g$ with $n$ punctures and negative Euler characteristic. 
The first is the \emph{Weil--Petersson (WP) metric} $\dwp$, a negatively curved~\cite{Royden1975,Tromba1986,Wolpert1986a} Kähler metric on Teichmüller space~\cite{Ahlfors1961}. It is geodesically convex~\cite{Wolpert1987} but not complete~\cite{Wolpert1975,Chu1976}. Its completion, the \emph{augmented Teichmüller space} $\overline{\teich}_{g,n}$, is obtained by adding strata corresponding to noded surfaces which arise as limits of sequences of hyperbolic structures where the hyperbolic lengths of some collection of disjoint simple closed geodesics go to zero~\cite{Bers1974,Abikoff1977,EarleMarden2012}. 
The \emph{Teichmüller metric} $\dteich$ between two hyperbolic metrics is the infimum of the conformal distortion among all quasiconformal homeomorphisms homotopic to the identity. The two metrics are related: the Weil--Petersson distance is bounded above by a multiple of the Teichmüller distance~\cite{Linch1974}. 

These metrics remain difficult to study directly and to compute explicitly. A standard approach is therefore to model them using combinatorial distances on graphs~\cite{MasurMinsky1999, Brock2003,CavendishParlier2012, RafiTao2013}. For the Weil--Petersson metric, this strategy was pioneered by Brock, who proved that the pants graph is quasi-isometric to $\bigl(\overline{\teich}_{g,n},\dwp\bigr)$~\cite{Brock2003}. In this paper, we focus on two other combinatorial models: the \emph{flip graph of triangulations} $\Fs_{g,n}$ and the \emph{hexagon graph} $\Hs_{g,n}$. 

\addtocontents{toc}{\SkipTocEntry}
\subsection{Two combinatorial models}
The flip graph of triangulations, defined when $n\geq 1$ and denoted by $\Fs_{g,n}$, has vertices given by isotopy classes of ideal triangulations of $\Sigma$. Two ideal triangulations are related by an edge if they differ by one flip move~\cite{KorkmazPapadopoulos2012}. The mapping class group $\Mod(\Sigma)$ acts on $\Fs_{g,n}$ properly discontinuously and cocompactly, hence it follows from the Schwarz--Milnor lemma that $\Fs_{g,n}$ is quasi-isometric to the mapping class group with the word metric~\cite{Mosher1995,KorkmazPapadopoulos2012,DisarloParlier2018,DisarloParlier2019}. In fact, the same argument applies to the \emph{$\epsilon$-thick part} $\teich^{\epsilon}_{g,n}$ of Teichmüller space, consisting of those hyperbolic structures with no essential closed geodesic of length less than $\epsilon$, equipped with the path metric $\dteichpath$ induced by $\dteich$~\cite{LackenbyPurcell2024}. Here $0<\epsilon\leq\epsilon_0$, where $\epsilon_0=2\arcsinh(1)$ is the Margulis constant. The flip and pants graphs capture different geometric properties, hence combining them in a single model provides a refined combinatorial approach for understanding the geometry of Teichmüller space.

To combine the flip and pants models, we consider not only ideal triangulations whose arcs end in cusps, but also triangulations admitting bi-infinite arcs spiralling around simple closed curves. Such spiralling triangulations are naturally associated with hexagon decompositions. The \emph{hexagon graph} $\Hs_{g,n}$, introduced by Gültepe and Parlier~\cite{GultepeParlier2025}, has vertices given by hexagon decompositions: pairs $(\Gamma, \mathcal{A})$ where $\Gamma$ is a simple multicurve on $\Sigma$ and $\As$ is a maximal disjoint family of orthogeodesic arcs (in particular $\Sigma\smallsetminus(\Gamma\cup\As)$ is a disjoint union of generalized right-angled hexagons).
Gültepe--Parlier showed that $\Hs_{g,n}$ is quasi-isometric to the pants graph, hence, by Brock's theorem, to $\bigl(\overline{\teich}_{g,n},\dwp\bigr)$~\cite{GultepeParlier2025}.

In fact, the hexagon graph interpolates between the two previous models: taking $\Gamma$ maximal recovers pants decompositions, and taking $\Gamma=\emptyset$ recovers ideal triangulations. This suggests studying the relation between the hexagon graph and Teichmüller space directly, rather than through the intermediate quasi-isometry with the pants graph. In this paper we describe explicit maps realising this quasi-isometry, together with upper bounds on the width of each model. Along the way we prove explicit distance estimates for the Teichmüller and Weil--Petersson metrics.

\addtocontents{toc}{\SkipTocEntry}
\subsection{Explicit quasi-isometries}The geometry of a hyperbolic surface with an ideal triangulation is captured by the way the triangles are glued together. The gluing along an ideal arc occurs with a shear called the \emph{shear parameter}. To a hexagon decomposition $(\Gamma,\As)$, we associate the noded surface $Q(\Gamma,\As)$ obtained by pinching every curve of $\Gamma$ and setting to zero the shearing coordinates of the induced spiralling triangulation $\As_\infty$ on the complement. When $\Gamma=\emptyset$ this is the hyperbolic surface whose shearing coordinates along the ideal triangulation $\As$ all vanish, so $Q$ restricts to a map on the flip graph.
 
\begin{theorem*}[\ref{MainTheorem3}]
    Let $g,n$ be non-negative integers with $2g-2+n>0$. Then
    \[
        Q\colon(\Hs_{g,n}, \dhex) \longrightarrow\bigl(\overline{\teich}_{g,n},\dwp\bigr)
    \]
    is a quasi-isometry, where $\dhex$ denotes the combinatorial metric of the hexagon graph.
\end{theorem*}
Let $\dflip$ denote the combinatorial metric of the flip graph of triangulations and let $0<\epsilon \leq \log\sqrt{3}$. Then the quasi-isometry (Proposition~\ref{MainTheorem4})
    \[
        Q\colon (\Fs_{g,n}, \dflip) \longrightarrow\bigl(\teich^{\epsilon}_{g,n},\dteichpath\bigr)
    \]
does not follow from Theorem~\ref{MainTheorem3} even though $\Fs_{g,n}$ is the subgraph of $\Hs_{g,n}$ on the vertices with $\Gamma=\emptyset$. Indeed, that subgraph inclusion is not a quasi-isometric embedding, see Remark~\ref{rem1}.

\addtocontents{toc}{\SkipTocEntry}
\subsection{Bounding the width}
Coarse surjectivity of $Q$ is the statement that the \emph{width} of the combinatorial model, that is the Hausdorff distance between the image by $Q$ of the graph and the target space, is uniformly bounded for a fixed topological type $(g,n)$. For the hexagon graph, we obtain in Proposition~\ref{WidthHexagonModel} the following bound on the width,
\[
    \sup_{X\in\overline{\teich}_{g,n}}\ \inf_{(\Gamma,\As)}\ \dwp\left(X,Q(\Gamma,\As)\right)
    \;\leq\; 47\sqrt{(2g-2+n)\log\bigl(8\pi(2g-2+n)\bigr)}.
\]
The proof combines the grafting estimates of Section~\ref{sec:wpgraft} with an analogue, for ideal triangulations, of Wolpert's bound on the Weil--Petersson distance to a stratum~\cite{Wolpert2003}, recalled in Theorem~\ref{thm:WolpertShear}. For a fixed number of punctures, Cavendish and Parlier showed that the Weil--Petersson diameter of the compactified moduli space is at most of order $\sqrt{g}\log g$ as $g\to\infty$~\cite[Theorem~1.2]{CavendishParlier2012}, and Wu recently showed that this order is sharp~\cite{wu2026}. For the flip graph, letting $0<\epsilon \leq \log\sqrt3$, we obtain the following bound in Lemma~\ref{TeichshearingIDEAL},
\[
    \sup_{X\in\teich^{\epsilon}_{g,n}}\ \inf_{\mathcal{T}\in\Fs_{g,n}}\ \dteich\left(X,Q(\mathcal{T})\right)
    \;\lesssim\;\frac{g+n}{\epsilon}.
\]
We write $A \lesssim B$ to mean that there exist universal constants $c_1, c_2>0$ independent of $g,n,\epsilon$, such that $A \leq c_1\cdot B+c_2$.
Rafi and Tao obtained matching bounds of order $\log\bigl((g+n)/\epsilon\bigr)$ for the Teichmüller diameter of the $\epsilon$-thick part of moduli space~\cite{RafiTao2013}.  Our bound on the width is linear rather than logarithmic in $g+n$: the loss occurs in the shear bound of Proposition~\ref{IdealBershear} and is avoided when triangulations spiralling around curves are allowed into the combinatorics, that is, once the flip graph is enlarged to the hexagon graph, as seen in the result presented in the following section.
 
\addtocontents{toc}{\SkipTocEntry}
\subsection{Projecting to the thick part of a stratum}

For $0<\epsilon<\epsilon_0$ we define the \emph{$\epsilon$-thick part of a stratum} as the product of the $\epsilon$-thick parts of its factors (which are smaller dimensional Teichmüller spaces).
We show that any point in the augmented Teichmüller space can be projected to a combinatorial local ``center point'', lying in the $\epsilon_T'$-thick part of a boundary stratum, by controlling the Weil--Petersson distance and the geometry of the arrival point. The resulting bound on the Teichmüller distance is of the same order as the upper bound for the Teichmüller diameter of the thick part of moduli space established by Rafi--Tao~\cite{RafiTao2013}, while the bound on the Weil--Petersson distance, of order $\sqrt{(g+n)\log(g+n)}$, is below the upper bound of order $\sqrt{g}\log g$ for the Weil--Petersson diameter established by Cavendish--Parlier~\cite{CavendishParlier2012} when $n$ is fixed.
 
\begin{theorem*}[\ref{bigCOR}]
    Let $X\in\overline{\teich}_{g,n}$. Then there are a hexagon decomposition $(\Gamma,\As)$ and a noded surface $X_\infty$ such that
    \begin{enumerate}
        \item $X_\infty$ lies in the $\epsilon_T'$-thick part of the stratum $S(\Gamma)$, where $\epsilon_T'\approx0.016$ is a universal constant, independent of the topology;
        \item the shearing coordinates of $X_\infty$ along $\As_\infty$ satisfy
        \[
            \shear_{\As_\infty}(X_\infty)\;\leq\;44\log\bigl(8\pi(2g-2+n)\bigr)+249;
        \]
        \item $\dwp(X,X_\infty)\leq5\sqrt{2\pi(2g-2+n)\log\bigl(8\pi(2g-2+n)\bigr)}$;
        \item $\dteich\bigl(X_\infty,Q(\Gamma,\As)\bigr)\lesssim\log(g+n)$.
    \end{enumerate}
\end{theorem*}

Along the way, we obtain other explicit results. For instance, in Proposition~\ref{MainTheorem},  we show that a point of $\overline{\teich}_{g,n}$ and the ``associated local center'' $Q(\Gamma,\As)$ from the previous result can each be approximated, in the
Weil--Petersson metric, by noded surfaces of $S(\Gamma)$ lying on a single Thurston stretch path, of explicit length. This gives a way of combining the Weil--Petersson projection with Thurston's stretch geometry, and may be of independent interest.
 
\addtocontents{toc}{\SkipTocEntry}
\subsection{Outline of the proofs}
The proof Theorem~\ref{MainTheorem3} loosely follows Brock's strategy but differs in the geometric tools needed. More precisely, we develop and use distance estimates along two types of hyperbolic metric deformations, \emph{shearing} and \emph{grafting}, built upon works of~\cite{Wolpert1975,Wolpert2003,Tanigawa1997,McMullen1998,DumasWolf2008,KahnMarkovic2008,hensel2008,ChoiDumasRafi2012,DiazKim2012}.

\noindent\emph{Shearing and the Teichmüller metric.} Give an hyperbolic surface with an ideal triangulation, truncating the cusps turns the triangulation into a decomposition into right-angled hexagons. If two hyperbolic surfaces carry the same ideal triangulations with bounded shears, the corresponding hexagons have uniformly comparable sides. Mapping them to one another, hexagon by hexagon, and correcting along the shared sides so that the local maps agree, yields a bi-Lipschitz homeomorphism with controlled constant and hence a bound on the Teichmüller distance (Proposition~\ref{PropTeichShear}). Combined with the fact that every surface carries an ideal triangulation with controlled shears~\cite[Section 3.2]{Bershear}, this gives the width of the flip model (Lemma~\ref{TeichshearingIDEAL}) and the cost of a single flip (Lemma~\ref{Flip-Lipschitz}).
 
\noindent\emph{Grafting, Shearing and the Weil--Petersson metric.} Grafting inserts Euclidean cylinders along a simple multicurve. As the grafting parameter tends to infinity the surface degenerates to the noded surface obtained by pinching that multicurve. Bounding the Weil--Petersson length of a grafting ray gives, in Proposition~\ref{WPGraftingRay},
an explicit way of travelling from a surface to a stratum. Refining it so as to control the shears of the arrival point as well (Proposition~\ref{TravellingToStratum}) is the first step towards the width of the hexagon model (Proposition~\ref{WidthHexagonModel}) and, together with the shearing estimates, gives Theorem~\ref{bigCOR}. The second step for the width is an estimate of the WP distance, inside the stratum $S(\Gamma)$, between the arrival point of the grafting ray and the local center $Q(\Gamma,\As)$. Both points carry the ideal triangulation $\As_\infty$, and the truncated lengths of its arcs are of order $\log(g+n)$ on both of them. Hence we use the analogue, for ideal triangulations, of Wolpert's bound on the Weil--Petersson distance to a stratum to bound the WP distance between them by a quantity of order $\sqrt{(g+n)\log(g+n)}$~\cite[Theorem 3.10]{AbadieWolpertShear}.

\addtocontents{toc}{\SkipTocEntry}
\subsection{Organisation}
Section~\ref{sec:background} fixes notation and recalls the material we use.
Section~\ref{sec:hexcoord} shows that truncated lengths vary continuously as the
surface degenerates (Lemma~\ref{lem:truncated_length_continuous}) to build the hexagonal coordinates associated with a hexagon decomposition
(Proposition~\ref{HexagonalCoordinates}). 

Section~\ref{sec:teich} bounds the Teichmüller distance between two surfaces carrying the same ideal triangulation. We construct a bi-Lipschitz map between truncated hexagons with comparable sides, first between their boundaries (Lemma~\ref{LemmaBiLipBoundaryHex}), then between the hexagons themselves (Lemma~\ref{LemmaBiLipHex}) and then correct and glue these local maps (Proposition~\ref{PropTeichShear}).

In Section~\ref{sec:wpgraft} we bound the
Weil--Petersson length of a grafting ray
(Proposition~\ref{WPGraftingRay}) and then refine this so as to control the
shear parameters of the arrival point (Proposition~\ref{TravellingToStratum}).

Section~\ref{sec:project} refines the grafting estimates of Section~\ref{sec:wpgraft} into a projection to the thick part of a stratum. We first enlarge a short hexagon decomposition so that its multicurve contains every curve of length at most $\epsilon_T$ (Lemma~\ref{LemmaNiceHexDecomp}). The effect of grafting along it, in the change of metric, is absorbed inside the long collars of these short curves (Lemma~\ref{LemmaStayThick}), so the limiting noded surface is thick in the stratum (Proposition~\ref{PropProjectThickStartum}).

Section~\ref{sec:hexwidth} gives an upper bound for the width of the hexagon model, that is, the coarse surjectivity constant of $Q$ (Proposition~\ref{WidthHexagonModel}).
Section~\ref{sec:hexproof} shows that $Q$ is Lipschitz (Proposition~\ref{QLipschitz}), by bounding the Weil--Petersson cost of a flip move (Lemma~\ref{FlipMoveWPDist}) and of an adding curve move (Lemma~\ref{LIpschADdingcurve}), and then that it is a quasi-isometry (Theorem~\ref{MainTheorem3}), using the discreteness of the truncated length spectrum (Lemma~\ref{TruncOrthoSpectrum}).

Section~\ref{sec:flipqi} considers the flip graph. A surface whose shearing coordinates along an ideal triangulation vanish has systole bounded from below (Lemma~\ref{SysShearZero}), so the estimates of Section~\ref{sec:teich} give the width of the flip model and the cost of a single flip, hence Proposition~\ref{MainTheorem4}. Combining all the previous estimates then yields Theorem~\ref{bigCOR}.

Section~\ref{sec:thurston} shows that a point of augmented Teichmüller space and the image under $Q$ of a suitable hexagon decomposition each lie at controlled Weil--Petersson distance from a noded surface and that these two noded surfaces are joined by a single Thurston stretch path, of explicit length (Proposition~\ref{MainTheorem}).

\addtocontents{toc}{\SkipTocEntry}
\subsection*{Acknowledgments}
I am grateful to my supervisor Hugo Parlier for helpful discussions and for his encouragement.

\addtocontents{toc}{\SkipTocEntry}
\subsection*{AI use}
I used generative language models (Claude Opus~5 and ChatGPT 5.6) to review the final draft, to improve the writing, and to help check the clarity and consistency of notation, statements and proofs. All mathematical content, including the results and their proofs, is my own.

\section{Background}\label{sec:background}
This section contains the background material and fixes notation we use. Most of it is standard, and a reader familiar with Teichmüller theory may prefer to start at the sub-section~\ref{sec-Trunc-length}.

\addtocontents{toc}{\SkipTocEntry}
\subsection{The Weil--Petersson metric and its comparison with the Teichmüller metric}
Let $\Sigma$ be an oriented surface of genus $g$ with $n$ punctures and $\chi(\Sigma)=2-2g-n<0$. Its \emph{Teichmüller space}, denoted $\teich_{g,n}$ or $\teich(\Sigma)$, is the space of isotopy classes of complete finite-area hyperbolic metrics on $\Sigma$.
Let $X \in \teich_{g,n}$ equipped with the hyperbolic metric $\rho^2\,|dz|^2$. 
The cotangent space of $\teich_{g,n}$ at $X$ is identified with the space of holomorphic quadratic differentials on $X$, denoted by $\Omega(X)$. Let $\varphi \in \Omega(X)$ be written in local coordinates as $\varphi = \varphi(z)\, dz^2$, the space $\Omega(X)$ carries the $L^p$-norm, denoted by $ \left\lVert \varphi \right\rVert_p$, in particular 
\[
    \left\lVert \varphi \right\rVert_1 = \int_X |\varphi| 
    \qquad \text{and} \qquad 
    \left\lVert \varphi \right\rVert_2^2 = \int_X \frac{|\varphi|^2}{\rho^2}.
\]
For $\varphi,\psi \in \Omega(X)$ we have the $L^2$-Hermitian inner product on $\Omega(X)$,
\[
    \left(\psi, \varphi \right) \coloneq 
    \int_X \frac{\psi \, \overline{\varphi}}{\rho^2}.
\]
Let $B(X)$ be the space of Beltrami differentials on $X$. Let $\mu \in B(X)$ be written in local coordinates $\mu = \mu(z)\, \frac{d\bar z}{dz}$. If $\|\mu\|_{\infty}<1$, then $\frac{\partial f}{\partial \bar{z}}=\mu \frac{\partial f}{\partial z}$ for some quasiconformal map $f:X\longrightarrow \C$. Let $\phi \in \Omega(X)$, there is a bilinear pairing between $B(X)$ and $\Omega(X)$ given by
\[
    \langle \mu, \phi \rangle = \int_X \mu \phi.
\]
Denote by $N(X) \subset B(X)$ the subset 
\[
    N(X) = \{ \mu \in B(X) \mid \langle \mu, \phi \rangle = 0 \text{ for all } \phi \in \Omega(X) \}.
\] 
The tangent space of $\teich_{g,n}$ at $X$ is identified with the quotient $B(X)/N(X)$. Every class in $B(X)/N(X)$ contains a unique harmonic Beltrami differential which is of the form $\mu = \rho^{-2} \overline{\psi}$ for some $\psi \in \Omega(X)$. 
The Weil--Petersson (WP) distance is induced by the norm 
\[
    \left\lVert [\mu] \right\rVert_{\text{WP}} =\sup_{\phi \in \Omega(X)\smallsetminus \{0\}} \frac{|\langle \mu, \phi \rangle|}{ \left\|\phi \right\|_2 }.
\] 
The Teichmüller metric $\dteich$ is defined by
\[
    \dteich(X,Y) = \frac{1}{2} \inf_{f \sim \mathrm{id}} \log K(f),
\]
where the infimum is taken over all quasiconformal maps $f : X \to Y$ isotopic to the identity, and $K(f)$ denotes the maximal quasiconformal dilatation of $f$. It is also induced by the norm dual to the $L^1$ norm on $\Omega(X)$: 
\[
    \left\lVert [\mu] \right\rVert_{1} = \sup_{\varphi \in \Omega(X)\setminus\{0\}} \frac{|\langle \mu, \varphi \rangle|}{ \left\lVert \varphi \right\rVert_1 }.
\]

By a result of Linch~\cite{Linch1974}, the WP distance is bounded from above by the Teichmüller distance times the square root of the area, namely
\[
    \dwp(X,Y) \leq \sqrt{\area(X)} \, \dteich(X,Y).
\]
where, by the Gauss--Bonnet theorem, $\area(X) = 2\pi(2g-2+n)$.

Let $\epsilon \leq \epsilon_0$, the \emph{$\epsilon$-thick part} of Teichmüller space $\teich^{\epsilon}_{g,n}$ is the subset consisting of complete hyperbolic metrics that do not admit essential closed geodesics of length $<\epsilon$. The \emph{$\epsilon$-thin part} is its complement.

\addtocontents{toc}{\SkipTocEntry}
\subsection{Collar and cusp neighborhoods} We recall the definitions and basic properties of collars and cusp neighborhoods, which will be used in later arguments.
Let $\wstd(\ell)$ be defined for $\ell > 0$ as
\[
    \wstd(\ell) =  \arcsinh\left(\frac{1}{\sinh\left( \frac{\ell}{2}\right)}\right).
\]
By~\cite[Theorem 4.1.1]{Buser}, the \emph{standard collar} of a simple closed geodesic $\gamma$ on $X$ consists of all the points in $X$ at distance $\leq \wstd(\ell_X(\gamma))$ from $\gamma$. It is an embedded annulus,
isometric to the cylinder $[-\wstd,\wstd]\times \Sph^1$
with the Riemannian metric $d\rho^{2}+\ell_X(\gamma)^2\cosh^{2}(\rho)dt^{2}$ and we denote it by $\mathcal{C}_{\wstd(\ell_X(\gamma))}(\gamma)$. When there is no ambiguity, we write $\mathcal{C}_{\wstd}(\gamma)$ and $\wstd(\gamma)$. For $\delta \leq \wstd(\gamma)$, the subset of points at distance at most $\delta$ from $\gamma$ is the \emph{$\delta$-collar}. It is denoted by $\mathcal{C}_\delta(\gamma)$ and we say that the collar has \emph{width} $\delta$.

If $X$ has punctures, each puncture corresponds to a \emph{cusp}, isometric to
\[
(-\infty, \log 2] \times \mathbb{S}^1, \qquad d\rho^2 + e^{2\rho}\, dt^2,
\]
and bounded by a horocycle of length $2$~\cite{Keen1974}. For $\delta \le 2$, the region bounded by a horocycle of length $\delta$ is the \emph{$\delta$-cusp neighborhood} $N_\delta(c)$. The case $\delta=2$ gives the \emph{standard cusp neighborhood} denoted by $N(c)$.

\begin{theorem}[Collar Theorem in~\cite{Keen1974}, Theorem 4.4.6 in~\cite{Buser}]\label{CollarThm}
Let $X$ be a hyperbolic surface and let $\gamma_1$ and $\gamma_2$ be two distinct disjoint simple closed geodesics. Then their standard collars $\mathcal{C}_{\wstd(\gamma_1)}(\gamma_1)$ and $\mathcal{C}_{\wstd(\gamma_2)}(\gamma_2)$ are disjoint embedded cylinders. Moreover, if $X$ has punctures, these collars and the standard cusp neighborhoods are pairwise disjoint.
\end{theorem}

\addtocontents{toc}{\SkipTocEntry}
\subsection{Pants decompositions, ideal spiralling triangulations and hexagon decompositions}\label{whereDefHexIs} We recall the definitions from~\cite{Bershear}.

A \emph{pants decomposition} of a hyperbolic surface $\Sigma$ with genus $g$ and $n$ punctures is a maximal collection $\mathcal{P}$ of disjoint essential simple closed geodesics such that each connected component of $\Sigma \smallsetminus \mathcal{P}$ is topologically a three–holed sphere, also called a pair of pants. Any pants decomposition consists of $3g-3+n$ curves and decomposes $\Sigma$ into $2g-2+n$ pairs of pants. 

A \emph{hexagon decomposition} of $\Sigma$ is a collection $\Gamma$ of disjoint simple closed geodesics together with a maximal collection $\mathcal{A}$ of disjoint orthogeodesic arcs so that $\Sigma\smallsetminus (\Gamma \cup \mathcal{A})$ is a disjoint union of generalized right-angled hexagons: that is either a right-angled hexagon, or a right-angled pentagon with one ideal vertex, or a quadrilateral with two right angles and two adjacent ideal vertices, or an ideal triangle. In this case, the surface is decomposed into $4g-4+2n$ geodesic generalized right–angled hexagons, and the maximal collection $\mathcal{A}$ contains $6g-6+3n$ orthogeodesic arcs. We will sometimes abuse notation by referring to a generalized right-angled hexagon simply as a right-angled hexagon.

A \emph{spiralling triangulation} $\mathcal{T}$ of $\Sigma$ is a maximal collection of disjoint complete geodesics, whose components are either closed curves or bi-infinite arcs ending in cusps or spiralling around closed components. By maximality, the complement of $\mathcal{T}$ consists of $(4g-4+2n)$ ideal triangles, whose edges are the $6g-6+3n$ bi-infinite components of $\mathcal{T}$, and $\mathcal{T}$ may contain $0$ to $(3g-3+n)$ closed components, determined by the spiralling bi-infinite arcs. In the case where no spiralling occurs, i.e. all bi-infinite arcs end in cusps, $\mathcal{T}$ is called an \emph{ideal triangulation}.

\begin{figure}[h]
    \centering
    \includegraphics[scale=0.16]{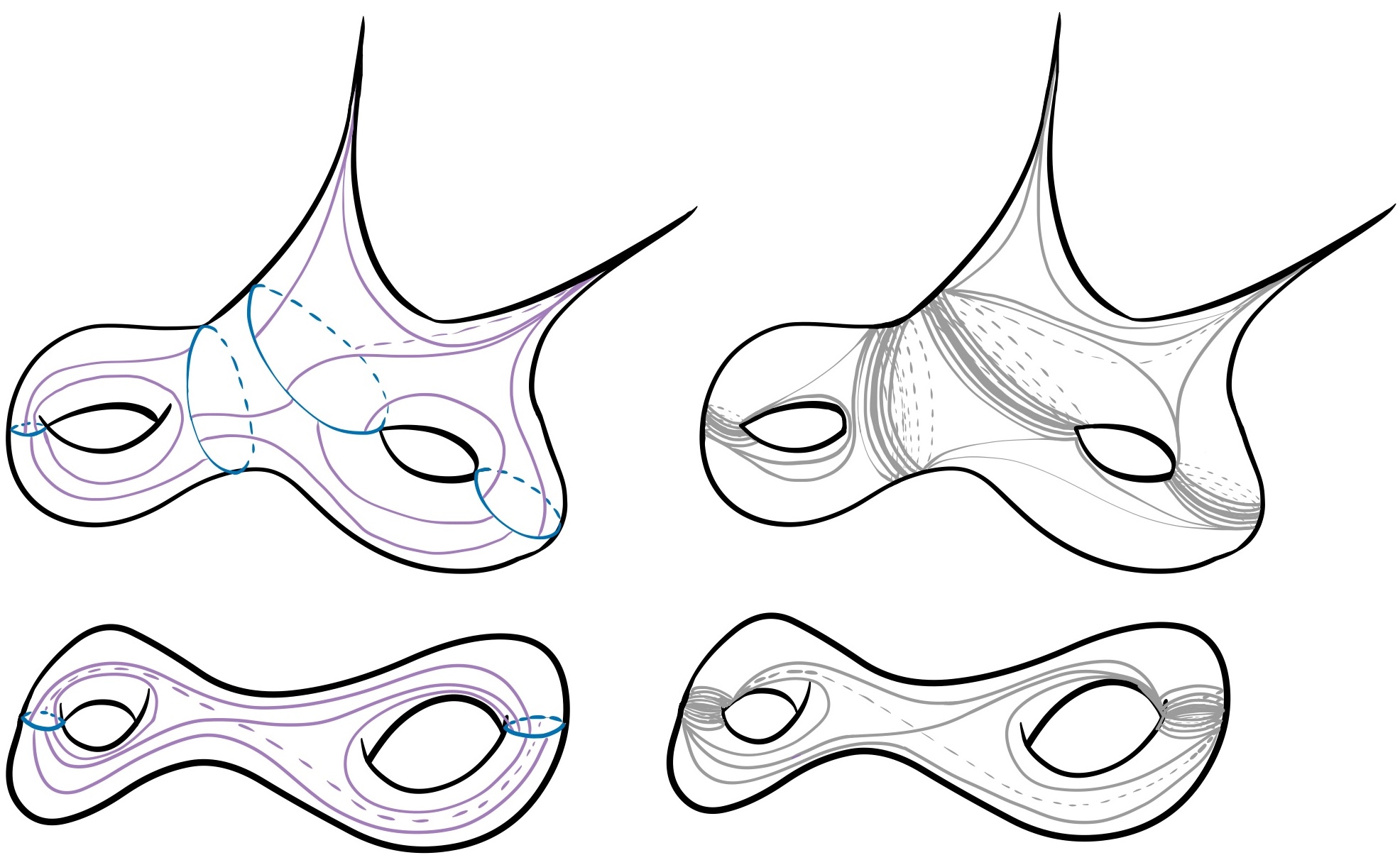}
    \caption{A hexagon decomposition and the associated spiralling triangulation, adapted from~\cite{Bershear}.}~\label{Spiralling_tirangulation}
    \vspace{-0.5cm}
\end{figure}

From a hexagon decomposition $(\Gamma, \mathcal{A})$ we can construct a spiralling triangulation $\mathcal{T}_{(\Gamma, \mathcal{A})}$ as follows. For any $a \in \mathcal{A}$ we define a bi-infinite geodesic $\lambda_a$ as follows.
\begin{itemize} 
    \item[-] If $a$ has both ends in a cusp, set $\lambda_a=a$.
    \item[-] If $a$ has endpoints on $\gamma_1$ and $\gamma_2$ in $\Gamma$, let $\lambda_a$ be the geodesic arc homotopic to the arc which follows $a$ and spirals infinitely many times around $\gamma_1$ and $\gamma_2$. Possibly $\gamma_1=\gamma_2$.
    \item[-] If $a$ has one endpoint in a cusp and the other on a simple closed curve $\gamma \in \Gamma$, then let $\lambda_a$ be the geodesic arc homotopic to the arc which follows $a$ from the cusp to $\gamma$ and spiralling infinitely many times around $\gamma$.
\end{itemize} 
There are different possible choices of spiralling. We pick one, once and for all: for each end of an arc $a \in \mathcal{A}$ lying on a curve $\gamma \in \Gamma$, the spiralling arc $\lambda_a$ turns left when following $a$ towards $\gamma$ (with respect to the orientation of $\Sigma$) and spirals accordingly around $\gamma$. In particular, all arcs arriving on the same side of $\gamma$ spiral in the same direction around $\gamma$.
Then, we define the \emph{spiralling triangulation} $\mathcal{T}_{(\Gamma, \mathcal{A})}$ associated with $(\Gamma, \mathcal{A})$ to be the collection of orthogeodesic arcs $(\lambda_a)_{a \in \mathcal{A}}$ together with the curves in $\Gamma$. See Figure~\ref{Spiralling_tirangulation}.

\addtocontents{toc}{\SkipTocEntry}
\subsection{Augmented Teichmüller space and extended Fenchel--Nielsen coordinates}\label{sec:AugmentedTeichmuller}
The Teichmüller space is homeomorphic to $\R^{6g-6+2n}$ via \emph{Fenchel--Nielsen coordinates}. Given a pants decomposition $\mathcal{P}$, these coordinates associate to a point in the Teichmüller space the lengths and twist parameters along each curve in $\mathcal{P}$,
\begin{equation*}
    \FN_{\Ps}\colon\teich_{g,n}\to(\R_{>0}\times\R)^{|\Ps|}\colon X\mapsto (\ell_X(\gamma),\tau_X(\gamma))_{\gamma\in\Ps}.
\end{equation*}
As mentioned earlier, the Weil--Petersson metric is non-complete. Wolpert showed that pinching lines---paths where all Fenchel--Nielsen coordinates remain fixed except for one or several length parameters that go to zero---define Cauchy sequences that do not converge inside Teichmüller space~\cite{Wolpert1975}. The Weil--Petersson completion, called the \emph{augmented Teichmüller space} and denoted by $\overline{\teich}_{g,n}$, is obtained by adding strata corresponding to noded surfaces which arise as limits of sequences of hyperbolic structures where the hyperbolic lengths of some collection $\Gamma$ of disjoint simple closed geodesics go to zero~\cite{Abikoff1977,Abikoff1980,Bers1974},
\[
    \overline{\teich}_{g,n}=\teich_{g,n} \cup \bigcup_{\Gamma\neq\emptyset} S(\Gamma).
\]
For each $\Gamma$, the stratum $S(\Gamma)$ is the Teichmüller space of the possibly disconnected surface $\Sigma \smallsetminus \Gamma$, which equals a product of lower-dimensional Teichmüller spaces
\[
    S(\Gamma)= \teich(\Sigma \smallsetminus \Gamma) \cong \prod_{i} \teich_{g_i,n_i},
\]
for some tuples $(g_i,n_i)$. In particular, $S(\emptyset)=\teich_{g,n}$.

The WP metric extends to $\overline{\teich}_{g,n}$~\cite{Masur1976}, and its restriction to a stratum $S(\Gamma)$ is the product of the Weil--Petersson metrics of the Teichmüller spaces of the components of $\Sigma\smallsetminus\Gamma$~\cite[Theorem~2(1), p.~17]{Wolpert2003}. The WP metric near a stratum $S(\Gamma)$ is close to the product of the Weil--Petersson metrics on its components~\cite{Masur1976, Wolpert2003, DaskalopoulosWentworth2003, Yamada2004}. The Fenchel--Nielsen coordinates extend continuously to the union of strata with pinched locus contained in $\Ps$ by allowing curves in $\Ps$ to have length zero~\cite{Abikoff1980}. When a curve has length zero, there is no twist around it, hence the \emph{extended} length-twist parameters take value in the space
\begin{equation*}
    \frac{\R_{\geq 0}\times\R}{\{0\}\times\R},
\end{equation*}
where the topology around $\Zero\coloneq[\{0\}\times\R]$ is defined as follows: $(\ell_n,\tau_n)\to\Zero$ if and only if $\ell_n\to 0$. Note that this is \emph{distinct} from the quotient topology on that space.\footnote{In the quotient topology, we would have $(\ell_n,\tau_n)\to\Zero$ if and only if $\ell_n\to 0$ and $(\tau_n)_{n\in\N}$ is bounded. In the augmented Teichmüller space, the twists can take arbitrary values when $\ell_n\to 0$.}

The \emph{extended Fenchel--Nielsen coordinates} with respect to $\Ps$ are given by the homeomorphism
\begin{equation*}
    \FN_{\Ps}\colon\bigcup_{\Gamma\subseteq\Ps} S(\Gamma)\to\left(\frac{\R_{\geq 0}\times\R}{\{0\}\times\R}\right)^{|\Ps|}\colon X\mapsto ([\ell_X(\gamma),\tau_X(\gamma)])_{\gamma\in\Ps},
\end{equation*}
for which we keep the notation $\FN_{\Ps}$.

\addtocontents{toc}{\SkipTocEntry}
\subsection{Shearing coordinates} 
Another coordinate system on $\teich_{g,n}$ is given by  the shearing coordinates~\cite{Bonahon1996,fock1998,Thurston1998}. 

All hyperbolic ideal triangles are isometric one to another and each has an inscribed circle tangent to its sides, the tangency points are called \emph{shear points}. Write $2\varrho=\log(3)/2$ for the radius of the inscribed circle. Thus, two ideal triangles $\Delta$ and $\Delta'$ sharing an edge $\lambda$ determine two shear points on $\lambda$. The signed distance between these shear points defines the \emph{shear} along $\lambda$, denoted by $S_\lambda(\Delta,\Delta')$, see Figure~\ref{shear2}. 
\begin{figure}[h]
    \centering
    \begin{overpic}[width=.55\linewidth,keepaspectratio]{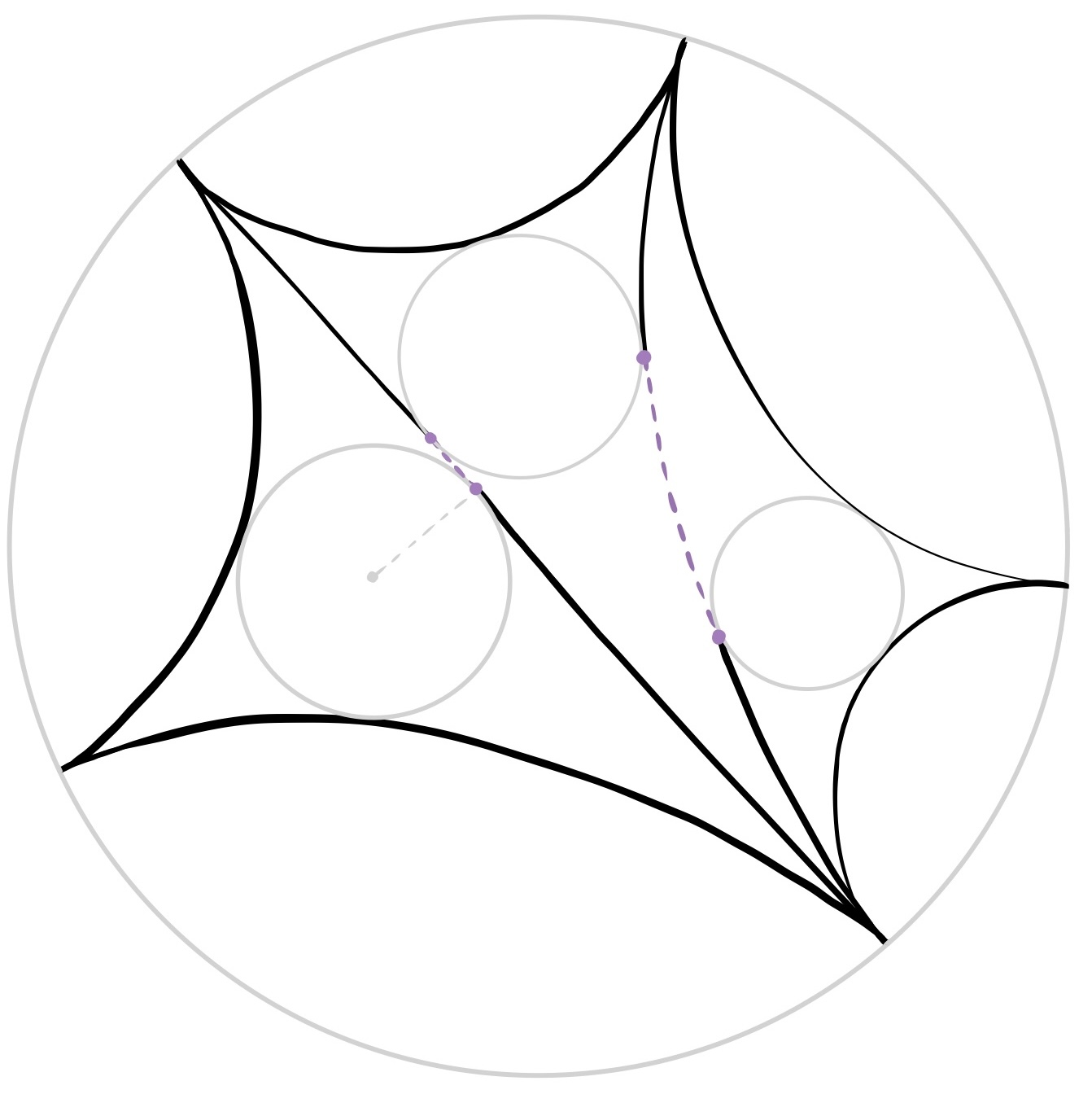}
    \put(38,49){\small\color{Gray}$2\varrho$}
    \put(27,80){\Large$\Delta$}
    \put(65,73){\Large$\Delta'$}
    \put(54,78){$\lambda$}
    \put(60,58){\small\color{Orchid}$S_{\lambda}(\Delta, \Delta')$\color{black}}
\end{overpic}\vspace*{-0.25cm}
    \caption{Shear along the common edge of two ideal triangles.}~\label{shear2}
\end{figure}
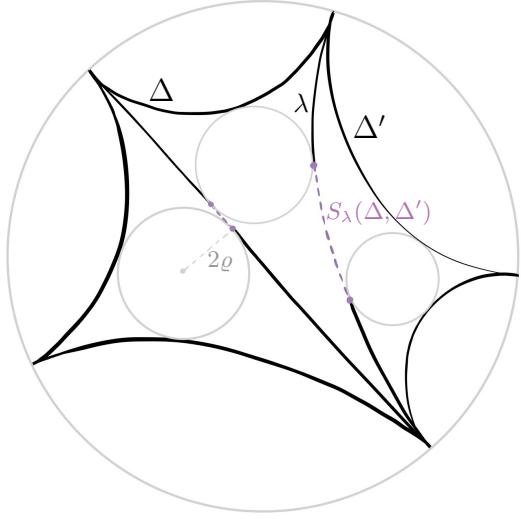

The shear parameters of a hyperbolic surface $X$, with respect to a spiralling triangulation $\mathcal{T}$, are defined by lifting $\mathcal{T}$ to $\widetilde{\mathcal{T}} \subset \Hh^2$. The lift decomposes $\Hh^2$ into ideal triangles. For each bi-infinite edge $\lambda_i$ of $\mathcal{T}$, choose a lift $\widetilde{\lambda}_i$ adjacent to two ideal triangles $\widetilde{\Delta}$, $\widetilde{\Delta}'$, and set 
\[
    S^{\mathcal{T}}_{\lambda_i} \coloneq S_{\widetilde{\lambda}_i}(\widetilde{\Delta},\widetilde{\Delta}').
\]
The shear parameters depend on a choice of orientation for each $\lambda_i$, fix one once and for all. We consider the maximum absolute value of the shear parameters of a surface $X$ along the bi-infinite components of $\mathcal{T}$: 
\[ 
    \shear_{\mathcal{T}}(X) \coloneq \max_{\lambda \in \mathcal{T}} |S^{\mathcal{T}}_{\lambda}|. 
\]  

\begin{remark}
    The shear parameters are not independent: the sum of shears along arcs ending at the same cusp is zero, and the sum of shears along arcs spiraling around a closed curve equals its length. However, if all triangles have vertices in cusps, there are $6g-6+3n-n=6g-6+2n$ independent shear parameters, giving a global parametrization of Teichmüller space. When $\mathcal{T}$ has closed components, the Fenchel-Nielsen twists around them must be added to obtain a full global parametrization, called the \emph{shearing coordinates}~\cite{Bonahon1996}. We will often abuse notation and let $\mathcal{T}$ refer to the set of bi-infinite components.
\end{remark}

\addtocontents{toc}{\SkipTocEntry}
\subsection{Truncation length function}\label{sec-Trunc-length} The definitions and notations of this subsection follow~\cite[Section~2.5]{Bershear}.

\begin{definition}[Definition 2.3~\cite{Bershear}]\label{def:truncation_function}
    A \emph{truncation function} is a continuous function $ \wabstract \colon \R_{> 0}\to\R_{\geq 0}$ satisfying the following properties:
    \begin{enumerate}
        \item For all $\ell\in\R_{>0}$, we have $0\leq \wabstract(\ell)\leq \wstd(\ell)$.
        \item The limit
        \begin{equation*}
            h_0(\wabstract)\coloneq \lim\limits_{\ell\to 0}\ell\cosh\left(\wabstract(\ell)\right)
        \end{equation*}
        exists and is strictly positive.
        \item There exists a uniform constant $\Delta_{\wabstract}> 0$ such that for all $\ell>0$,
        \begin{equation*}
            \wstd(\ell) - \wabstract(\ell) \leq \Delta_{\wabstract}.
        \end{equation*}
    \end{enumerate}
\end{definition}

Observe that the condition $w\leq w^{\std}$ implies that $h_0(w)\leq h_0(w^{\std})=2$.

\begin{definition}[Definition 2.4~\cite{Bershear}]
    Let $\Gamma\subset\Sigma$ be a simple multicurve and $\wabstract$ a truncation function.
    \begin{enumerate}
        \item For a point $X\in\overline{\teich}(\Sigma)$ belonging to a stratum $S(\Gamma^0)$ with $\Gamma^0\subseteq\Gamma$, its $(\Gamma,\wabstract)$\emph{-truncation} is the subsurface $X_{\Gamma}^{\wabstract}\coloneq \overline{X-N^{\wabstract}_{\Gamma}}$ where
        \begin{equation*}
            N^{\wabstract}_{\Gamma}\coloneq \bigcup_{i=1}^n N_{h_0(\wabstract)}(c_i) \cup \bigcup_{\gamma\in\Gamma-\Gamma^0} \Cs_{\wabstract(\ell_X(\gamma))}(\gamma)\cup\bigcup_{\gamma\in\Gamma^0} N_{h_0(\wabstract)}(\gamma)
        \end{equation*}
        is the union of the closed $h_0(w)$-horocyclic neighborhoods of the cusps and nodes in $X$ and the $w(\ell_X(\gamma))$-collar neighborhoods of the curves in $\Gamma \smallsetminus \Gamma^0$.
        \item The $(\Gamma,\wabstract)$\emph{-truncation} of an orthogeodesic arc $a$ on $X-\Gamma$ is the subarc
        \begin{equation*}
            a^{\wabstract}_{\Gamma}\coloneq a\cap X^{\wabstract}_{\Gamma}
        \end{equation*}
        and its $(\Gamma,\wabstract)$\emph{-truncated length} is defined as
        \begin{equation*}
            \ell^{\wabstract}_{X,\Gamma}(a)\coloneq\ell_X(a^{\wabstract}_{\Gamma}).
        \end{equation*}
    \end{enumerate}
\end{definition}

The truncated length $\ell_{X,\Gamma}^w(a)$ is always strictly positive since $w\leq w^{\std}$ and by the collar Theorem~\ref{CollarThm}.

Given a truncation function $\wabstract$, the injectivity radius at points of $\partial N_{h_0(\wabstract)}$ and at points of $\partial\Cs_{\wabstract(\ell)}$ equals, respectively,
\begin{align*}
    r_0^{\wabstract}&\coloneq\arcsinh(\tfrac{1}{2}h_0(\wabstract)),\\
    r_{\ell}^{\wabstract}&\coloneq\arcsinh(\sinh(\tfrac{\ell}{2})\cosh(\wabstract(\ell))).
\end{align*}
Note that $r_{\ell}^{\wabstract}\to r_{0}^{\wabstract}$ as $\ell\to 0$. More generally, for $\ell>0$ and $0<v\leq \wstd(\ell)$, the injectivity radius at the boundary of a width-$v$ collar about a curve of length $\ell$ equals
\begin{equation*}
    r_{\ell}(v)\coloneq\arcsinh(\sinh(\tfrac{\ell}{2})\cosh(v)).
\end{equation*}
Similarly, for $h\leq h_0(\wstd)$, the injectivity radius at a horocycle of length $h$ about a cusp or node equals
\begin{equation*}
    r(h)\coloneq\arcsinh(\tfrac{1}{2}h).
\end{equation*}

\begin{definition}
    Let $\epsilon_T = \frac{3\varrho}{2} = \frac{3\log(3)}{8}$. Define the \emph{shear truncation function} by
    \begin{equation*}
        \wtrunc \colon \R_{> 0} \to \R_{\geq 0}, \quad 
        \wtrunc(\ell) = \begin{cases}
            
            0 & \text{if } \ell \geq \epsilon_T, \\
            \arccosh\left(\frac{\epsilon_T}{\ell}\right) & \text{if } \ell < \epsilon_T.
        \end{cases}
    \end{equation*}
    By construction, the difference between the width of the standard collar $\wstd$ and the shear truncation satisfies 
    \[ \Delta_{\wtrunc} = \wstd(\epsilon_T) < 2.28, \]
    since the supremum is attained at $\ell=\epsilon_T$. Moreover $\wtrunc$ satisfies $h_0(\wtrunc) = \epsilon_T$.
\end{definition}

\addtocontents{toc}{\SkipTocEntry}
\subsection{Bers constant theorem and Wolpert's estimate}

Bers proved that for any $g,n$ satisfying $2g-2+n>0$ there exists a constant $B_{g,n}$, called the Bers constant~\cite{Bers1974}, such that for every $X \in \overline{\teich}_{g,n}$ there exists a pants decomposition of $X$ whose curves have lengths bounded above by $B_{g,n}$. 
Parlier has shown that the Bers constant is bounded above by $\area(X)=2\pi(2g-2+n)$~\cite{Parlier2023}. By the Bers constant theorem, the sets  
\[
    V_{B_{g,n}}(P)=\left\{ X \in \overline{\teich}_{g,n} \;|\; \forall \gamma \in P, \;\ell_X(\gamma) \leq 2B_{g,n} \right\},
\] 
as $P$ ranges over all pants decompositions, form a cover of the augmented Teichmüller space~\cite{Brock2003}. 

The Weil--Petersson distance between any two points is difficult to compute. However, Wolpert~\cite{Wolpert2003} showed that the Weil--Petersson distance from a point $X$ in Teichmüller space to the stratum obtained by pinching a multicurve $\Gamma$ is bounded by:
\[ 
    \dwp(X,S(\Gamma)) \leq \sqrt{2\pi \sum_{\gamma \in \Gamma} \ell_X(\gamma)}.
\]
In particular, if $\Gamma$ is a pants decomposition $P$, the corresponding stratum consists of only one point, which is the maximally noded surface $N_P$ obtained by pinching all the curves in $P$ to nodes. Since every $X \in \overline{\teich}_{g,n}$ lies in a region $V_{B_{g,n}}(P)$ for some $P$, Wolpert's estimate gives \[\dwp(X,N_P)\leq \sqrt{ 2\pi \cdot 2B_{g,n}\cdot(3g-3+n)}.\]
Thus every point $X$ in $\overline{\teich}_{g,n}$ lies at uniformly bounded Weil--Petersson distance from some maximally noded surface. Cavendish and Parlier~\cite{CavendishParlier2012} refined this upper bound to 
\[
    \dwp(X,N_P)\leq \sqrt{2\pi(2g-2+n)} \log(2\pi(2g-2+n)).
\]
For punctured surfaces, an analogue of Wolpert's estimate is proved in~\cite[Theorem 3.10]{AbadieWolpertShear}: the lengths of the curves of a pants decomposition are replaced by the truncated lengths of the arcs of an ideal triangulation, or controlled by its shearing coordinates. Let $Y\in\teich_{g,n}$ with $n\geq1$ and let $\mathcal{T}$ be an ideal triangulation of $Y$. The cusp neighborhoods $N_{1/4}(c)$, bounded by the horocycles of length $1/4$, are contained in the standard cusp neighborhoods, hence embedded and pairwise disjoint (Theorem~\ref{CollarThm}), and every arc $a\in\mathcal{T}$ meets them only along its two terminal rays. We denote by $\ell^{1/4}_Y(a)$ the length of the geodesic segment of $a$ outside these neighborhoods, and by
\[
    \Lambda_{\mathcal{T}}(Y)\coloneq\sum_{a\in\mathcal{T}}\ell^{1/4}_Y(a)
\]
the sum of these truncated lengths. The \emph{oscillation norm} of Kahn--Markovi\'c~\cite{KahnMarkovic2008} controls the sums of consecutive shears around the cusps. List the arc-ends at each cusp $c$ in cyclic order, $e_{c,1},\dots,e_{c,v(c)}$, an arc with both ends at $c$ being listed twice, and denote by $S_{e_{c,r}}^{\mathcal{T}}(Y)$ the shear parameter of the arc containing $e_{c,r}$. Then
\[
    \mathcal{O}_{\mathcal{T}}(Y)\coloneq\max_{c}\;\max_{1\leq j\leq k\leq v(c)}\;\Bigl|\sum_{r=j}^{k}S_{e_{c,r}}^{\mathcal{T}}(Y)\Bigr|.
\]

\begin{theorem}[{\cite[Theorems~3.9 and~3.10]{AbadieWolpertShear}}]\label{thm:WolpertShear}
    Let $g,n$ be integers with $n\geq1$ and $2g-2+n>0$, and let $Y,Z\in\teich_{g,n}$ be two hyperbolic surfaces equipped with the same topological ideal triangulation $\mathcal{T}$. Then
    \begin{enumerate}
        \item\label{thm:WolpertShear-trunc} $\dwp(Y,Z)\leq C_W\sqrt{n+\Lambda_{\mathcal{T}}(Y)+\Lambda_{\mathcal{T}}(Z)}$, where $C_W\coloneq\sqrt{8+2\pi B}<5.82$ and $B=2\sqrt2+\log(1+\sqrt2)/\log 2$;
        \item\label{thm:WolpertShear-osc} $\dwp(Y,Z)\leq 21\sqrt{(2g-2+n)\Bigl(\log\bigl(24(2g-2+n)\bigr)+\mathcal{O}_{\mathcal{T}}(Y)+\mathcal{O}_{\mathcal{T}}(Z)\Bigr)}$.
    \end{enumerate}
\end{theorem}

The constants depend neither on the topology nor on the geometry of $Y$ and $Z$. 

\begin{remark}\label{eq:WP-product-strata}
    Both bounds apply on strata, component by component.  Let $\Gamma$ be a simple multicurve on $\Sigma$, possibly empty, let $\Sigma_1,\dots,\Sigma_k$ be the components of $\Sigma\smallsetminus\Gamma$, of respective types $(g_i,n_i)$, and let $\mathcal{T}$ be an ideal triangulation of $\Sigma\smallsetminus\Gamma$ whose arcs end at the cusps of $\Sigma$ or at the punctures created by $\Gamma$. It induces ideal triangulations $\mathcal{T}_i$ of the $\Sigma_i$. For a noded surface $Y=(Y_i)_i\in S(\Gamma)$ we set
    \[
        \Lambda_{\mathcal{T}}(Y)\coloneq\sum_i\Lambda_{\mathcal{T}_i}(Y_i), \qquad \mathcal{O}_{\mathcal{T}}(Y)\coloneq\max_i\mathcal{O}_{\mathcal{T}_i}(Y_i).
    \]
    Recall that $S(\Gamma)=\prod_i\teich(\Sigma_i)$ and that the restriction of the Weil--Petersson metric to $S(\Gamma)$ is the product of the Weil--Petersson metrics on its components~\cite{Wolpert2003}. Hence any $Y=(Y_i)_i$ and $Z=(Z_i)_i$ in $S(\Gamma)$ satisfy
    \begin{equation}\label{eq:WP-product-strata2}
        \dwp(Y,Z)\leq \sqrt{\sum_i\dwp(Y_i,Z_i)^2},
    \end{equation}
    and it suffices to apply Theorem~\ref{thm:WolpertShear} on each component. 
\end{remark}

\addtocontents{toc}{\SkipTocEntry}
\subsection{Short hexagon decompositions and ideal triangulations with bounded shear parameters} 

\begin{definition}\label{ShortHexagonDecomposition}
    Let $\wabstract \colon \R_{> 0} \to \R_{\geq 0}$ be a truncation function and let $L,L_A>0$. A hexagon decomposition $(\Gamma,\mathcal{A})$ on $X \in \overline{\teich}_{g,n}$ is said to be $(L,L_A, \wabstract)$-\emph{short} if:
    \begin{enumerate}
        \item $\ell_X(\gamma) \le L$ for every $\gamma \in \Gamma$, and
        \item $\ell_{X, \Gamma}^{\wabstract}(a) \le L_A$ for every arc $a \in \mathcal{A}$.
    \end{enumerate}
    If $L=L_A$, to simplify we say that $(\Gamma,\mathcal{A})$ is $(L,\wabstract)$-\emph{short}.
\end{definition}

The choice of truncation function is inessential by~\cite[Proposition 2.10]{Bershear}: given two truncation functions $w_1,w_2$, if a hexagon decomposition is $(L,L_A,w_1)$-short, then it is $(L,L_A+2(\Delta_{w_1}+\Delta_{w_2}),w_2)$-short.

Parlier proved that every closed hyperbolic surface admits a ``short'' hexagon decomposition~\cite{Parlier2016}. The same argument extends to arbitrary truncation functions and to surfaces with punctures. More precisely, we have the following result:

\begin{proposition}[Proposition 2.8 in~\cite{Bershear}]
    Let $\wabstract \colon \R_{> 0} \to \R_{\geq 0}$ be a truncation function, and define $R_{g,n}$ by
    \[ R_{g,n} \coloneq\operatorname{arccosh}\Biggl( \frac{1}{2 \, \sin\bigl(\frac{\pi}{12g - 6 + 6n}\bigr)} \Biggr) \leq \log\bigl(8\pi(2g-2+n)\bigr). \]
    Then, for any $X \in \teich_{g,n}$ with $2g-2+n > 0$, there exists a $(2R_{g,n},\;2R_{g,n} + 2\Delta_{\wabstract},\wabstract)$-short hexagon decomposition $(\Gamma, \mathcal{A})$ on $X$.
\end{proposition}

More specifically, we fix the following bounds:
\[
    L(g,n) \coloneq2\log\bigl(8\pi(2g-2+n)\bigr)
\]
and
\[    
    L_A(g,n) \coloneq L(g,n) + 2\Delta_{\wtrunc}.
\]
When the topological type $(g,n)$ of the surface is fixed, we simply write $L$ and $L_A$.
A hexagon decomposition $(\Gamma,\mathcal{A})$ is called \emph{short} if it is $(L,L_A,\wtrunc)$-short. 
Since each stratum is a product of lower dimensional Teichmüller spaces, we have:

\begin{cor}\label{ShortHexagonDecompositionExistence}
    For every hyperbolic surface $X \in \overline{\teich}_{g,n}$, there exists a short hexagon decomposition of $X$.
\end{cor}

The following theorem provides a shearing-coordinate analogue of Bers' theorem. It replaces bounded lengths in a pants decomposition with uniformly bounded shear parameters along the non-closed components of a spiralling triangulation. The key idea in the proof is that shear points cannot lie in cusp neighborhoods or within collars around short curves (Proposition~3.4 in~\cite{Bershear}). Consequently, given a short hexagon decomposition, the shear along an edge of the induced spiralling triangulation is supported on a subarc of uniformly bounded length. This yields explicit bounds on the shear parameters in terms of bounds for the lengths of the hexagon sides.

\begin{theorem}[Proposition 3.6 in~\cite{Bershear}]\label{BershearRmk2}
    Let $X\in \teich_{g,n}$ with $2g-2+n>0$. Let $L', L_A'>0$ and let $(\Gamma, \As)$ be a $(L',L_A',\wtrunc)$-short hexagon decomposition of $X$. Then the shear parameter along any arc component $\lambda$ of $\mathcal{T}_{(\Gamma, \As)}$ is bounded by
    \[
        |S^{\mathcal{T}_{(\Gamma, \mathcal{A})}}_{\lambda}(X)|\leq 4L_A' + 4\max(L',\epsilon_T)+ 4\epsilon_T + 8\varrho.
    \]
    where $\varrho=\frac{\log(3)}{4}$ and $\epsilon_T = \frac{3\varrho}{2}$.
\end{theorem}

\begin{cor}[Proposition 3.5 in~\cite{Bershear}]\label{Bershear}
    For every surface $X \in \overline{\teich}_{g,n}$ and every short hexagon decomposition $(\Gamma,\mathcal{A})$ on $X$, the induced spiralling triangulation $\mathcal{T}_{(\Gamma,\mathcal{A})}$ satisfies
    \[
        \shear_{\mathcal{T}_{(\Gamma,\mathcal{A})}}(X) \leq 16\log\bigl(8\pi(2g-2+n)\bigr) + 30.
    \]
\end{cor}

\begin{remark}\label{RmkBersIdealTriangulation}
    If a restrict ourselves to consider only ideal triangulations (where all ideal arcs end in cusps) of a punctured hyperbolic surface $X$ of type $(g,n)$ with $n\geq 1$. Then Section 3.2 in~\cite{Bershear} shows that the shear points of an ideal triangulation on $X$ cannot lie within cusp neighborhoods whose area is smaller than $\delta_1'=0.27$. This value comes from a strict lower approximation for the area of the region delimited by an ideal triangle's inscribed circle and its contact triangle. Hence, there exists an ideal triangulation whose shear parameters admit an upper bound depending on the systole $\sys(X)$, that is the length of the shortest essential closed geodesic on $X$. This is the content of the following result.
\end{remark}

\begin{proposition}[Section 3.2 of~\cite{Bershear}]\label{IdealBershear}
    For every $\epsilon \leq \epsilon_0$ and $X \in \teich_{g,n}^{\epsilon}$ in the $\epsilon$-thick part of Teichmüller space, there exists an ideal triangulation $\mathcal{T}$ such that
    \[
        \shear_{\mathcal{T}}(X) \lesssim \frac{g+n}{\epsilon}.
    \]
\end{proposition}

\addtocontents{toc}{\SkipTocEntry}
\subsection{The pants, flip and hexagons graphs} Let $\Sigma$ be a hyperbolic surface of genus $g$ with $n$ punctures. Several combinatorial graphs have been introduced to model the geometry of Teichmüller space.

\addtocontents{toc}{\SkipTocEntry}
\subsubsection{The pants graph}
The vertices of the \emph{pants graph} $\mathscr{P}_{g,n}$ are given by isotopy classes of pants decompositions of $\Sigma$. Two vertices are connected by an edge if the corresponding pants decompositions differ by an elementary move, namely replacing a single curve by another curve
intersecting it minimally. 

The pants graph was introduced by Hatcher and Thurston~\cite{HatcherThurston1980}. It is a connected and infinite graph and it is not Gromov hyperbolic in general~\cite{BrockFarb2006}. The mapping class group acts on $\mathscr{P}_{g,n}$, and the quotient is the modular pants graph, which is finite. The mapping class group is the automorphism group of the pants graph~\cite{Margalit2004, ParlierPetri2018}. Brock proved that $\mathscr{P}_{g,n}$, equipped with its combinatorial metric, is quasi-isometric to the Teichmüller space endowed with the Weil--Petersson metric \cite[Theorem~1.1]{Brock2003}. Vertices in the pants graph correspond to maximally noded surfaces in the boudary strata.

\addtocontents{toc}{\SkipTocEntry}
\subsubsection{The flip graph of triangulations}
In the case where $\Sigma$ has at least one cusp, another model to consider is the \emph{flip graph}, whose vertices are isotopy classes of ideal triangulations. Two triangulations are connected by an edge if they differ by a flip, that is, a diagonal exchange in one ideal quadrilateral.

The flip graph is infinite and connected~\cite{Hatcher1991}. The diameter of the flip graph of polygons with $m$ sides is finite and equals $2m-10$ for $m>12$~\cite{SleatorTarjanThurston1988, Pournin2014}. The quotient by the mapping class group action is the \emph{modular flip graph} whose diameter has order $g\log(g)+n\log(n)$~\cite{DisarloParlier2019}. In the \emph{simultaneous flip graph} several disjoint flips may be performed at once. The diameter of the \emph{modular simultaneous flip graph}, where we consider triangulations up to homeomorphism, is of order at most $\log^2(g+n)$ and at least $\log(g+n)$~\cite{DisarloParlier2018}. In general, the automorphism group of the flip graph is the mapping class group~\cite{KorkmazPapadopoulos2012}. The mapping class group acts properly discontinuously and cocompactly on the flip graph, thus the flip graph is quasi-isometric to the mapping class group and therefore to the thick part of Teichmüller space~\cite{DisarloParlier2019, LackenbyPurcell2024, Penner2012}. 

\addtocontents{toc}{\SkipTocEntry}
\subsubsection{The hexagon graph}\label{definition_hexagon_graph}
Introduced by Gültepe and Parlier~\cite{GultepeParlier2025}, the \emph{hexagon graph} $\mathscr{H}_{g,n}$ has its vertices given by hexagon decompositions of $\Sigma$. Two vertices $H = (\Gamma, \mathcal{A})$ and $H' = (\Gamma', \mathcal{A}')$ are connected by an edge if they are related by one of two types of \emph{elementary moves} (see Figure~\ref{Hexagon_Graph}):
\begin{itemize}
    \item[-] \textit{Flip:} The multicurves are identical ($\Gamma = \Gamma'$), and the multiarcs $\mathcal{A}$ and $\mathcal{A}'$ differ by a single arc flip.
    \item[-] \textit{Curve addition/removal:} $H'$ is obtained from $H$ by adding a \emph{compatible} curve $\alpha$. A curve is compatible if it is disjoint from $\Gamma$ and intersects each arc in $\mathcal{A}$ at most once. The new set of orthogeodesic arcs $\mathcal{A}'$ is then obtained by is obtained by cutting each arc of $\mathcal{A}$ at its intersection with $\alpha$ and then taking the resulting orthogeodesic arcs up to isotopy in $\Sigma \smallsetminus \Gamma'$, duplicate are identified.
\end{itemize}
\vspace{-0.2cm}
\begin{figure}[h]
    \centering
    \includegraphics[scale=0.14]{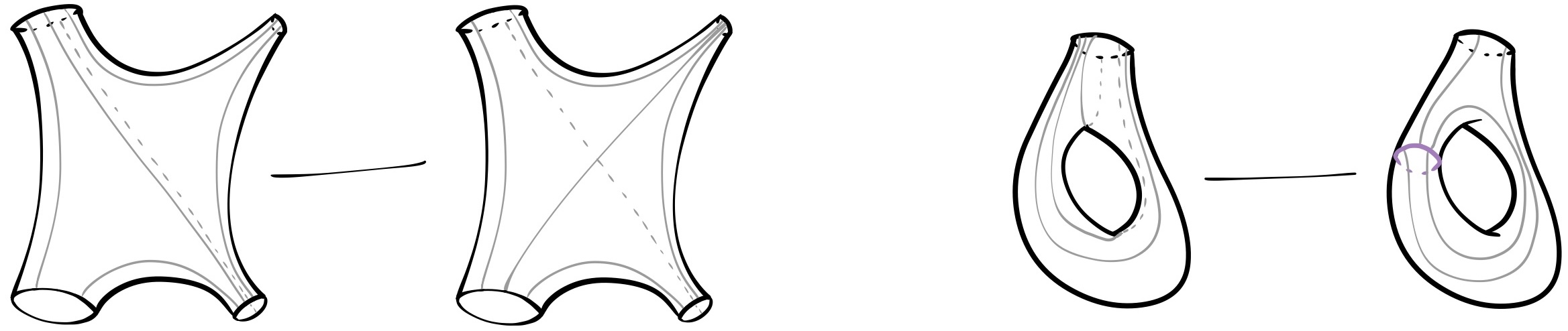}
    \caption{Elementary moves in the hexagon graph}~\label{Hexagon_Graph}
    \vspace{-0.5cm}
\end{figure} 
The hexagon graph $\mathscr{H}_{g,n}$ is quasi-isometric to the standard pants graph $\mathscr{P}_{g,n}$~\cite[Theorem 1.1]{GultepeParlier2025} and so by Brock it is quasi-isometric to the Teichmüller space with the Weil--Petersson metric. 
In the next parts of this work, we study the relation between the hexagon graph and Teichmüller space directly, rather than through the intermediate quasi-isometry with the pants graph.

\section{Hexagonal coordinates and truncated lengths}\label{sec:hexcoord}
The truncated length of an orthogeodesic arc was defined in the sub-section~\ref{sec-Trunc-length}. Here we prove that it varies continuously as the surface degenerates (Lemma~\ref{lem:truncated_length_continuous}), and use it to build a system of coordinates on the union of strata $\bigcup_{\Gamma^0\subseteq\Gamma}S(\Gamma^0)$ adapted to a hexagon decomposition (Proposition~\ref{HexagonalCoordinates}).

\begin{lemma}\label{lem:truncated_length_continuous}
    Let $\Gamma\subset\Sigma$ be a simple multicurve and let $\wabstract$ be a truncation function. Let $a$ be an arc whose interior is contained in $\Sigma\setminus\Gamma$ and whose endpoints lie in cusps or on $\Gamma$. Then the $(\Gamma,\wabstract)$-truncated length function
    \[
        \begin{aligned}
            \ell^{\wabstract}_{\bullet,\Gamma}(a)\colon\bigcup_{\Gamma^0 \subseteq \Gamma} S\left(\Gamma^0\right) &\longrightarrow \R_{>0} \\
            X &\longmapsto \ell^{\wabstract}_{X,\Gamma}(a),
        \end{aligned}
    \]
    where $\ell^{\wabstract}_{X,\Gamma}$ is the $(\Gamma,\wabstract)$-truncated length of the orthogeodesic representative of $a$ on $X$, is continuous.
\end{lemma}
\begin{proof}
    To prove this, we use the fact that the topology on $\overline{\teich}(\Sigma)$ can be described via bilipschitz homeomorphisms defined on complements of neighborhoods of the nodes, see~\cite[Section 1]{Abikoff1977}. Let us set up the definitions and notation introduced by Abikoff, who proves that this topology coincides with the one induced from the Fenchel--Nielsen coordinates~\cite[Theorem 1]{Abikoff1977}.
    
    Given an essential simple multicurve $\Gamma\subset\Sigma$, let $\Sigma_{\Gamma}$ denote the quotient space of $\Sigma$ obtained by collapsing each component of $\Gamma$ into a point, called a \emph{node}. Let $\Sigma'_{\Gamma}\subseteq\Sigma_{\Gamma}$ be the complement of the nodes.

    A \emph{deformation of nodal hyperbolic surfaces} is a triple $(X_1,X_2,f)$ where
    \begin{itemize}
        \item $X_i$ is a complete hyperbolic structure on $\Sigma'_{\Gamma_i}$ where $\Gamma_i\subset\Sigma$ is an essential simple multicurve, for $i=1,2$;
        \item $f\colon \Sigma_{\Gamma_1}\to\Sigma_{\Gamma_2}$ is a continuous surjection such that
        \begin{enumerate}[label=(\roman*)]
            \item $f$ sends nodes to nodes;
            \item $f^{-1}\restr{\Sigma'_{\Gamma_2}}$ is a homeomorphism onto its image;
            \item the preimage under $f$ of each node in $\Sigma_{\Gamma_2}$ is a point or an essential simple closed curve.
        \end{enumerate}
    \end{itemize}
    Two deformations $(X_1,X_2,f)$ and $(X_1,X_3,g)$ of $X_1$ are equivalent if there exists a homeomorphism $h\colon\Sigma_{\Gamma_2}\to\Sigma_{\Gamma_3}$ whose restriction to $\Sigma'_{\Gamma_2}\to\Sigma'_{\Gamma_3}$ is an isometry from $X_2$ to $X_3$, and such that $h\circ f$ is homotopic to $g$.

    The augmented Teichmüller space $\overline{\teich}(\Sigma)$ is the space of \emph{equivalence classes of deformations} of $X_0$ where $X_0$ is any choice of complete hyperbolic structure on $\Sigma$.
    
    Let us fix a deformation $(X_0,Y_0,f_0)$ (denoted by $Y_0$ for conciseness) and let $\Gamma_0$ denote the simple multicurve associated to $Y_0$, i.e.\ $Y_0$ is a hyperbolic structure on $\Sigma'_{\Gamma_0}$. Let $K>1$ and $N$ be a neighborhood of the nodes and cusps, i.e.\ a neighborhood of $C_0\coloneq \Sigma_{\Gamma_0}-\Sigma'_{\Gamma_0}$ and of the original cusps in $\Sigma$. The $(K,N)$-neighborhood of $Y_0$ is the set of deformations $(X_0,X,f)$ of $X_0$ that admit a deformation $(X,Y_0,g)$ satisfying the following conditions
    \begin{itemize}
        \item $g^{-1}\restr{\Sigma_{\Gamma_0}-N}$ is $K$-bilipschitz;
        \item $g\circ f$ is homotopic to $f_0$.
    \end{itemize}
    One obtains a basis of neighborhoods of $Y_0$ by taking $K\searrow 1$ and smaller and smaller neighborhoods $N$ of the cusps and nodes.

    We are now ready to prove continuity of the truncated length function $\ell^w_{\bullet,\Gamma}(a)$. Let $\{\gamma_i\}_{1\leq i\leq j}$ be the collection of curves in $\Gamma$ meeting $\partial a$, where $j\in\{0,1,2\}$ and the collection is empty if $j=0$. Given $\Gamma^0\subseteq\Gamma$ and $X\in S(\Gamma^0)$, corresponding to a deformation $(X_0,X,f)$, the geodesic representative $a^*$ of $a$ in $X$ can be described as follows. Up to an ambient isotopy, we can assume that $a\cap f^{-1}(N_{\Gamma}^w)$ is the union of two subarcs (possibly points if the collar width is zero). Then $a^*\cap X_{\Gamma}^w$ is the shortest geodesic representative of $f(a)\cap X_{\Gamma}^w$, in the homotopy class relative to $\partial N_{\Gamma}^w$. It is then automatically orthogonal to $\partial N_{\Gamma}^w$.

    Let $(X_m)_{m\in\N}$ be a sequence in $\overline{\teich}(\Sigma)$ converging to $X_{\infty}$, with each $X_m\in S(\Gamma_m)$ for some $\Gamma_m\subseteq\Gamma$ and $X_{\infty}\in S(\Gamma_{\infty})$. Thus $\Gamma_m\subseteq\Gamma_{\infty}$ for all $n$ large enough. Since $\Gamma_m$ can take only finitely many values, it suffices to consider a subsequence with $\Gamma_m$ constant. Thus we can assume that $\Gamma_m=\Gamma_1\subseteq\Gamma_{\infty}$ for all $n\in\N$.

    If both endpoints of $a$ lie on $\Gamma-\Gamma_{\infty}$, then continuity of the truncated length directly follows: the non-truncated length of $a$ varies continuously (e.g.\ by a doubling) and the truncation width varies continuously (by definition). If at least one endpoint of $a$ lies on $\Gamma_{\infty}-\Gamma_1$ and the other lies on some component $\gamma\subseteq\Gamma-\Gamma_{\infty}$, then upon doubling at $\gamma$ we can assume that both endpoints of $a$ lie on $\Gamma_{\infty}-\Gamma_1$. From now on, we make this assumption and we denote by $\gamma_1,\gamma_2$ the components of $\Gamma_{\infty}-\Gamma_1$ containing the endpoints of $a$. They may coincide.

    We are going to show that
    \begin{equation*}
        \limsup\limits_{m\to+\infty}\ell_{X_m,\Gamma}^w(a)\leq\ell_{X_{\infty,\Gamma}}^w(a)\leq\liminf\limits_{m\to+\infty}\ell_{X_m,\Gamma}^w(a).
    \end{equation*}

    Consider a neighborhood $N$ of the cusps and nodes in $X_{\infty}$ that is contained deep enough in the $h_0(w)$-horocyclic neighborhoods, e.g.\
    \begin{equation*}
        N\coloneq \bigcup_{i=1}^n N_{h}(c_i)\cup\bigcup_{\gamma\subset\Gamma_{\infty}}N_{h}(\gamma)
    \end{equation*}
    for some $h>0$ much smaller than $h_0(w)$.

    When $m\to +\infty$, we have $X_m\to X_{\infty}$, hence
    \begin{equation*}
        \left|\frac{r^w_{\ell_{X_m}(\gamma_i)}}{r^w_{0}}\right|\to 1.
    \end{equation*}

    Let $1<K<2$. Let $m_K\in\N$ be such that, for all $m\geq m_K$, the following two conditions are satisfied:
    \begin{itemize}
        \item $X_m$ is contained in the $(K,N)$-neighborhood around $X_{\infty}$, let $(X_m,X_{\infty},g_m)$ be the corresponding deformation, i.e.\ $g_m$ is in the correct homotopy class and is such that
        \[
            g_m^{-1}\restr{X_{\infty}-N}\colon X_{\infty}-N\to X_m
        \]
        is a $K$-bilipschitz homeomorphism onto its image $Y_m\subset X_m$.
        \item $K^{-1}<r^w_{\ell_{X_m}(\gamma_i)}/r^w_{0}<K$ for $i=1,2$.
    \end{itemize}

    Fix some $m\geq m_K$. First, let $h_K<h_0(w)$ be such that
    \begin{equation*}
        r(h_K)=\arcsinh(\tfrac{1}{2}h_K)=K^{-2} r_0^w.
    \end{equation*}
    Then the horocycles of respective length $h_0(w)$ and $h_K$ lie at distance $\delta_K\coloneq\ln(h_0(w)/h_K)$ one from each other.

    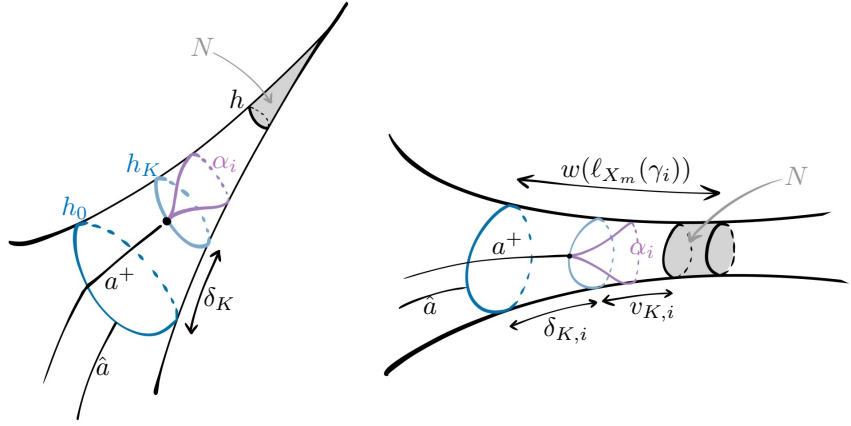
\begin{figure}[h]
        \centering
        \begin{overpic}[width=.9\linewidth,keepaspectratio]{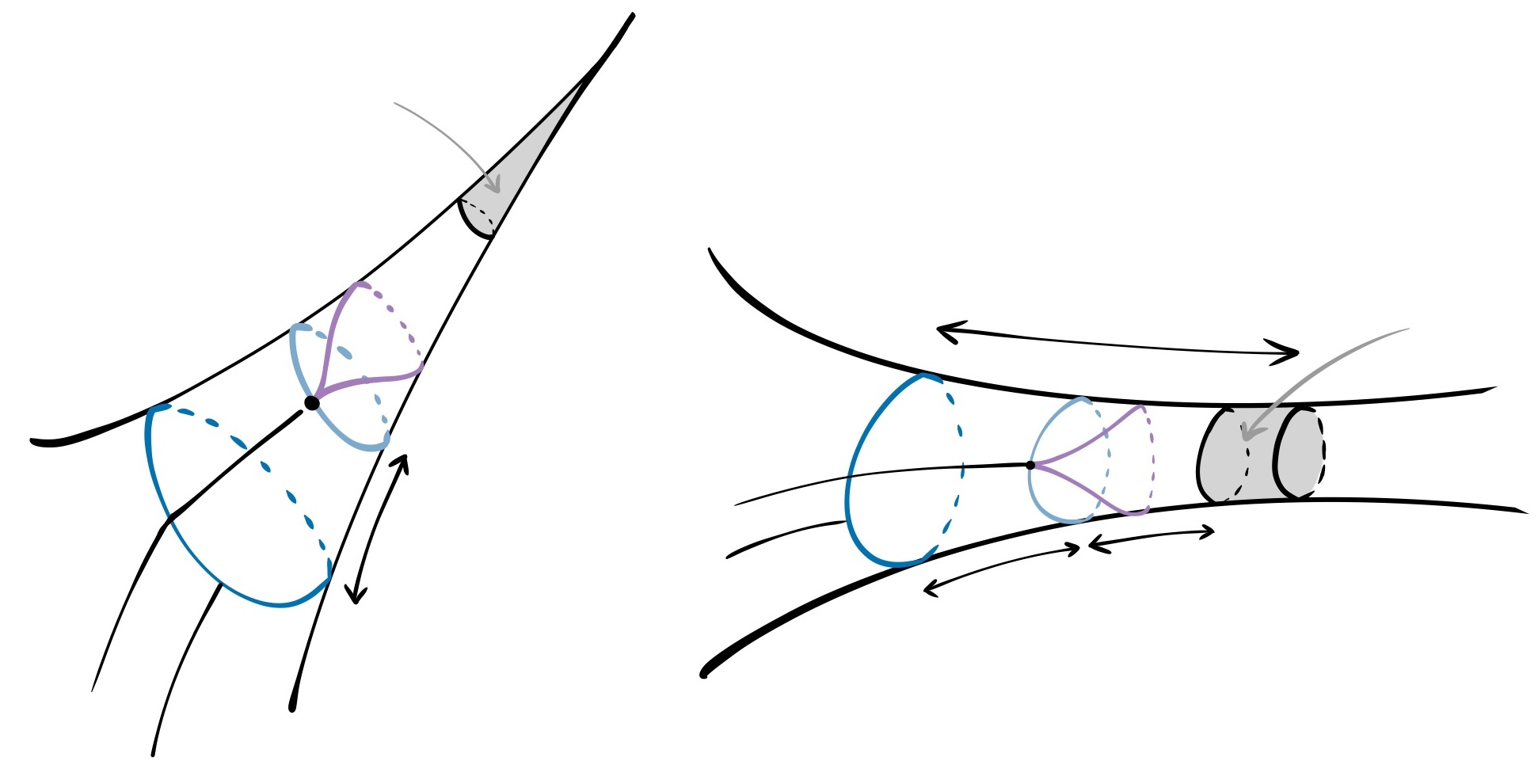}
        \put(23,44){\color{Gray}$N$}
        \put(27.5,38){$h$}
        \put(15.5,30.5){\color{RoyalBlue}$h_K$}
        \put(8,25.5){\color{RoyalBlue}$h_0$}
        \put(13,17){$a^+$}
        \put(24.5,15){$\delta_K$}
        \put(12,7){$\hat{a}$}
        \put(25.5,31){\color{Orchid}$\alpha_i$\color{black}}
        
        \put(92,29){\color{Gray}$N$}
        \put(64,11.5){$\delta_{K,i}$}
        \put(74,14){$v_{K,i}$}
        \put(66,30){$w(\ell_{X_m}(\gamma_i))$}
        \put(74,21){\color{Orchid}$\alpha_i$\color{black}}
        \put(58,21){$a^+$}
        \put(50,14){$\hat{a}$}
    \end{overpic}
        \caption{Lasso argument to show that truncated length function is continuous.}~\label{CtsTrunclength}
    \end{figure}

    Let $\hat{a}$ denote the truncated geodesic arc representing $a$ in $X_{\infty}$ and orthogonal to $\partial N_{h_0(w)}(\gamma_1)$ and $\partial N_{h_0(w)}(\gamma_2)$, i.e.\ $\ell_{X_{\infty},\Gamma}^w(a)=\ell_{X_{\infty}}(\hat{a})$. Let $a^+$ denote the arc obtained from $\hat{a}$ by extending both ends with geodesic segments of length $\delta_K$. Thus $a^+$ has endpoints on the $h_K$-horocycles at $\gamma_i$, $i=1,2$, and let $\alpha_i$ be the unique geodesic loop at each endpoint, see Figure~\ref{CtsTrunclength}. We have $\ell_{X_{\infty}}(\alpha_i)=2\cdot r(h_K)$.

    Since $N$ was chosen to be deep enough in the node neighborhoods, the concatenation $\alpha_1\cup a^+\cup \alpha_2$ is disjoint from $N$. Thus, the arc $g_m^{-1}(a^+)$ has length
    \begin{equation*}
        \ell_{X_m}(g_m^{-1}(a^+))\leq K\cdot \ell_{X_{\infty}}(a^+)=K\cdot (\ell_{X_{\infty},\Gamma}^w(a)+2\delta_K)
    \end{equation*}
    with endpoints contained in loops of length
    \begin{equation*}
        \ell_{X_m}(g_m^{-1}(\alpha_i))\leq K\ell_{X_{\infty}}(\alpha_i)=2K^{-1}r_0^w<2r^w_{\ell_{X_m}(\gamma_i)}.
    \end{equation*}
    By the above strict inequality, the loop $g_m^{-1}(\alpha_i)$ must be properly contained inside the collar neighborhood $\Cs_{w(\ell_{X_m}(\gamma_i))}(\gamma_i)$, for $i=1,2$. This implies that $g_m^{-1}(a^+)$ meets $\partial\Cs_{w(\ell_{X_m}(\gamma_i))}(\gamma_i)$ for $i=1,2$, and thus
    \begin{equation*}
        \ell_{X_m,\Gamma}^w(a)\leq\ell_{X_m}(g_m^{-1}(a^+))\leq K\cdot (\ell_{X_{\infty},\Gamma}^w(a)+2\delta_K).
    \end{equation*}
    At the limit $K\to 1$, we have
    \begin{equation*}
        \delta_K=\ln\left(\frac{h_0(w)}{h_K}\right)=\ln(2\sinh(r_0^w))-\ln(2\sinh(K^{-2}r_0^w))\to 0
    \end{equation*}
    and thus
    \begin{equation*}
        \limsup\limits_{m\to +\infty}\ell_{X_m,\Gamma}^w(a)\leq\ell_{X_{\infty},\Gamma}^w(a).
    \end{equation*}

    We now show the opposite inequality with a similar argument. Fix again $1<K<2$ and $m\geq m_K$. For $i=1,2$, let $v_{K,i}<w(\ell_{X_m}(\gamma_i))$ be such that
    \begin{equation*}
        r_{\ell_{X_m}(\gamma_i)}(v_{K,i})=\arcsinh(\sinh(\tfrac{\ell_{X_m}(\gamma_i)}{2})\cosh(v_{K,i}))=K^{-2}r_{\ell_{X_m}(\gamma_i)}^w,
    \end{equation*}
    and let $\delta_{K,i}\coloneq w(\ell_{X_m}(\gamma_i))-v_{K,i}$.

    Again, let $\hat{a}$ be the truncated geodesic arc representing $a$ in $X_m$ and orthogonal to $\partial\Cs_{w(\ell_{X_m}(\gamma_i))}(\gamma_i)$, $i=1,2$, i.e.\ $\ell_{X_m,\Gamma}^w(a)=\ell_{X_m}(\hat{a})$. Let $a^+$ be obtained by extending both ends of $\hat{a}$ by geodesic segments of length $\delta_{K,1}$ and $\delta_{K,2}$, respectively. Thus the endpoints of $a^+$ lie on $\partial\Cs_{v_{K,i}}(\gamma_i)$, $i=1,2$. Let $\alpha_i$ be the unique geodesic loop at each endpoint. Then $\ell_{X_m}(\alpha_i)=2\cdot r_{\ell_{X_m}(\gamma_i)}(v_{K,i})$.

    As above, since $a^+$ is contained in $Y_m=g^{-1}(X_{\infty})-N$ (by the choice of deep enough neighborhood $N$), we obtain that $g_m(a^+)$ has length
    \begin{equation*}
        \ell_{X_{\infty}}(g_m(a^+))\leq K\cdot\ell_{X_m}(a^+)=K\cdot (\ell_{X_m,\Gamma}^w(a)+\delta_{K,1}+\delta_{K,2}),
    \end{equation*}
    and, for each $i=1,2$, its $i$-th endpoint is contained in a loop peripheral to the node at $\gamma_i$ in $X_{\infty}$ and of length
    \begin{equation*}
        \ell_{X_{\infty}}(g_m(\alpha_i))\leq K\cdot \ell_{X_m}(\alpha_i)=2K^{-1} r_{\ell_{X_m}(\gamma_i)}^w<2r_0^w,
    \end{equation*}
    hence the loop $g_m(\alpha_i)$ must be contained in $N_{h_0(w)}(\gamma_i)$. Therefore, we obtain
    \begin{equation*}
        \ell_{X_{\infty},\Gamma}^w(a)\leq \ell_{X_{\infty}}(g_m(a^+))\leq K\cdot (\ell_{X_m,\Gamma}^w(a)+\delta_{K,1}+\delta_{K,2}).
    \end{equation*}
    Again, at the limit $K\to 1$, we have (using the inequality $x-y\leq\coth(y)\ln\tfrac{\cosh(x)}{\cosh(y)}$ for $x\geq y>0$, which follows from $\ln\tfrac{\cosh(x)}{\cosh(y)}=\int^x_{y}\tanh(s)\mathrm{d}s\geq (x-y)\tanh(y)$)
    \begin{align*}
        \delta_{K,i}=w(\ell_{X_m}(\gamma_i))-v_{K,i}\leq \coth(v_{K,i})\cdot \ln\left(\frac{\sinh(r_{\ell_{X_m}(\gamma_i)}^w)}{\sinh(K^{-2}\cdot r_{\ell_{X_m}(\gamma_i)}^w)}\right)\to 0
    \end{align*}
    and thus
    \begin{equation*}
        \ell_{X_{\infty},\Gamma}^w(a)\leq\liminf\limits_{m\to +\infty}\ell_{X_m,\Gamma}^w(a),
    \end{equation*}
    concluding the proof: $\ell_{X_m,\Gamma}^w(a)\to \ell_{X_{\infty},\Gamma}^w(a)$ when $m\to+\infty$.

    The proof applies readily to the remaining cases, where $a$ may have some endpoint on $\Gamma_1$ or in original cusps of $\Sigma$, and is even slightly simpler, as $h_0(w)$-horocyclic neighborhoods remain constant along the sequence, in place of the degenerating collar.
\end{proof}

\begin{proposition}[Hexagonal coordinates]\label{HexagonalCoordinates}
    Let $\wabstract$ be a truncation function. Let $(\Gamma,\As)$ be a hexagon decomposition of $\Sigma$. The map
    \begin{align*}
        \begin{split}
            \Hex_{(\Gamma,\As)}^{\wabstract}\colon\bigcup_{\Gamma^0 \subseteq \Gamma} S(\Gamma^0)&\longrightarrow\left(\frac{\R_{\geq 0}\times\R}{\{0\}\times\R}\right)^{|\Gamma|}\times\R_{>0}^{|\As|}\\
            X &\longmapsto \left(([\ell_X(\gamma),\tau_X(\gamma)])_{\gamma\in\Gamma},(\ell_{X,\Gamma}^{\wabstract}(a))_{a\in\As}\right)
        \end{split}
    \end{align*}
    is a closed embedding.
\end{proposition}
\begin{proof}
    We write $\Hex\coloneq \Hex_{(\Gamma,\As)}^{\wabstract}$ throughout. The map $\Hex$ is continuous by Lemma \ref{lem:truncated_length_continuous} and by continuity of the extended Fenchel--Nielsen coordinates. Since the domain and target are metrizable, it suffices to show that $\Hex$ is injective and sequentially proper.

    Given the data of $\Hex(X)$, one can reconstruct $X$ by gluing collar neighborhoods, cusp neighborhoods and right-angled hexagons (which are uniquely determined by the lengths of three opposite sides) with the given truncated lengths. This shows that $\Hex$ is injective.

    It remains to show that $\Hex$ is sequentially proper. Let $(X_m)_{m\in\N}$ be a sequence in the domain such that $\Hex(X_m)$ is convergent. Let $\Ps$ be a pants decomposition that contains $\Gamma$. We shall argue that the sequence of extended Fenchel--Nielsen coordinates $(\FN_{\Ps}(X_m))_{m\in\N}$ has a convergent subsequence. If the coordinates corresponding to a curve $\gamma\in\Gamma$ were divergent, then $\Hex(X_m)$ would diverge as well. Thus, we can assume that the coordinates corresponding to all curves in $\Gamma$ remain bounded or tend to $\Zero$.
    
    Let $\gamma'\in\Ps\setminus\Gamma$. We need to show that its length-twist parameters \[(\ell_{X_m}(\gamma'),\tau_{X_m}(\gamma'))_{m\in\N}\] are contained in a compact subset of $\R_{>0}\times\R$. Since $\As$ is a maximal system of simple arcs in $\Sigma\setminus\Gamma$, there is an arc $a\in\As$ that intersects $\gamma'$.

    Assume that the length-twist parameters along $\gamma'$ are not contained in a compact subset of $\R_{>0}\times\R$. Passing to a subsequence, we must fall into one of the following cases:
    \begin{itemize}
        \item $\ell_{X_m}(\gamma')\to 0$; or
        \item $\tau_{X_m}(\gamma')\to \pm\infty$ while $\ell_{X_m}(\gamma')>\epsilon>0$ for all $m\in\N$; or
        \item $\ell_{X_m}(\gamma')\to+\infty$.
    \end{itemize}
    Either of the first two cases forces $\ell_{X_m,\Gamma}^{\wabstract}(a)\to+\infty$: the first by the collar lemma and the second because of infinite twisting, which contradicts the assumption that $\left(\Hex(X_m)\right)_{m\in\N}$ converges.
    
    Let us assume that we fall into the third case, i.e.\ $\ell_{X_m}(\gamma')\to+\infty$. The curve $\gamma'$ is homotopic to a finite concatenation of subarcs of the $a_{\Gamma}^{\wabstract}$ ($a\in\As$) and of the cusp and collar boundaries (this finite decomposition does not depend on $m$). The former have lengths $\leq\ell_{X_m,\Gamma}^{\wabstract}(a)$, which are bounded by assumption. The latter have respective lengths $h_0(\wabstract)$ and
    \[
        h_{\ell_{X_m}(\gamma)}(w)=\ell_{X_m}(\gamma)\cosh(\wabstract(\ell_{X_m}(\gamma)))\leq\ell_{X_m}(\gamma)\cosh(\wstd(\ell_{X_m}(\gamma))),
    \]
    for some $\gamma\in\Gamma$, which are bounded since the $\ell_{X_m}(\gamma)$ are bounded. Thus $\ell_{X_m}(\gamma')$ remains bounded as well. Therefore, $X_m$ has extended Fenchel--Nielsen coordinates along $\Ps$ whose parameters along components of $\Ps-\Gamma$ do not degenerate. Thus, up to a subsequence, $X_m\to X_{\infty}$ for some $X_{\infty}\in \bigcup_{\Gamma^0 \subseteq \Gamma} S(\Gamma^0)$.
\end{proof}

\section{Estimates on the Teichmüller distance from ideal triangulations}\label{sec:teich}
In this section, we bound the Teichmüller distance between punctured hyperbolic surfaces along shearing deformations. The argument uses Penner's decorated Teichmüller theory. Given a hyperbolic surface equipped with an ideal triangulation, truncating along horocycles around the cusps reduces the problem to studying a decomposition into right-angled hexagons. Uniform control over the geometry of these truncated hexagons yields a bi-Lipschitz map with a controlled constant between the original and the ``sheared'' surfaces, which provides an estimate on their Teichmüller distance.

\subsection{Bi-Lipschitz construction between truncated hexagons}\label{sec:bilip} We begin by recalling several notions from Penner's theory of decorated Teichmüller space, see~\cite[Section~4]{Penner2012}.
Let $\h_1$ and $\h_2$ be two horocycles centered at distinct ideal points of the Poincaré disk $\D$. Let $a$ be the geodesic connecting their respective centers and let $\delta$ denote the signed hyperbolic distance between $\h_1 \cap a$ and  $\h_2 \cap a$. The \emph{lambda-length} of $\h_1$ and $\h_2$ is defined by 
\begin{equation}\label{lambda_hyper_length}
    \lambda(\h_1, \h_2) = e^{\delta/2}.
\end{equation}
A \emph{decoration} of a geodesic is a choice of horocycle at each endpoint. Lambda-lengths are invariant under Möbius transformations. Fixing a decoration on an ideal triangulation determines a canonical truncation of its edges and hence finite edge lengths. Let $B_1, B_2, B_3$ be three horocycles centered at three distinct ideal points. For $\{i,j,l\}=\{1,2,3\}$ let $a_j$ be the geodesic joining the centers of $B_i$ and $B_l$  and set $\lambda_j=\lambda(\h_i, \h_l)$. The hyperbolic length of the horocyclic segment $\h_j \subset B_j$ between $a_i$ and $a_l$ is called the $h$-length of $\h_j$ and is denoted by $h_j=\ell(\h_j)$, see Figure~\ref{FigureLmbdaHlength}. The $h$-lengths and lambda-lengths are related by: 
\begin{equation}\label{h_lambda_relation}
    h_j = \frac{\lambda_j}{ \lambda_i\lambda_l}
\end{equation}
and conversely, 
\begin{equation}\label{lambda_h_relation}
    \lambda_j = \frac{1}{\sqrt{h_i h_l}}.
\end{equation}
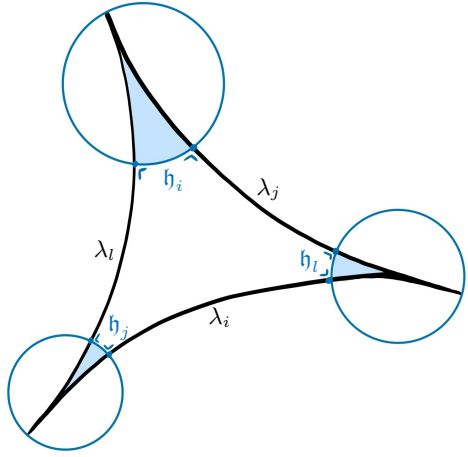
\begin{figure}[h]
    \centering
    \begin{overpic}[width=0.5\linewidth,keepaspectratio]{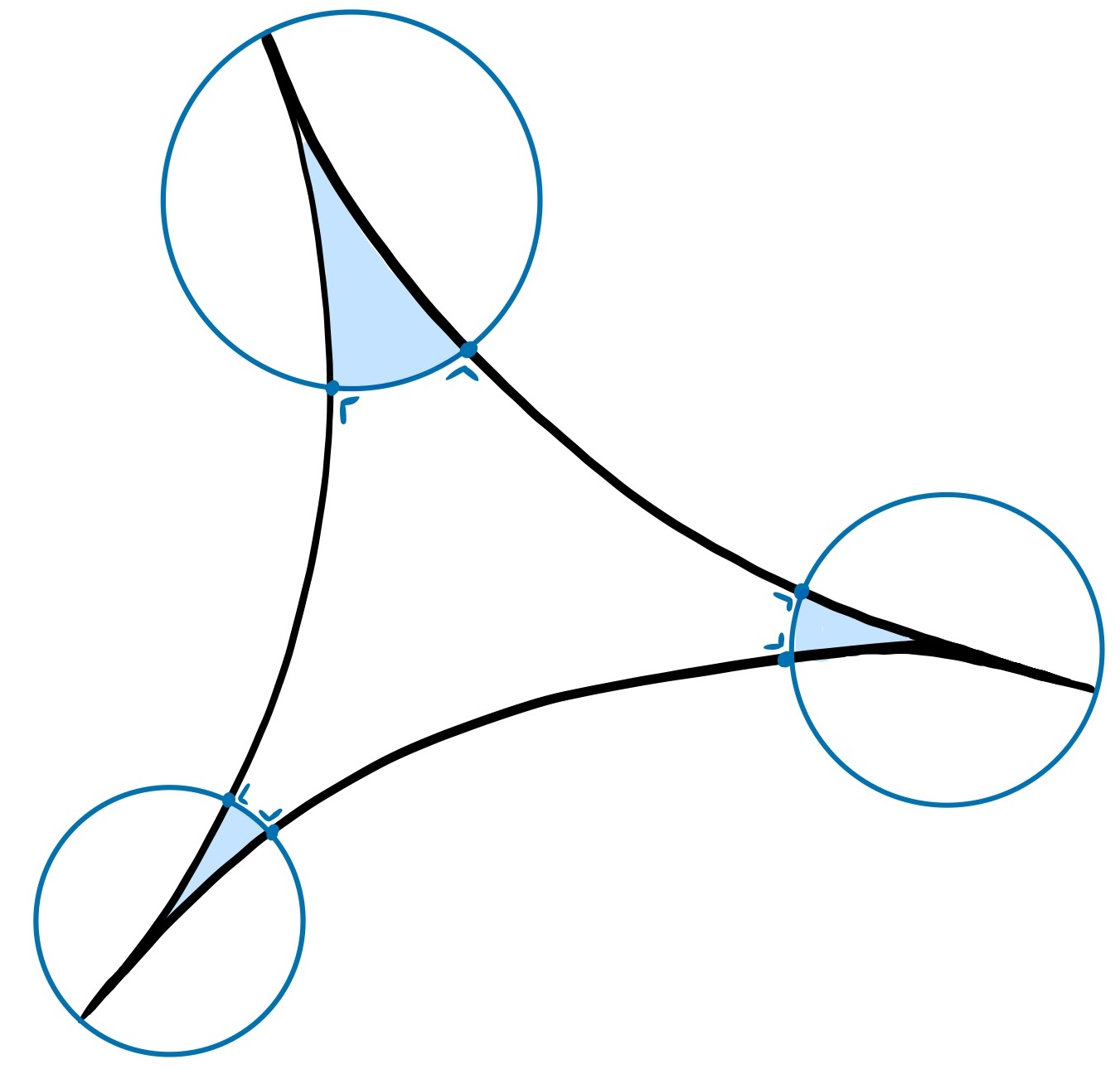}
    \put(35.5,57){\small\color{RoyalBlue}$\h_i$}
    \put(24,26.5){\small\color{RoyalBlue}$\h_j$}
    \put(64,40){\small\color{RoyalBlue}$\h_l$}
    \put(45,28){\small$\lambda_i$}
    \put(55,55){\small$\lambda_j$}
    \put(21,42){\small$\lambda_l$}
    \end{overpic}
    \vspace{-0.4cm}
    \caption{An ideal triangle with a decoration}\label{FigureLmbdaHlength}\vspace{-0.2cm}
\end{figure}
Lambda-lengths also determine shear parameters. Let two decorated ideal triangles share an edge and form an ideal quadrilateral. Let $\lambda_a,\lambda_b,\lambda_c,\lambda_d$ denote the lambda-lengths of the boundary edges in cyclic order, labelled so that the common edge joins the corner between $a$ and $b$ to the corner between $c$ and $d$. Then the shear parameter $s$ along the common edge is given by 
\[
    s=\log\left(\frac{\lambda_b\lambda_d}{\lambda_a\lambda_c}\right).
\] 
In this section we will use three expressions for hyperbolic distances in the upper half-plane $\Hh^2 = \{z = x+iy \mid y > 0\}$:
\begin{itemize}
    \item[-] The distance between any two points $z$ and $z'$ satisfies $\cosh(d(z,z')) = 1 + \frac{|z-z'|^2}{2\Im(z)\Im(z')}$.
    \item[-] The distance from a point $z = x+iy$ to the imaginary axis $i\mathbb{R}_{>0}$ is given by $d(z, i\mathbb{R}_{>0}) = \arcsinh\left(\frac{|x|}{y}\right)$.
    \item[-] If two points $z=x+iy_0$ and $z'=x+iy_1$ lie on the same vertical line, their distance is $d(z,z') = \left|\log\left(\frac{y_0}{y_1}\right)\right|$.
\end{itemize}

\begin{lemma}\label{Shear0-HorocyclicSegments}
    Let $X$ be a punctured hyperbolic surface and let $\mathcal{T}$ be an ideal
    triangulation of $X$ whose shearing coordinates vanish, $\shear_{\mathcal{T}}(X)=0$.
    Fix the decoration of $X$ given by horocycles of length $\eta$ around each cusp.
    Then, for every cusp $c$, the arcs of $\mathcal{T}$ cut the horocycle of the
    decoration around $c$ into $v(c)$ segments, all of the same $h$-length
    \[
        h(c)=\frac{\eta}{v(c)},
    \]
    where $v(c)$ is the number of ends of arcs of $\mathcal{T}$ at $c$, an arc with
    both ends at $c$ being counted twice.
\end{lemma}

\begin{proof}
    Fix a cusp $c$ and let $\tilde{c}$ be a lift to $\mathbb{D}$. Let $a_1, \dots, a_k$ be the lifts of the ideal arcs incident to $\tilde{c}$, indexed cyclically, where $k=v(c)$ and for $i=1,\dots,k$ denote by $\lambda_i$ the lambda-length along $a_i$. For each $i$, let $\mu_i$ denote the lambda-length of the side of the ideal triangle bounded by $\lambda_i$ and $\lambda_{i+1}$ opposite to $\tilde{c}$.
    Since all the shearing coordinates are zero, the formula between shear and lambda-lengths gives:
    \[
        1 = \frac{\mu_i \lambda_{i-1}}{\mu_{i-1} \lambda_{i+1}}
    \]
    and hence
    \[
        \frac{\mu_i}{\lambda_{i+1}} = \frac{\mu_{i-1}}{\lambda_{i-1}} 
    \]
    for every $i=1,\dots,k$. 
    
    Let $h_i(c)$ be the $h$-length of the horocyclic segment between $\lambda_i$ and $\lambda_{i+1}$.
    By~(\ref{h_lambda_relation})
    \[
        h_i(c) = \frac{\mu_i}{\lambda_i \lambda_{i+1}} = \frac{1}{\lambda_i} \frac{\mu_i}{\lambda_{i+1}}.
    \]
    Therefore
    \[ h_i(c) = \frac{1}{\lambda_i} \frac{\mu_{i-1}}{\lambda_{i-1}} = \frac{\mu_{i-1}}{\lambda_{i-1} \lambda_i} = h_{i-1}(c). \]
    Thus all horocyclic segments at $c$ have the same $h$-length, which we denote by $h(c)$. Since the boundary of the horocycle neighbourhood around $c$ has length $L$,
    \[ \eta = \sum_{i=1}^{v(c)} h_i(c) = v(c) \cdot h(c). \]
    Hence \[h(c) = \frac{\eta}{v(c)}\] which proves the lemma.
\end{proof}

A \emph{truncated hexagon} $H$ is a truncated ideal triangle, that is a right-angled hexagon with three alternating geodesic sides $\beta_1$, $\beta_2$, $\beta_3$ obtained by truncating an ideal triangle $T$ with vertices $v_1$, $v_2$, $v_3$ along pairwise disjoint horoballs $B_1$, $B_2$, $B_3$ centered at these vertices. We write in cyclic order
\[
    \partial H=(\h_1,\beta_3,\h_2,\beta_1,\h_3,\beta_2),
\] 
where $\h_i \subset \partial B_i$ and for $\{i,j,l\}=\{1,2,3\}$ the geodesic side $\beta_i$ lies on the side $a_i$ of $T$ joining $v_j$ and $v_l$. 

Next we construct a bi-Lipschitz map between two truncated hexagons. The first lemma provides a bi-Lipschitz map between the boundaries of two truncated hexagons with uniformly bounded geodesic sides and horocyclic sides that are uniformly comparable.

Next we construct a bi-Lipschitz map between two truncated hexagons, in two steps. Lemma~\ref{LemmaBiLipBoundaryHex}  constructs a bi-Lipschitz map $\partial H\to\partial H'$ between the boundaries.  Lemma~\ref{LemmaBiLipHex} extends any such boundary map to a bi-Lipschitz map $H\to H'$ of the hexagons.

\begin{lemma}\label{LemmaBiLipBoundaryHex}
    Let $b_{\max}>0$ and $0<h_{\min}<1/2$ be constants. Let $H$ and $H'$ be two
    truncated hexagons in $\Hh^{2}$ with
    \[
        \partial H=(\h_1,\beta_3,\h_2,\beta_1,\h_3,\beta_2),
        \qquad
        \partial H'=(\h_1',\beta_3',\h_2',\beta_1',\h_3',\beta_2'),
    \]
    and assume that, for $i=1,2,3$,
    \[
        \ell(\beta_i)\leq b_{\max}
        \qquad\text{and}\qquad
        h_{\min}\leq\ell(\h_i),\,\ell(\h_i')\leq\tfrac12 .
    \]
    Then there exists a bi-Lipschitz map
    \[
        F_{\partial}\colon\partial H\longrightarrow\partial H'
    \]
    with bi-Lipschitz constant
    \[
        C_{\partial}
        =\max\left\{N,\ \frac{\sqrt2\,N}{c_0},\ \frac{\sqrt3\,N}{c_1},\
        \frac{\sqrt3\,N}{c_2},\ C_3\right\},
    \]
    where
    \begin{align*}
        M&=\frac{1}{2h_{\min}}, \qquad C=\frac{\log(1/h_{\min})}{\log 2}, \qquad N=\max\{M,C\},\\
        c_0&=\sqrt2\arccosh(9/8),\\
        c_1&=\frac{\arccosh\bigl(1+(b_{\max}^{2}+1)/2\bigr)}{\sqrt{b_{\max}^{2}+1}},\\
        c_2&=\frac{1}{\sqrt2}\min\left\{2\arccosh\left(\frac98\right),\ \frac{\arccosh\bigl(1+h_{\min}^{2}b_{\max}\bigr)}{2b_{\max}}\right\},\\
        C_3&=\max\left\{\frac{2b_{\max}+1}{\log 4},\ \frac{4\log(1/h_{\min})+1}{\log 4}\right\}.
    \end{align*}
    In particular $C_{\partial}$ depends only on $h_{\min}$ and $b_{\max}$.
\end{lemma}

\begin{proof}
    The geometry of a truncated hexagon is determined by the lengths of its three alternating geodesic sides. Set 
    \[
        \delta_i\coloneq \ell(\beta_i), \qquad h_i=\ell(\h_i), \qquad i=1,2,3
    \] and let $\lambda_i$ be the corresponding lambda-length relative to the decoration provided by the horocycles. We use the same notation for the hexagon $H'$ and write \[k_i \coloneq \frac{h_i'}{h_i}.\]
    The relation between $h$-lengths and lambda-lengths then gives:
    \[ 
        \lambda_i' = \frac{1}{\sqrt{k_j k_l h_j h_l}} = \frac{1}{\sqrt{k_j k_l}} \lambda_i.
    \]
    By (\ref{lambda_hyper_length}) and (\ref{lambda_h_relation}), the hyperbolic length $\delta_i$ satisfies: 
    \begin{equation}\label{delta-truncation}
        \delta_i=2\log(\lambda_i)=\log\left(\frac{1}{h_j}\right)+\log\left(\frac{1}{h_l}\right).
    \end{equation}
    For $i=1,2,3$ let $m_i$ be the point on  $\beta_i$ which decomposes $\beta_i$ into two subsegments $\beta_i^j$ and $\beta_i^l$ of length, respectively, $\log\left(\frac{1}{h_j}\right)$ and $\log\left(\frac{1}{h_l}\right)$.

    Define $f:\partial H\to \partial H'$ by sending each horocyclic side $\h_i$ linearly onto $\h_i'$ with scaling factor $k_i$ and by sending each geodesic side $\beta_i$ piecewise linearly onto $\beta_i'$. Hence on $\beta_i^j$ the map has scaling factor:  
    \[
        \kappa_{j}=\frac{\log(1/h_j')}{\log(1/h_j)}
    \]
    and on $\beta_i^l$ the scaling factor is:
    \[
        \kappa_{l}=\frac{\log(1/h_l')}{\log(1/h_l)}
    \]
    Note that $\beta_{i}^j$ and $\beta_{l}^j$ are both re-scaled by the same factor $\kappa_j$. By construction, the map $f$ sends the truncation point $m_i$ to the corresponding point $m_i'$ on $\beta_i'$. We now check that $f$ is bi-Lipschitz using a case-by-case argument. Let $p,q \in \partial H$.

    \begin{itemize}
        \item \textbf{Same horocyclic side.}
        If $p,q$ lie on the same horocyclic side $\h_i$, then the horocyclic arclength is scaled by $k_i$. The hyperbolic distance of the arclength $x$ is $\phi(x)=2\arcsinh(x/2)$. Since $\phi$ is concave with $\phi(0)=0$ the function $\phi(u)/u$ is decreasing. Hence
        \[
            \frac{d(f(p),f(q))}{d(p,q)} = k_i \in \left[\frac{1}{M}, M\right],
        \] 
        where $M=\frac{1}{2h_{\min}}$.
        \item \textbf{Same geodesic side.}
        Assume $p,q \in \beta_i$, and suppose both points lie in the same subsegment, say $\beta_i^j$.
        The map $f$ acts by a constant stretch factor
        \[
        \kappa_j=\frac{\log(1/h_j')}{\log(1/h_j)}.
        \]
        Since $h_j,h_j'\in[h_{\min},1/2]$, this factor satisfies
        \[
         \kappa_j \in [1/C, C],
        \]
        where $C=\frac{\log(1/h_{\min})}{\log(2)}$. Thus
        \[
        \frac{d(f(p),f(q))}{d(p,q)} \in [1/C, C].
        \]
        
        The same estimate holds on $\beta_i^l$.
        
        If $p\in \beta_i^j$ and $q\in \beta_i^l$, recall that $m_i$ is the ``center point'' for the stretch factor along $\beta_i$. Then
        \[
        d(p,q)=d(p,m_i)+d(m_i,q),
        \qquad
        d(f(p),f(q))=\kappa_j d(p,m_i)+\kappa_l d(m_i,q).
        \]
        Using the previous bounds on $\kappa_j,\kappa_l$ gives
        \[
        \frac{d(f(p),f(q))}{d(p,q)} \in [1/C, C].
        \]
        \item \textbf{Consecutive sides.}
        Assume $p \in \h_i$ and $q \in \beta_j$, and let $c$ be the corner point between $h_i$ and $\beta_j$. Let 
        \[
            x = d_{\h_i}(p, c) 
            \qquad 
            y = d_{\beta_j}(q, c)
        \] be the distances along the horocycle and geodesic sides respectively. 
        Let \[y'= d(f(q), f(c)),\] 
        by construction, 
        \begin{itemize}
            \item if $q\in \beta_{j}^i$ then $y'=\kappa_i \cdot y$;
            \item if $q\in \beta_{j}^l$ then $y'=\kappa_i \cdot \log(1/h_i) + \kappa_l \cdot (y-\log(1/h_i))$.
        \end{itemize}
        In both cases, from earlier computations we know that $\kappa_i, \kappa_l \in [1/C,C]$ thus $ k_{\beta_j}(y) =y'/y $ satisfies $k_{\beta_j}(y) \in [1/C,C]$. 
        
        The distance \[z = d(p, q)\] is given by 
        \[
            \cosh(z) = \cosh(y) + \frac{1}{2}x^2 e^y.
        \]
        Set \[z=g_0(x,y)=\arccosh\left(\cosh(y) + \frac{1}{2}x^2 e^y\right).\]
        \begin{claim}
            One has the coarse bounds 
            \[
                c_0 \sqrt{x^2 +y^2} \leq g_0(x,y) \leq x+y,
            \] 
            where $c_0=\sqrt{2}\arccosh(9/8)$.
        \end{claim}
        \begin{proofclaim}
            Indeed, the upper bound follows from the triangle inequality. For the lower bound, observe that $\cosh(z)\geq \cosh(y)$, so
            \[
                z \geq y.
            \]
            Additionally, since $y\geq 0$ one has 
            \[
                \cosh(z)\geq 1+\frac{1}{2}x^2.
            \] 
            Let $\phi(x)=\arccosh\left(1+\frac{1}{2}x^2\right)$. The function $\phi$ is concave for $x>0$ and $\phi(0)=0$ thus $\phi(x)/x$ is decreasing. Recall that $h_i \leq 1/2$, by letting $c_0'=2\phi(1/2)$ we obtain $\phi(x)\geq c_0' x$. Thus \[z \geq c_0' x.\] Combining these two lower bounds on $z$, we obtain 
            \[
                2z^2 \geq y^2 + (c_0')^2 x^2 \Rightarrow z \geq \frac{\min(c_0',1)}{\sqrt{2}} \sqrt{x^2 + y^2},
            \] 
            and we can choose $c_0 = \frac{\min(c_0',1)}{\sqrt{2}}$ to conclude the claim.
        \end{proofclaim} 
        Under $f$, the new distance 
        \[
            z' = d(f(p), f(q)).
        \] 
        satisfies 
        \[
            z' = g_0(k_i x, k_{\beta_j}(y) y).
        \]  
        The bounds on $g_0$ yield: 
        \begin{gather*}
          c_0\sqrt{k_i^2 x^2 + k_{\beta_j}(y)^2 y^2}  \leq z' \leq k_i x +  k_{\beta_j}(y) y\\
          c_0 \min(k_i, k_{\beta_j}(y))\sqrt{x^2 + y^2} \leq z' \leq \max(k_i, k_{\beta_j}(y))(x+y)\\
          \frac{c_0\min(k_i, k_{\beta_j}(y))}{\sqrt{2}}(x+y) \leq z' \leq \max(k_i, k_{\beta_j}(y)) \sqrt{2} \sqrt{x^2 + y^2}\\
        \frac{c_0\min(k_i, k_{\beta_j}(y))}{\sqrt{2}} g_0(x,y) \leq z' \leq \frac{\max(k_i, k_{\beta_j}(y)) \sqrt{2}}{c_0} g_0(x,y)
        \end{gather*}
        Set $N=\max\{M,C\}$, since $k_i \in [1/M,M]$ and $k_{\beta_j}\in [1/C, C]$ we obtain \[ \frac{c_0}{\sqrt{2}N} \leq \frac{d(f(p), f(q))}{d(p,q)} \leq \frac{\sqrt{2}}{c_0}N.\] Thus, we choose $C_0 =\frac{2}{\arccosh(9/8)}N$ to write  \[\frac{1}{C_0} \leq \frac{d(f(p), f(q))}{d(p,q)} \leq C_0.\]

        \item \textbf{Non-consecutive sides, separated by a single intermediate side.} There is two cases.
        
        \smallskip
        
        \textit{Case 1: the middle side is geodesic.} Assume $p\in\h_i$, and $q\in \h_j$. Let $c_i$ be the corner $\h_i \cap \beta_l$ and $c_j$ be the corner $\h_j \cap \beta_l$. Let 
        \[
            x = d_{h_i}(p, c_i), \qquad w = d_{h_j}(q, c_j).
        \]
        The distance $z= d(p, q)$ satisfies
        \[
            \cosh(z)= \cosh(\delta_l) + \frac{e^{\delta_l}}{2}(x^2+w^2+x^2w^2)+xw.
        \]
        Writing $z=g_1(x,w,\delta_l)$, one obtains coarse bounds
        \[
        c_1 \sqrt{\delta^2+x^2+w^2}
        \leq g_1(x,w, \delta_l) \leq \frac{1}{c_1}(\delta+x+w),
        \]
        with  $c_1= \frac{\arccosh(1+(b_{\max}^2+1)/2)}{\sqrt{b_{\max}^2+1}}$.

        Indeed, the upper bound follows from the triangle inequality. For the lower bound, observe that, 
        \[
            \cosh(z)\geq \cosh(\delta_l) + \frac{1}{2}(x^2+w^2) \geq 1 + \frac{\delta_l^2+x^2+w^2}{2}.
        \] 
        Next, recall that $0 < \delta_l < b_{\max}$ and $x,w<1/2$ and so 
        \[
            \delta_l^2+x^2+w^2< b_{\max}^2+1.
        \]
        Let $\phi(u)=\arccosh(1+\frac{u^2}{2})$, since $\phi$ is concave for $u>0$ and $\phi(0)=0$ we obtain 
        \[
            \phi(u)\geq \frac{\phi\left(\sqrt{b_{\max}^2+1}\right)}{\sqrt{b_{\max}^2+1}}u.
        \] 
        Thus, 
        \begin{align*}
            z &\geq \arccosh\left( 1 + \frac{\delta_l^2+x^2+w^2}{2} \right)
            =\phi\left(\sqrt{\delta_l^2 + x^2 + w^2}\right)\\
            &\geq c_1 \sqrt{\delta_l^2 + x^2 + w^2}
        \end{align*}
        where $c_1= \frac{\arccosh(1+(b_{\max}^2+1)/2)}{\sqrt{b_{\max}^2+1}}$.
        
        Under $f$, the new distance, 
        \[
            z'=d(f(p),f(q)).
        \] 
        satisfies
        \[
                z'= g_1(k_i x,k_j w, k_{\beta_l}\delta_l)
        \]
        where $k_{\beta_l}=\delta_l'/\delta_l$. Recall that $\delta_l'$ is the hyperbolic length of $\beta_l'$ and satisfies
        \[
            \delta_l'= \kappa_j\log\left(\frac{1}{h_j}\right)+\kappa_i\log\left(\frac{1}{h_i}\right)
        \]
        where $k_{\beta_l} \in [1/C, C]$. Since $k_i,k_j\in [1/M,M]$ and $k_{\beta_l} \in [1/C,C]$ by defining $C_1=\frac{\sqrt{3}}{c_1} N$ we obtain:
            \[ \frac{1}{C_1} \leq \frac{d(f(p), f(q))}{d(p,q)} \leq C_1.\]
            
            \textit{Case 2: the middle side is horocyclic.} Assume that $p\in \beta_i$ and $q \in \beta_j$. Let $c_i$ be the corner $\beta_i \cap \h_l$ and $c_j$ be the corner $\beta_j \cap \h_l$. Let 
            \[
                x = d_{\beta_i}(p, c_i) \qquad w = d_{\beta_j}(q, c_j).
            \] 
            From hyperbolic computations the distance $z = d(p, q)$, satisfies, 
            \[
                \cosh(z)=\cosh(x-w)+\frac{h_l^2}{2}e^{x+w}.
            \]
            Writing $z=g_2(x,w,h_l)$, we obtain the following bounds
            \begin{claim}
                \[
                c_2\sqrt{x^2+w^2+h_l^2}\leq g_2(x,w,h_l) \leq \frac{1}{c_2}(x+w+h_l),
            \] 
            where $c_2 = \frac{1}{\sqrt{2}}\min\left\{ 2\arccosh\left(\frac{9}{8}\right),\;\frac{\arccosh(1 + h_{\min}^2b_{\max} )}{2b_{\max}} \right\}$.
            \end{claim}
            \begin{proofclaim}
                The upper bound follows from the triangle inequality. For the lower bound, first observe that 
                \[
                    \cosh(z)\geq 1+ \frac{h_l^2}{2}.
                \] 
                Next, define $\phi(u)=\arccosh(1+ u^2/2)$, as before, for $u>0$ the function $\phi$ is concave. Since $h_l \in [h_{\min},1/2]$ we obtain:
                \begin{equation}\label{eq_g2_bound1}
                    z\geq \phi(h_l)\geq c_2'h_l,
                \end{equation}
                where $c_2'= 2\phi(1/2)=2\arccosh\left(\frac{9}{8}\right)$. 
                
                Secondly observe that 
                \[
                    \cosh(z)\geq 1+\frac{h_{\min}^2}{2}e^{x+w} \geq 1+\frac{h_{\min}^2}{2}(x+w).
                \] 
                Let $\psi(u)=\arccosh\left( 1+ \frac{h_{\min}^2}{2}u\right)$. Because $\psi$ is concave and $\psi(0)=0$ for $0<u<2b_{\max}$ we obtain $\psi(u)\geq \frac{\psi(2b_{\max})}{2b_{\max}}u$ and so \begin{equation}\label{eq_g2_bound2}
                    z\geq \psi(x+w) \geq c_2''(x+w),
                \end{equation}
                where $c_2'' = \frac{\arccosh(1 + h_{\min}^2b_{\max} )}{2b_{\max}}$. By combining (\ref{eq_g2_bound1}) and (\ref{eq_g2_bound2}) we obtain \begin{align*}
                    2z^2 &\geq (c_2')^2h_l^2 +  (c_2'')^2(x+w)^2 \\
                    &\geq (c_2')^2h_l^2 +  (c_2'')^2(x^2+w^2) \\
                    &\geq \min(c_2', c_2'')^2(x^2+w^2+h_l^2).
                \end{align*} 
                Hence the claim follows by letting $c_2 = \frac{1}{\sqrt{2}}\min(c_2', c_2'')$. 
            \end{proofclaim}

        Under $f$, the new distance 
        \[
            z'=d(f(p),f(q))
        \]
        satisfies
        \[
            z'=g_2(k_{\beta_i}(x) x, k_{\beta_j}(w) w, k_l h_l)
        \]
        where, as proved above, $k_{\beta_i}(x), k_{\beta_j}(w) \in [1/C, C]$. Since $k_l \in [1/M, M]$, by defining $C_2 = \frac{\sqrt{3}}{c_2}N$ we obtain: 
        \[ 
            \frac{1}{C_2} \leq \frac{d(f(p), f(q))}{d(p,q)} \leq C_2.
        \]
        
        \item \textbf{Non-consecutive sides, separated by two intermediate sides.} Assume $p \in \h_i$ and $q \in \beta_i$. On one hand \[d(p,q)\geq d(\h_i, \beta_i) \geq  \log(\frac{2}{h_i}) \geq \log(4),\] similarly $d(f(p),f(q))\geq \log(4)$. On the other hand \[ d(p, q) \leq \ell(\beta_i) + \ell(\h_j) + \ell(\beta_l) + \ell(\h_i) \leq 2 b_{\max} + 1\] and 
        \begin{align*}
            d(f(p), f(q)) &\leq \ell(\beta_i') + \ell(\h_j') + \ell(\beta_l') + \ell(\h_i') \leq 2 \max_i \delta_i' + 1 \\ &\leq 4 \log(1/h_{\min})+1.
        \end{align*}
        Consequently, \[ \frac{1}{C_3} \leq \frac{d(f(p), f(q))}{d(p,q)} \leq C_3 \]
        where $C_3= \max\left\{ \frac{2 b_{\max} + 1}{\log(4)}, \frac{4 \log(1/h_{\min})+1}{\log(4)} \right\}$.
        \end{itemize}
    We conclude that $f:\partial H \longrightarrow \partial H'$ is a bi-Lipschitz map with bi-Lipschitz constant \[C_{\partial}=\max\left\{ N, \frac{\sqrt{2}N}{c_0},\frac{\sqrt{3}N}{c_1},\frac{\sqrt{3}N}{c_2}, C_3 \right\},\] where 
    \begin{align*}
        M&=\frac{1}{2h_{\min}}, \qquad C=\frac{\log(1/h_{\min})}{\log(2)}, \qquad N=\max\{M,C\},\\
        c_0&=\sqrt{2}\arccosh(9/8),\\
        c_1&= \frac{\arccosh(1+(b_{\max}^2+1)/2)}{\sqrt{b_{\max}^2+1}},\\
        c_2 &= \frac{1}{\sqrt{2}}\min\left\{ 2\arccosh\left(\frac{9}{8}\right),\;\frac{\arccosh(1 + h_{\min}^2b_{\max} )}{2b_{\max}} \right\}\\
        C_3&= \max\left\{ \frac{2 b_{\max} + 1}{\log(4)}, \frac{4 \log(1/h_{\min})+1}{\log(4)} \right\}.
    \end{align*}
\end{proof}

The following lemma extends the map from Lemma~\ref{LemmaBiLipBoundaryHex} to a bi-Lipschitz map between two truncated hexagons with uniformly bounded geodesic sides and horocyclic sides that are uniformly comparable.

\begin{lemma}\label{LemmaBiLipHex}
    Let $H$ and $H'$ be two truncated hexagons with 
    \[
        \partial H=(\h_1,\beta_3,\h_2,\beta_1,\h_3,\beta_2),
        \qquad
        \partial H'=(\h_1',\beta_3',\h_2',\beta_1',\h_3',\beta_2').
    \]
    Assume that the geodesic sides of $H$ satisfy
    \[
    \ell(\beta_i)\le b_{\max},
    \qquad i=1,2,3,
    \]
    and that the horocyclic sides of $H$ and $H'$ are uniformly bounded from above and below:
    \[
        h_{\min} \le \ell(\h_i), \;\ell(\h_i') \le 1/2,
    \]
    where $0<h_{\min}<1/2$. Then there exists a bi-Lipschitz map
    \[
    F : H \longrightarrow H'
    \]
    whose bi-Lipschitz constant $C_F$ depends only on $h_{\min}$, and $b_{\max}$.
    Specifically,
    \[
        C_F = \max\{C_{f}, 1\},
    \]
    where 
    \[ 
        C_{f} = \frac{3+\sqrt{5}}{2} \frac{C}{2h_{\min}}\sqrt{C^2+ \frac{1+(C-1)^2}{4h_{\min}^2}}
    \]
     with  \[C=\frac{\log(1/h_{\min})}{\log(2)}.\] 
\end{lemma}

\begin{proof} 
    Let $T$ and $T'$ be two decorated ideal triangles in the Poincaré disk $\D$. Truncating along the chosen horocycles produces right-angled hexagons $H$ and $H'$. Assume that $H$ and $H'$ satisfy the hypotheses of the lemma. Set 
    \[
        \delta_i\coloneq \ell(\beta_i), \qquad h_i=\ell(\h_i), \qquad i=1,2,3
    \] and let $\lambda_i$ be the corresponding lambda-length relative to the decoration provided by the horocycles. We use the same notation for the hexagon $H'$ and write \[k_i \coloneq \frac{h_i'}{h_i}.\]
    
    Let $v_1, v_2, v_3$ denote the ideal vertices of $T$. The spine of $T$, denoted by $\spine(T)$, is the set of points with more than one closest-projection point to $\partial T$. The truncated spine 
    \[
        \spine(H)=\spine(T) \cap H
    \]
    is a tripod with three geodesic legs, denoted by $\sigma_1,\sigma_2,\sigma_3$, meeting at a single point $O \in H$. The point $O$ is the center of the inscribed circle of $T$. Equivalently, $O$ is the center of the contact triangle $\Delta$, whose sides are three horocyclic leaves emanating from $v_1$, $v_2$ and $v_3$ and meeting tangentially. We use the analogous notation $\spine(H')$, $\sigma_i'$, $O'$ and $\Delta'$ for the hexagon $H'$. 
    
    For $i=1,2,3$, let
    \[
        L_i=d(O,\h_i) \qquad L_i'=d(O',\h_i').
    \] 
    A direct computation yields 
    \[
        L_i=\log\left(\frac{2}{\sqrt{3}h_i}\right)
    \] 

    Since  $h_i \in [h_{\min}, 1/2]$ we have 
    \[
        \log\left(\frac{4}{\sqrt{3}}\right) \leq L_i \leq \log\left(\frac{2}{\sqrt{3}h_{\min}}\right).
    \] 
    Similarly,
    \[
        L_i'=\log\left(\frac{2}{\sqrt{3}h_i'}\right)=L_i-\log(k_i).
    \] 
    Observe also that the distance from the horocycle side $\h_i$ to the contact triangle $\Delta$ is
    \[  
        L_i- \log\left( \frac{2}{\sqrt{3}}\right)= \log\left(\frac{1}{h_i}\right).
    \]

    Let \[f:\partial H \longrightarrow \partial H'\] be the boundary map given by Lemma~\ref{LemmaBiLipBoundaryHex}. We use the notation introduced in the proof of that lemma. 

    We now equip $H$ with the horocyclic foliation from its underlying ideal triangle $T$. That is, for each ideal vertex $v_i$, consider the family of horocycles centered at $v_i$. These horocycles are orthogonal to the adjacent geodesic sides of $T$ and to the corresponding leg of the spine. Three distinguished leaves meet tangentially and bound the contact triangle $\Delta$. The same construction holds for $\Delta' \subset T'$. 

    For $i=1,2,3$ let $m_i$ be the point on  $\beta_i$ which decomposes $\beta_i$ into two subsegments $\beta_i^j$ and $\beta_i^l$ given by relation~(\ref{delta-truncation}). The points $m_1,m_2,m_3$ are the three vertices of $\Delta$, and similarly the points $m_1', m_2', m_3'$ are the vertices of $\Delta'$. Indeed consider $m_i$, by definition 
    \[
        d(m_i, \h_j)=\log\left(\frac{1}{h_j}\right)
    \]
    Hence $m_i$ lies on the horocycle leaf based at $v_j$ whose length is 
    \[
        h_j e^{\log\left(1/h_j\right)}=1.
    \]
    This leaf is one side of the contact triangle $\Delta$. Since $f$ maps each point $m_i$ to the corresponding point $m_i'$, it maps the vertices of $\Delta$ to the vertices of $\Delta'$.

    For $i=1,2,3$, let $Q_i$ denote the region bounded by $\h_i$, the adjacent subsegment $\beta_l^i$, $\beta_j^i$ and the side of the contact triangle $\Delta$ joining $m_l$ to $m_j$: 
    \[
        Q_i=(\h_i,\; \beta_l^i,\; \partial\Delta(m_l,m_j),\; \beta_j^i).
    \]
    Let $Q_i'$ be the corresponding region in $H'$. 

    \begin{figure}[h]
        \centering
        \begin{overpic}[width=0.7\linewidth,keepaspectratio]{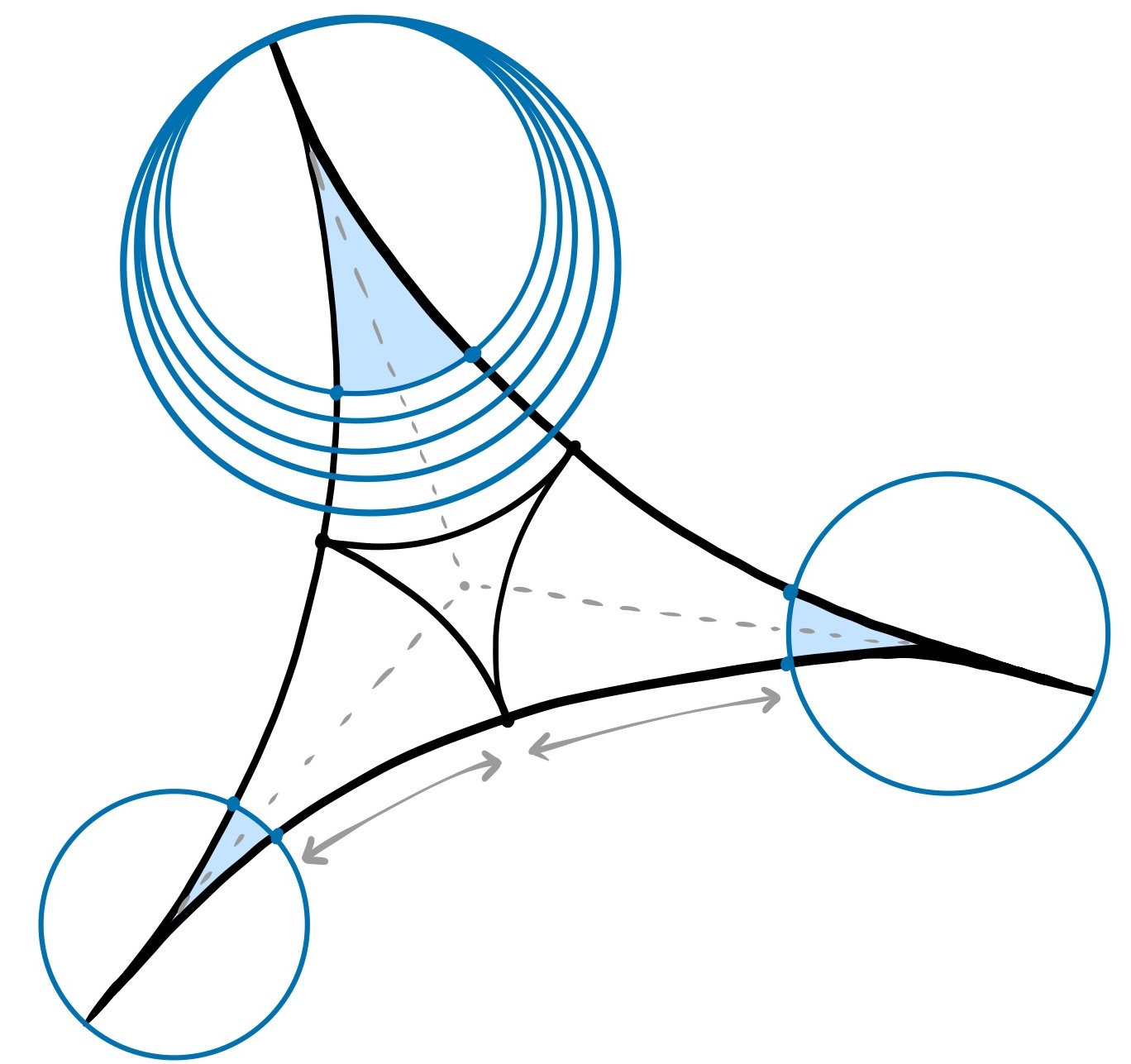}
        \put(34.5,62){\small\color{RoyalBlue}$\h_i$}
        \put(23,23.25){\small\color{RoyalBlue}$\h_j$}
        \put(66,38){\small\color{RoyalBlue}$\h_l$}
        \put(29,32){\small\color{Gray}$\sigma_j$}
        \put(56,41){\small\color{Gray}$\sigma_l$}
        \put(52,54){\small$m_j^{i,l}$}
        \put(45,34){\small$m_i$}
        \put(21,45){\small$m_l^{i,j}$}
        \put(33,18){\small\color{Gray}$\log\left(\frac{1}{h_j}\right)$}
        \put(54,26){\small\color{Gray}$\log\left(\frac{1}{h_l}\right)$}
        \end{overpic}
        \vspace{-0.4cm}
        \caption{Construction of a bi-Lipschitz map between truncated hexagons.}\label{FigureConstHexBiLipMap}
    \end{figure}
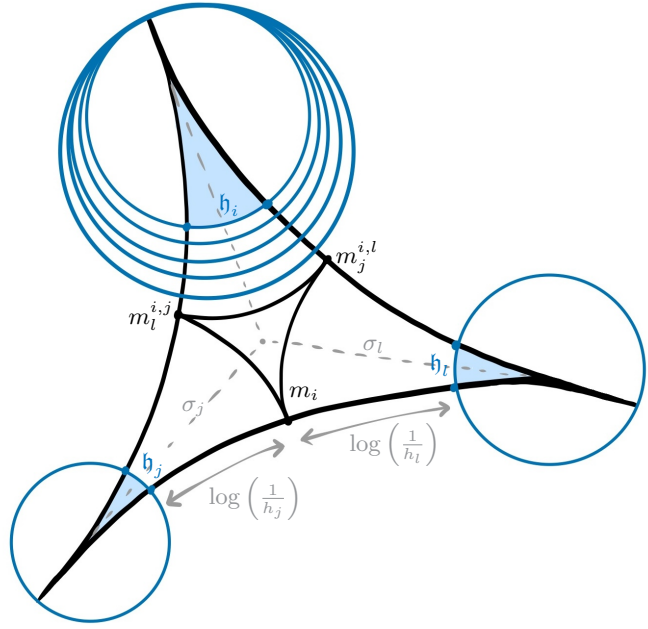
    
    For each $i=1,2,3$ let 
    \[
        \alpha_i : [0, L_i] \longrightarrow \sigma_i
    \] 
    be the parametrization of the $i$-th leg of the spine. For $0 \leq s \leq  \log\left(1/h_i\right)$ let $C_i(s)$ denote the horocycle leaf orthogonal to $\beta_j,\beta_l,\sigma_i$ and passing through $\alpha_i(s)$. Let 
    \[
        \lambda_i : [0, \tau_i(s)] \longrightarrow C_i(s)
    \] 
    be the arc-length parametrization of $C_i(s)$. The length of $C_i(s)$ between $\beta_j$ and $\beta_l$ is 
    \[
        \tau_i(s)=h_{i} e^{s}
    \] 
    Similarly let 
    \[
        \tau_i'(s')=h_{i}' e^{s'}
    \] 
    denote the corresponding leaf length in $Q_i'$. Every point of $H\smallsetminus \Delta$ lies on a horocycle leaf. Consequently, each point of $Q_i$ can be written in coordinates 
    \[p=(s,t),\] 
    where 
    \[ 0 \leq s \leq \log\left(\frac{1}{h_i}\right), \qquad 0 \leq t \leq \tau_i(s).\]
    Here $s$ encodes the distance from the horocycle side $\h_i$ along the spine while $t$ records the position along the leaf $C_i(s)$. 
    
    Recall that both boundary segment $\beta_l^i$ and $\beta_j^i$ are re-scaled by the same factor 
    \[
        \kappa_i=\frac{\log(1/h_i')}{\log(1/h_i)}.
    \]
    We define 
    \[f_{i}: Q_i \longrightarrow Q_i'\]
    by 
    \[ 
        f_{i}(s,t)= \left(\sigma(s,t), \theta(s,t)\right) = \left(s \cdot \kappa_i, \frac{\tau_i'(s\cdot \kappa_i)}{\tau_i(s)}t \right).
    \]
    Thus $ f_{i}$ rescales the spine direction by a factor $\kappa_i$ and maps each horocycle leaf linearly onto the corresponding horocycle leaf of $Q_i'$.

    We first verify that $f_{i}$ agrees with boundary map $f: \partial H \longrightarrow \partial H'$ constructed in Lemma~\ref{LemmaBiLipBoundaryHex}. 
    \begin{itemize}
        \item[-] The geodesic side $\beta_l^i$ is given by $t=0$ $(s,0)$ and hence \[f_{i}(s,0)=(s \cdot \kappa_i, 0)\] which is precisely the restriction of $f$ to $\beta_l^i$.
        \item[-] Similarly, the side $\beta_j^i$ is given by $t=\tau_i(s)$, and therefore \[f_{i}(s, \tau_i(s))=(s\kappa_i, \tau_i'(s\cdot \kappa_i))\] which agrees with the re-scaling of $\beta_j^i$ under $f$.
        \item[-] On the horocyclic segment $\h_i$, corresponding to $s=0$, we obtain 
        \[ 
            f_{i}(0,t)= \left(0, \frac{\tau_i'(0)}{\tau_i(0)}t\right)=\left(0,\frac{h_i'}{h_i}t\right)
        \] 
        so the map stretches $h_i$ with the factor $k_i= h_i'/h_i$, as required.
        \item[-] Finally, the side of the contact triangle $\partial \Delta(m_l, m_j)$ is given by \[s=\log\left(\frac{1}{h_i}\right).\] Since
        \[
            \tau_i\left(\log\left(\frac{1}{h_i}\right)\right)=\tau_i'\left(\log\left(\frac{1}{h_i'}\right)\right)=1
        \] 
        its image by $f_{i}$ is the corresponding side of $\Delta'$. Hence $f_{i}$ coincides with $f$ on $\partial Q_i$.
    \end{itemize}

    In the coordinates $(s,t)$ the hyperbolic metric is \[(1+t^2)ds^2+dt^2-2t\,ds \,dt.\] Since $0\leq t \leq \tau_i(s)=h_i e^s \leq 1$ the corresponding matrix
    \[
        \left( \begin{matrix}
        1+t^2 & -t\\
        -t & 1 \\
        \end{matrix} \right)
    \]  
    has determinant $1$ and its eigenvalues lies in $\left[\frac{3-\sqrt{5}}{2}, \frac{3+\sqrt{5}}{2}\right]$. The hyperbolic metric is therefore $\frac{1+\sqrt{5}}{2}$-bi-Lipschitz to the Euclidean matric $ds^2+dt^2$. To prove that $f_i$ is bi-Lipschitz it suffices to show that it is bi-Lipschitz with respect to the Euclidean metric and multiply the resulting constant by $(\frac{1+\sqrt{5}}{2})^2=\frac{3+\sqrt{5}}{2}$.  Hence it remains to show that the absolute value of determinant of the Jacobian matrix of $f_{i}$ is bounded below and that each entry of the matrix is bounded above.
    We compute the partial derivatives of $(\sigma(s,t),\theta(s,t))$ with respect to $(s,t)$. A direct calculation gives 
    \[\frac{\partial\sigma}{\partial s} = \kappa_i, \qquad 
    \frac{\partial \sigma}{\partial t} = 0 \]
    and 
    \[\frac{\partial \theta}{\partial t} = \frac{\tau_i'(s\cdot \kappa_i)}{\tau_i(s)}= k_i^{1+\frac{s}{\log(h_i)}}\]
    while 
    \[ 
        \frac{\partial\theta}{\partial s} = \frac{\partial}{\partial s} \left(t\frac{\tau_i'(s\cdot \kappa_i)}{\tau_i(s)} \right)= t k_i^{1+\frac{s}{\log(h_i)}} \frac{\log(k_i)}{\log(h_i)}= t \frac{\tau_i'(s\cdot \kappa_i)}{\tau_i(s)} (\kappa_i-1).
    \]

    The Jacobian matrix of $f_{i}$ is given by 
    \[
        Df_{i}=\left( \begin{matrix}
            \kappa_i & 0\\
            t(\kappa_i-1) \frac{\tau_i'(s\kappa_i)}{\tau_i(s)} & \frac{\tau_i'(s\cdot \kappa_i)}{\tau_i(s)} \\
    \end{matrix} \right). \]

    We compute its determinant: 
    \[
        \det(Df_{i})=\kappa_i\frac{\tau_i'(s\cdot \kappa_i)}{\tau_i(s)}= \kappa_i k_i^{1+\frac{s}{\log(h_i)}}.
    \] 
    Recall that, by proof Lemma~\ref{LemmaBiLipBoundaryHex}, $\kappa_i\in[1/C, C]$ and $k_i\in [1/M, M]$. Since $s\in [0,\log(1/h_i)]$ we obtain \[|\det(Df_{i})|\geq \frac{1}{CM}.\]

    Next, we bound the entries of the Jacobian matrix of $f_{i}$. Since $0 \leq s\leq \log\left(\frac{1}{h_{\min}}\right)$, we have:
    \[
        \left| \frac{\partial \sigma}{\partial s}\right| \leq C, \qquad \left| \frac{\partial \theta}{\partial t} \right| \leq M,
    \]
    and 
    \[
        \left|\frac{\partial\theta}{\partial s}\right| \leq M \frac{\log(M)}{\log(2)}.
    \]
    Since the spectral norm $\|Df_{i}\|$ is always less than or equal to the Frobenius norm, we obtain: 
    \[ 
        \|Df_{i}\| \leq \sqrt{C^2 + M^2\left(1+\frac{\log^2 M}{\log^2 2}\right)}.
    \]
    Since the norm of the inverse is bounded by the original norm divided by the determinant,
    \[ \| Df_{i}^{-1} \| \leq CM
    \sqrt{C^2 + M^2\left(1+\frac{\log^2 M}{\log^2 2}\right)}.\]
    Thus $f_{i}$ is bi-Lipschitz with constant   
    \[ C_{f_{i}} = \frac{3+\sqrt{5}}{2} CM
    \sqrt{C^2 + M^2\left(1+\frac{\log^2 M}{\log^2 2}\right)}.
    \]

    For $i=1,2,3$, at $s=\log(1/h_i)$ one has $f_i(s,t)=(\log(1/h_i'), t)$ thus the map $f_{i}$ sends isometrically $\partial \Delta \cap Q_i$ onto $\partial \Delta' \cap Q_i'$. Let 
    \[ I: \Delta \longrightarrow \Delta' \] 
    be the isometry between the two contact triangle, in particular it maps $m_i$ to $m_i'$. Define 
    \[ F:H\longrightarrow H' \]  
    by setting 
    \[ F|_{Q_i}=f_{i}, \qquad F|_{\Delta}=I.\]
    The maps agree on their common boundaries, so $F$ is well-defined. Since for each $i=1,2,3$ the map $f_i$ is $C_{f_{i}}$-bi-Lipschitz and $I$ is an isometry it follows that $F$ is bi-Lipschitz with bi-Lipschitz constant 
    \[C_F = \{ \max_i C_{f_{i}}, 1\}.\]
\end{proof}

It is clear that any two compact hexagons are bi-Lipschitz. Lemma~\ref{LemmaBiLipBoundaryHex} and Lemma~\ref{LemmaBiLipHex} state that any two truncated hexagons whose geodesic sides are bounded above by $b_{\max}$ and whose horocyclic sides lie in $[h_{\min},1/2]$ are bi-Lipschitz homeomorphic, via a map with an explicit constant $C_F=C_F(h_{\min},b_{\max})$. The constant depends on the pair of hexagons only through the two bounds $h_{\min}$ and $b_{\max}$ and the map is defined on the boundary by a function of the lengths of the sides. 

\subsection{Teichmüller distance along a shearing deformation}\label{sec:shearpath}
We now apply the two previous lemmata hexagon by hexagon. Truncating the cusps of a surface carrying an ideal triangulation decomposes it into truncated hexagons. Two such surfaces with bounded shear parameters satisfy the hypotheses of Lemma~\ref{LemmaBiLipBoundaryHex} and Lemma~\ref{LemmaBiLipHex}. Gluing the local maps, after a correction along the shared geodesic sides, yields a global bi-Lipschitz map with an explicit bi-Lipschitz constant and hence a bound on the Teichmüller distance. First we need the following auxiliary lemma, which
provides the correction to glue the maps defined on each hexagon.

\begin{lemma}\label{correction_lemma}
    Let $H$ be a truncated hexagon with \[
        \partial H=(\h_1,\beta_3,\h_2,\beta_1,\h_3,\beta_2) 
    \]
    and let $0 < h_{\min} < 1/2$ such that for $i=1,2,3$ \[h_{\min} \leq h_i=\ell(\h_i) \leq 1/2 .\] Set $\rho_0=\frac{1}{2}\arcsinh(h_{\min})$ and $\delta_{\max}=\max_i \ell(\beta_i)$. For $i=1,2,3$, let $\psi_i:\beta_i \longrightarrow \beta_i$ be a homeomorphism fixing both endpoints of the geodesic segment $\beta_i$, differentiable except at finitely many points and such that for some $K \geq 1$,
    \[
        \frac{1}{K}\leq \psi_i'(x)\leq K
    \]
    at every point $x\in\beta_i$ where $\psi_i'$ exists. Then there exists a homeomorphism \[\Psi:H \longrightarrow H\] such that\begin{enumerate}
        \item $\Psi|_{\beta_i}= \psi_i$ for $i=1,2,3$;
        \item $\Psi$ is the identity on the horocyclic sides $\h_1$, $\h_2$, $\h_3$ and outside the set of points at distance at most $\rho_0$ from $\beta_1 \cup \beta_2 \cup \beta_3$;
        \item $\Psi$ is differentiable outside a finite union of analytic arcs and satisfies at every differentiable point \[ ||D\Psi||\leq C_{\Psi}, \quad ||D\Psi^{-1}||\leq C_{\Psi}, \] where $C_{\Psi}=K\sqrt{1+K^2+\cosh^2(\rho_0)\frac{\delta_{\max}^2}{\rho_0^2}}$.
    \end{enumerate}
\end{lemma}
\begin{proof}
    Let $H$ be a truncated hexagon obtained by truncating an ideal triangle $T=(a_1,a_2,a_3)$ with vertices $v_1$, $v_2$, $v_3$. Fix $i\in \{1,2,3\}$ and lift $a_i$ to the imaginary axis in $\Hh^2$, parametrized by $z=(X,y)$, such that the lift of $T$ is contained in the right half of the upper half-plane $\left\{z\in \Hh^2 \middle| X \geq 0\right\}$. Consider the Fermi coordinates 
    \[
        \R \times \R_{>0} \longrightarrow \Hh^2, (x,\rho) \longmapsto e^x(\tanh(\rho), \sech(\rho))
    \] 
    which parametrize $\left\{z\in \Hh^2 \middle| X \geq 0\right\}$. The curve $\{\rho=0\}$ parametrizes $a_i$ by its arc length $x$, while the curve $\{x=x_0\}$ corresponds to the geodesic through $ie^{x_0}$ orthogonal to $a_i$. In these coordinates, the hyperbolic metric is 
    \[
        \frac{dX^2+dy^2}{y^2}=\cosh^2(\rho)dx^2+d\rho^2.
    \]
    We define 
    \[
        R_i = \left\{
                p \in \mathbb{H}^2
                \;\middle|\;
                d(p,a_i) \leq \rho_0,\; p\in T,\; \text{and the foot of }p\text{ on }a_i\text{ lies on }\beta_i
            \right\}.
    \]
    In Fermi coordinates $R_i=[x_i^-,x_i^+]\times[0,\rho_0]$ where $[x_i^-,x_i^+]$ is the arc length interval of $\beta_i$ on $a_i$. Consequently $x_i^{-}-x_i^{+}= \ell(\beta_i)\leq \delta_{\max}$.
    \begin{claim}
       For each $i\in \{1,2,3\}$: 
       \begin{enumerate}[label=(\alph*)]
            \item\label{claimA} $R_i \subset H$ and $R_i$ meets $\partial H \smallsetminus \beta_i$ only at the two endpoints of $\beta_i$.
            \item\label{claimB} If $p\in R_i \cap R_j$ with $i\neq j$ then $d(p,\beta_i)=d(p,\beta_j)=\rho_0$.
        \end{enumerate}
    \end{claim}
    \begin{proofclaim}
        To prove~\ref{claimA} we estimate the following distances: \begin{enumerate}[label=(\roman*)]
            \item \underline{Distance from $R_i$ to the horoballs $B_j$ and $B_l$.} Lift $a_i$ to the imaginary axis. Let $p=(R\sin(\alpha), R\cos(\alpha))$ with $R>0$ and $|\alpha|<\pi/2$. The geodesic orthogonal to $a_i$ corresponds to curve $\{R=c\}$ for some constant $c$. Hence the foot of $p$ on $a_i$ is the point $iR$. Let $i y_0=a_i \cap \partial B_j$ be the endpoint of $\beta_i$ on $\h_j$ and similarly let $i y_1= a_i \cap \partial B_l$ be the endpoint such that $\beta_i=\left\{i R \;\middle|\; y_0 \leq R \leq y_1\right\}$. Hence the horoball $B_j$ is given by
            \[
                B_j=\left\{(X,y) \;\middle|\; X^2+\left(y-\frac{y_0}{2}\right)^2 \leq \frac{y_0^2}{4} \right\} = \left\{ (R,\alpha) \;\middle|\; R \leq y_0\cos(\alpha) \right\}.
            \]
            If $p \in R_i$, then the foot of $p$ on $\beta_i$ is $iR$ which satisfies $y_0 \leq R \leq y_1$. Since $|\alpha|\leq \pi/2$ we have $R\geq y_0 \geq y_0\cos(\alpha)$. Hence $p\in R_i \cap B_j$ if and only if $\alpha=0$ and $R=y_0$. Thus, $R_i \cap B_j=\{iy_0\}$. Similarly, we obtain  $R_i \cap B_l=\{iy_1\}$.
            \item \underline{Distance from $R_i$ to the horoball $B_i$.} Lift the vertex $v_i$ to $\infty$, then $a_j$ and $a_l$ lift to two vertical lines say $x=0$, and $x=w$ for some constant $w>0$. Hence $a_i$ lifts to a half-circle of radius $w/2$ joining the feet of the two lines. On one hand, the horocycle segment of $\partial B_i$ between $a_j$ and $a_l$ has length $\ell(\h_i)$ and one has $\partial B_i = \{ z \;|\; \Im(z)= w/\ell(\h_i) \}$. On the other hand, every point of $a_i$ has height at most $w/2$. Thus we have: 
            \[
                d(a_i, B_i) \geq \log\left(\frac{w/\ell(\h_i)}{w/2}\right)=2\log\left(\frac{2}{\ell(\h_i)}\right) \geq \log(4).
            \]
            Since every point of $R_i$ lies at distance at most $\rho_0 < 1/4$ from $a_i$ we deduce that 
            \[
                d(R_i, B_i)\geq \log(4)-\rho_0 >0.
            \]
            \item\label{point3} \underline{Distance from $R_i$ to the horoballs $a_j$ and $a_l$.} Lift $v_l$ to $\infty$, $a_i$ to the axis $\{x=0\}$ and $a_j$ to the vertical line $\{x=w\}$ for some fixed $w>0$. Then as before, $\partial B_l = \{ z \;|\; \Im(z)= Y_0\}$ where $Y_0=w/\ell(\h_l)$. The side $\beta_i$ is given by $\beta_i=\{(0,y)\;|\;y_0'\leq y\leq Y_0\}$ for some constant $y_0'>0$. Every $p=(0,y)\in \beta_i$ satisfies 
            \begin{align*}
                d(p,a_j) &= \arcsinh\left(\frac{w}{y}\right) \\
                &\geq \arcsinh\left(\frac{w}{Y_0}\right)=\arcsinh(\ell(\h_l)) \\
                &\geq \arcsinh(h_{\min})=2\rho_0. 
            \end{align*}
            Consequently, $d(\beta_i, a_j) \geq 2\rho_0$ and so  $d(R_i, a_j) \geq \rho_0$. Similarly, we obtain $d(\beta_i, a_l) \geq 2\rho_0$ and thus $d(R_i, a_l) \geq \rho_0$. 
        \end{enumerate} 
        We conclude that $R_i$ is connected in $T$, disjoint from $a_j \cup a_l$. Moreover, $R_i \subset H$ and $R_i$ meets $\h_1\cup \h_2 \cup \h_3$ only at the two endpoints of $\beta_i$ and meet neiter $\beta_j$ nor $\beta_l$. 
        Thus Claim~\ref{claimA} follows.

        Let $p\in R_i \cap R_j$ for $i\neq j$, by the conclusion of step~\ref{point3} one has \[2\rho_0 \leq d(\beta_i, a_j) \leq  d(\beta_i, \beta_j) \leq d(\beta_i, p)+d(p, \beta_j).\] However, since the foot of $p$ lies on $\beta_i \subset a_i$ we also have $d(p,\beta_i)=d(p,a_i) \leq \rho_0$, and similarly, $d(p,\beta_j)=d(p,a_j) \leq \rho_0$. Therefore, $d(p,\beta_i)=d(p,\beta_j)=\rho_0$ which proves Claim~\ref{claimB}.
        \end{proofclaim}

        Using the Fermi coordinates $(x,\rho)$ we identify $R_i$ with $[x_i^-,x_i^+]\times[0,\rho_0]$ such that $\beta_i=\{\rho=0\}$ on $a_i$ and $\psi_i$ is an increasing homeomorphism of $[x_i^-,x_i^+]$ fixing the two endpoints. Define
        \[
            \Psi(x,\rho)=
                \begin{cases} 
                \left(  \left(1-\frac{\rho}{\rho_0}\right)\psi_i(x) + \frac{\rho}{\rho_0}x,\; \rho\right), & \text{on } R_i \text{ for } i=1,2,3, \\
                (x,\rho), & \text{on } H\smallsetminus (R_1\cup R_2 \cup R_3). \\
                \end{cases}
        \]
        This map is well-defined homeomorphism. It is the identity on $R_i \cap R_j$, on $\{\rho=\rho_0\}$ and on $\{x=x_i^\pm\}$ since $\psi(x_i^\pm)=x_i^\pm$. We compute its Jacobian:
        \[
            D\Psi=\left( \begin{matrix}
                            \left(1-\frac{\rho}{\rho_0}\right)\psi_i'(x)+\frac{\rho}{\rho_0} & \cosh(\rho)\frac{x-\psi_i(x)}{\rho_0}\\
                            0 & 1 \\
                    \end{matrix} \right)
    \]
    and estimate a lower bound for its determinant:
    \[  
        \det(D\Psi)=\left(1-\frac{\rho}{\rho_0}\right)\psi_i'(x)+\frac{\rho}{\rho_0} \geq \frac{1}{K}.
    \]
    The spectral norm is bounded by the Frobenius norm,
    \begin{align*}
        \|D\Psi\| &\leq \sqrt{1+\left(\left(1-\frac{\rho}{\rho_0}\right)\psi_i'(x)+\frac{\rho}{\rho_0}\right)^2+\left(\cosh(\rho)\frac{x-\psi_i(x)}{\rho_0}\right)^2}\\
        &\leq \sqrt{1+K^2+ \cosh^2(\rho_0)\frac{\delta_{\max}^2}{\rho_0^2}} \leq C_{\Psi}
    \end{align*}
    and since the spectral norm of the inverse is bounded by the spectral norm divided by the determinant we obtain:
    \begin{align*}
        \|D\Psi^{-1}\| &\leq K\sqrt{1+K^2+ \cosh^2(\rho_0)\frac{\delta_{\max}^2}{\rho_0^2}} \leq C_{\Psi}=C_{\Psi}
    \end{align*}
    Outside $R_1 \cup R_2 \cup R_3$ the map is the identity. Thus $\Psi$ is differentiable except along finitely many sides (the $R_i$ and the curves $\{x = c\}$ corresponding to the finitely many points at which $\psi_i$ is not differentiable) with $\|D\Psi^\pm\|\leq C_{\Psi}$. In particular it is bi-Lipschitz.
\end{proof}

\begin{proposition}\label{PropTeichShear}
    Let $X, Y$ be two punctured hyperbolic surfaces in $\teich_{g,n}$, with $2g-2+n>0$ and $n\geq 1$. Let $\mathcal{T}$ be an ideal triangulation of $X$ and $Y$ such that
    \[ 
        \shear_{\mathcal{T}}(Y)=0.
    \] 
    Further assume that the length of every subarc $\hat a$, obtained from an ideal arc $a\in \mathcal{T}$ by removing the portions contained in standard cusp neighborhoods, satisfies
    \[
        \ell_X(\hat a) \leq G(g,n).
    \]
    Then, there exists a bi-Lipschitz map $F : X \longrightarrow Y$ whose bi-Lipschitz constant depends only on $g,n$ and $G(g,n)$. Consequently, 
    \[ 
        \dteich(X, Y) \lesssim \log(g+n) + G(g,n). 
    \]
\end{proposition}

\begin{proof}
    Let $X,Y\in \teich_{g,n}$ satisfy the assumptions of the proposition. We construct a bi-Lipschitz map  $F:X \longrightarrow Y$ with controlled bi-Lipschitz constant as follows: after truncating the cusps, we first define the map locally on each truncated hexagon using Lemma~\ref{LemmaBiLipHex}. We then modify the local maps along each geodesic side, this costs a controlled correction by Lemma~\ref{correction_lemma}, to ensure that the local maps glue together. Finally we extend it to the whole surface. 

    Let $\hat{X}$ and $\hat{Y}$ be the truncated surfaces obtained from $X$ and $Y$ by removing the cusp neighborhoods bounded by the horocycles of length $1/2$. The ideal triangulation $\mathcal{T}=(a_i)_{i=1}^{6g-6+3n}$ of $X$ (resp. $Y$) induces a decomposition into truncated hexagons 
    \[ 
        \hat{X} = \cup_{k=1}^{4g-4+2n} H_k, \qquad  \hat{Y} = \cup_{k=1}^{4g-4+2n} H_k'.
    \]
    Fix a pair of corresponding hexagons $H=H_k$ and $H'=H_k'$ with boundaries labelled by
    \[
        \partial H= (\h_1, \beta_3, \h_2, \beta_1, \h_3, \beta_2), \]
    and
    \[
        \partial H'= (\h_1', \beta_3', \h_2', \beta_1', \h_3', \beta_2'),
    \] 
    and set $\delta_i\coloneq\ell_{\hat{X}}(\beta_i)$, $h_i\coloneq\ell_{\hat{X}}(\h_i)$ and $h_i'\coloneq\ell_{\hat{Y}}(\h_i')$. Since each geodesic side $\beta_i$ is obtained from the truncated arc $\hat{a_i}$ (truncated along standard cup neighborhoods) by extending it to the horocycles of length $1/2$ we obtain
    \[ 
        4 \log(2) \leq \delta_i \leq G_1(g,n), \quad G_1(g,n)\coloneq G(g,n) + 2 \log(4)
    \] 
    Let $ \lambda_i=\sqrt{e^{\delta_i}}$ denote the corresponding lambda-length such that 
    \[ 
        4 \leq \lambda_i \leq e^{G_1(g,n)/2}.
    \]
    Each  horocyclic segment $\h_i$ is a subarc of a horocycle of length $\frac{1}{2}$, hence $ h_i,\; h_i' \leq \frac{1}{2}$. By the relation~(\ref{h_lambda_relation}) between $h$-lengths and $\lambda$-length since $\lambda_i \geq 4$ we obtain
    \[ 
        h_i=\frac{\lambda_i}{\lambda_j \lambda_l} \geq \frac{4}{e^{G_1(g,n)}} .
    \] 
    On $\hat{Y}$, since $\shear_{\mathcal{T}}(Y)=0$, Lemma~\ref{Shear0-HorocyclicSegments}, with $L=1/2$ gives:
    \[
        h_i' = \frac{1}{2v(c_i)}
    \] 
    where $v(c_i)$ is the number of ends of arcs of $\mathcal{T}$ incident to $c_i$. Since we have $6g-6+3n$ arcs, one has $v(c_i)\leq 12g-12+6n$ and therefore $k_i= \frac{h_i'}{h_i}$ satisfies 
    \[ 
        h_{\min}' \coloneq \frac{1}{12(2g-2+n)} \leq h_i' \leq \frac{1}{2}.
    \]
    Set 
    \[
        h_{\min}\coloneq \min \{4e^{-G_1}, h_{\min}'\}, \quad b_{\max}\coloneq G_1(g,n).
    \]
    
    Every pair $(H_k,H_k')$ satisfies the hypotheses of Lemma~\ref{LemmaBiLipHex} with these values of $h_{\min}$ and $b_{\max}$. Let
    \[
        F_k\colon H_k\longrightarrow H_k'
    \]
    be the resulting maps, which satisfy $\|DF_k^{\pm1}\|\leq C_F$ at every point of differentiability, where $C_F=C_F(h_{\min},b_{\max})$ is the constant of Lemma~\ref{LemmaBiLipHex}. In the notation of that lemma we record, for later use, that the factors $k_i=h_i'/h_i$ satisfy $k_i\in[1/M,M]$ with
    \[
        M \coloneq \frac{1}{2h_{\min}} = \max\left\{6(2g-2+n),\;\frac{e^{G_1}}{8}\right\}.
    \]
    Moreover, the stretch factors $\kappa_i=\log(1/h_i')/\log(1/h_i)$ of the boundary maps satisfy $\kappa_i\in[1/C,C]$ with
    \[
        C\coloneq\frac{\log(1/h_{\min})}{\log 2},
    \]
    and that $M\leq C_{\partial}\leq C_F$.

    Let $H,\overline{H}$ be adjacent hexagons sharing a geodesic side $\beta=\hat a\subset\hat X$. Let $H',\overline{H}'$ be the corresponding hexagons sharing the common side $\beta'=\hat a'\subset\hat Y$ and let $F_H, F_{\overline{H}}$ be the corresponding maps provided by Lemma~\ref{LemmaBiLipHex}. The maps $F_H$ and $F_{\overline{H}}$ agree at the corners of the hexagons and are compatible along the horocyclic sides. Indeed, both maps send each endpoint of $\beta$ to an endpoint of $\beta'$. Moreover, around each cusp $c$ of $X$, the horocyclic sides of the incident hexagons partition the horocycle and the maps send this partition to the corresponding partition on $\hat{Y}$, respecting the cyclic order and matching at the shared corners. Hence the maps $F_k$ fit together into a well-defined piecewise-linear homeomorphism between the two horocycles at $c$, with the same total length $1/2$. The two maps $F_H$ and $F_{\overline{H}}$ send $\beta$ onto the full side $\beta'$. However, the restrictions of $F_H$ and $F_{\overline H}$ to the interior of $\beta$ do not agree in general, when the shear along $a$ is non zero. By construction, $F_H\restr{\beta}$ is piecewise linear with a single breakpoint at the point $m$ of $\beta$ located at distance $\log(1/h_j)$ (resp. $\log(1/h_l)$) from the corner $\beta\cap\h_j$, (resp. $\beta\cap\h_l$), see relation~(\ref{delta-truncation}). Likewise, $F_{\overline{H}}\restr{\beta}$ splits $\beta$ at the tangency point $\overline{m}$ of the other adjacent triangle. The distance between $m$ and $\overline{m}$ is exactly the shear parameter along $a$ which can be non-zero and so the piecewise stretch factors differ. Therefore, we correct the maps $F_H$, using Lemma~\ref{correction_lemma}.

    For each arc $a\in\mathcal{T}$, let $\beta=\hat a\subset\hat X$ and $\beta'=\hat a'\subset\hat Y$ be the corresponding truncated arcs (extended to the length-$1/2$ horocycles) and let $\delta_a,\delta_a'$ denote their lengths. Each endpoint of $\beta$ is the intersection of $a$ with a horocycle of $\hat X$ and corresponds to an endpoint of $\beta'$. Define
    \[
        \Phi_a\colon\beta\longrightarrow\beta'
    \]
    to be the unique homeomorphism of constant speed $\delta_a'/\delta_a$ with respect to the arc-length parameter that maps $\beta$ onto $\beta'$, matching the corresponding endpoints. By computing $\delta_a$ and $\delta_a'$ using relation~(\ref{delta-truncation}) in either adjacent hexagon we obtain
    \[
        \frac{\delta_a'}{\delta_a}
        =\frac{\log(1/h_j')+\log(1/h_l')}{\log(1/h_j)+\log(1/h_l)}
        \;\in\;\bigl[\min(\kappa_j,\kappa_l),\,\max(\kappa_j,\kappa_l)\bigr]\subset\left[\frac1C,\,C\right].
    \]
    Fix $k$ and write $H=H_k$, $H'=H_k'$. For each geodesic side $\beta_i$ of $H$, lying on the arc $a_i$, define
    \[
        \psi_i\coloneq\Phi_{a_i}\circ\bigl(F_k\restr{\beta_i}\bigr)^{-1}\colon\;\beta_i'\longrightarrow\beta_i'.
    \]
    By the discussion above, $\psi_i$ is a homeomorphism of $\beta_i'$ fixing both endpoints. The map $F_k\restr{\beta_i}$ is piecewise linear with slopes $\kappa_j,\kappa_l\in[1/C,C]$ (proof of Lemma~\ref{LemmaBiLipBoundaryHex}) and $\Phi_{a_i}$ has constant slope $\delta_{a_i}'/\delta_{a_i}\in[\min(\kappa_j,\kappa_l),\max(\kappa_j,\kappa_l)]$. Hence $\psi_i$ is piecewise linear with slopes in $[1/K,K]$, where
    \[
        K\coloneq C^2 .
    \]
    Apply Lemma~\ref{correction_lemma} to the truncated hexagon $H'$, with the lower bound $h_{\min}'$ for its horocyclic sides, the reparametrizations $\psi_1,\psi_2,\psi_3$, and
    \[
        \varrho_0=\frac{1}{2}\arcsinh(h_{\min}'), \qquad
        \delta_{\max}=\max_i\ell(\beta_i')\leq 2\log\frac{1}{h_{\min}'}=2\log\bigl(12(2g-2+n)\bigr),
    \]
    where the last bound is obtained from relation~(\ref{delta-truncation}) and $h_i'\geq h_{\min}'$. We obtain a homeomorphism $\Psi_k\colon H_k'\to H_k'$ with $\|D\Psi_k^{\pm1}\|\leq C_\Psi$, equal to the identity on the horocyclic sides, and restricting to $\psi_i$ on each $\beta_i'$. Define
    \[
        \widetilde F_k\coloneq\Psi_k\circ F_k\colon\;H_k\longrightarrow H_k' .
    \]
    Then, at every point of differentiability,
    \[
        \bigl\|D\widetilde F_k^{\pm1}\bigr\|\leq C_\Psi\, C_F.
    \]
    The restriction of $\widetilde F_k$ to each geodesic side $\beta_i$ equals $\psi_i\circ F_k\restr{\beta_i}=\Phi_{a_i}$, and the restriction of $\widetilde F_k$ to each horocyclic side is unchanged: it is the linear map of factor $k_i$. Define
    \[
        F\colon\hat{X}\longrightarrow\hat{Y}, \qquad F\restr{H_k}=\widetilde F_k .
    \]
    This is well defined: if two hexagons $H_k$ and $H_{k'}$ (possibly $k=k'$, when the two sides of an arc bound the same triangle) are adjacent along $\beta=\hat a$, both restrictions to $\beta$ equal $\Phi_a$. Together with the corner and horocycle compatibilities, $F$ is a well-defined homeomorphism, since each $\widetilde F_k$ is a homeomorphism respecting the boundary identifications.
    
    In the upper half-plane, a cusp neighborhood bounded by a horocycle of length $1/2$ is isometric to
    \[
        \left\{ z\in \Hh \;\middle|\; \ima(z) \geq 2 \right\}/(z \mapsto z + 1),
    \]
    equipped with the metric $ds^2=\frac{dx^2+dy^2}{y^2}$. The boundary horocycle $\{\ima(z)=2\}$ has length $1/2$. Around each cusp of $\hat X$, the map $F$ restricts on the boundary horocycle to a piecewise-linear homeomorphism onto the corresponding horocycle of $\hat Y$, with slopes among the factors $k_i\in[1/M,M]$. In these coordinates, we extend $F$ to the cusp neighborhood by defining
    \[
        \widetilde{F}(x,y)\coloneq(F(x),y).
    \]
    Its Jacobian is $D\widetilde F=\bigl(\begin{smallmatrix}\partial_x F&0\\0&1\end{smallmatrix}\bigr)$. Performing this construction at every cusp yields a homeomorphism
    \[
        \widetilde F\colon X\longrightarrow Y
        \qquad\text{with}\qquad
        \bigl\|D\widetilde F^{\pm1}\bigr\|\leq C_\Psi C_F \text{ a.e.}
    \]

    Let $p,q\in X$ and let $\gamma$ be a minimizing geodesic joining them. The segment $\gamma$ meets the union of the arcs of $\mathcal{T}$ and of the $n$ horocycles of length $1/2$ in finitely many points and subsegments:
    \[
        \gamma=\gamma_1\cup\dots\cup\gamma_m,
    \]
    each contained in a single closed hexagon or closed cusp neighborhood. Applying the local estimates on each piece yields
    \[
        \ell_{Y}(\widetilde F\circ\gamma)=\sum_{i=1}^m\ell_{Y}(\widetilde F\circ\gamma_i)\leq C_\Psi C_F \sum_{i=1}^m \ell_{X}(\gamma_i) = C_\Psi C_F\, \ell_{X}(\gamma),
    \]
    and so $d_Y(\widetilde F(p),\widetilde F(q))\leq C_\Psi C_F\, d_X(p,q)$. The same argument applied to $\widetilde F^{-1}$ gives the lower bound. Thus $\widetilde F$ is $C_\Psi C_F$-bi-Lipschitz.

    Since we work in dimension two, a $C$-bi-Lipschitz map is also a $C^2$-quasiconformal map, hence
    \[
        \dteich(X,Y)\leq\frac{1}{2}\log\bigl((C_\Psi C_F)^2\bigr)=\log C_\Psi+\log C_F .
    \]
    It remains to estimate $C_\Psi$. Since $h_{\min}'\leq\frac{1}{12}$ and $\arcsinh$ is concave with $\arcsinh(0)=0$, we have $\arcsinh(h_{\min}')\geq 12\arcsinh(\tfrac{1}{12})\,h_{\min}'\geq 0.998\,h_{\min}'$, so
    \[
        \varrho_0\geq 0.499\, h_{\min}', \qquad \cosh(\varrho_0)\leq\cosh\Bigl(\tfrac12\arcsinh\bigl(\tfrac{1}{12}\bigr)\Bigr)<1.001 ,
    \]
    and therefore
    \[
        \cosh(\varrho_0)\,\frac{\delta_{\max}}{\varrho_0}
        \leq\frac{1.001}{0.499}\cdot 12(2g-2+n)\cdot 2\log\bigl(12(2g-2+n)\bigr)
        \leq 49\,(2g-2+n)\log\bigl(12(2g-2+n)\bigr) .
    \]
    Using $\sqrt{1+u^2+v^2}\leq 1+u+v$ for $u,v\geq0$ and $K=C^2$,
    \[
        C_\Psi\leq C^2\Bigl(1+C^2+49\,(2g-2+n)\log\bigl(12(2g-2+n)\bigr)\Bigr) .
    \]
    We conclude that
    \[
        \dteich(X,Y)\;\leq\;\log C_F \;+\; 2\log C \;+\; \log\Bigl(1+C^2+49\,(2g-2+n)\log\bigl(12(2g-2+n)\bigr)\Bigr) .
    \]
    Finally, $\log(1/h_{\min})=\max\bigl\{G_1-\log 4,\;\log\bigl(12(2g-2+n)\bigr)\bigr\}$, so that
    \[
        C\leq\frac{G_1(g,n)+\log\bigl(12(2g-2+n)\bigr)}{\log 2},
    \]
    and $C_F$ is the explicit function of $h_{\min}$ and $b_{\max}=G_1(g,n)$ given by Lemma~\ref{LemmaBiLipHex}. In particular, from direct computations we obtain
    \[
        \dteich(X, Y) \lesssim \log(g+n) + G(g,n) .
    \]
\end{proof}

\begin{remark}
    Let $\Sigma$ be a hyperbolic surface of type $(g,n)$, note that Proposition~\ref{PropTeichShear} holds with $\teich(\Sigma)$ replaced by a stratum $S(\Gamma)$ and $\mathcal{T}$ by an ideal triangulation of $\Sigma \smallsetminus \Gamma$.
\end{remark}

\section{Estimates on the Weil--Petersson distance from hexagon decompositions}\label{sec:wpgraft}

Given a hexagon decomposition $(\Gamma, \mathcal{A})$, define
\[
    V(\Gamma, \mathcal{A})=
    \left\{ X \in \overline{\teich}_{g,n} \; \middle| \; \max_{\gamma \in \Gamma}\ell_X(\gamma) \leq L, \quad \max_{a\in \mathcal{A}}  \ell_{X, \Gamma}^{\wtrunc}(a) \leq L_A \right\}.
\] 
    
By Corollary~\ref{ShortHexagonDecompositionExistence} the union of these regions over all hexagon decompositions covers $\overline{\teich}_{g,n}$. 
Let $\mathcal{A}_\infty$ be the ideal triangulation obtained from $(\Gamma, \mathcal{A})$ by pinching every curve $\gamma \in \Gamma$ to a node. Equivalently, $\mathcal{A}_\infty$ is the spiralling triangulation $\mathcal{T}_{(\Gamma, \mathcal{A})}$ on the stratum $S(\Gamma)$. Since shearing coordinates with respect to $\mathcal{A}_\infty$ give a global parametrization of $S(\Gamma)$, any admissible choice of shear parameters determines a unique noded surface in this stratum.

In this section, we develop a tool for projecting a surface lying in $V(\Gamma, \mathcal{A})$ to the stratum $S(\Gamma)$. The main idea is to use \emph{grafting} along the curves of $\Gamma$. As the grafting parameter tends to infinity, the curves of $\Gamma$ are pinched and the surface in $V(\Gamma, \mathcal{A})$ converges to a noded surface in the stratum $S(\Gamma)$. Our goal is to control the Weil--Petersson distance along this grafting ray and the shearing coordinates of the limit noded surface with respect to the ideal triangulation $\mathcal{A}_\infty$.

\subsection{Weil--Petersson distance along a grafting ray}\label{sec:graftray}
We briefly introduce some background on complex projective structure and grafting. Let $\Sigma$ be a surface with $\chi(\Sigma)<0$. A \emph{complex projective structure} on $\Sigma$ is a maximal atlas of charts with values in $\CP^1$ whose transition maps are restrictions of Möbius transformations. The space of marked complex projective structures on $\Sigma$ up to isotopy is denoted by $\mathcal{CP}(\Sigma)$. 
Passing to the universal cover $\widetilde{\Sigma}$, a projective structure $Z\in \mathcal{CP}(\Sigma)$ determines \emph{a developing map} 
\[
    f:\widetilde{\Sigma} \longrightarrow \CP^1
\] 
and \emph{a holonomy representation} 
\[
    r:\pi_1(\Sigma) \longrightarrow \mathrm{PSL}(2,\C)
\] 
satisfying \[f(\gamma \cdot x)=r(\gamma)f(x)\] for every $x \in \widetilde{\Sigma}$ and $\gamma \in \pi_1(\Sigma)$. The pair $(f,r)$ is uniquely determined up to conjugation by elements of $\PSL(2,\C)$.

Since Möbius transformations are biholomorphic, every complex projective structure $Z\in \mathcal{CP}(\Sigma)$ determines an underlying Riemann surface structure. This gives the forgetful projection 
\[
    \pi: \mathcal{CP}(\Sigma) \longrightarrow \teich(\Sigma).
\]
For $X \in \teich(\Sigma)$, the difference between two projective structures in $\pi^{-1}(X)$ is measured by the \emph{Schwarzian derivative}. 
Let $f:\Delta \rightarrow \CP^1$ be a locally univalent complex function on an open domain $\Delta \subseteq \C$, its Schwarzian derivative is defined by: 
\[
    \Schw(f)=\left(\frac{f''}{f'}\right)'-\frac{1}{2}\left(\frac{f''}{f'}\right)^2
\]
In particular, $\Schw(f)=0$ if and only if $f$ is a Möbius transformation.

Let $X \in \overline{\teich}(\Sigma)$ and let $\gamma \subset X$ be a simple closed geodesic. We denote by $\Grr_{t\cdot \gamma}(X)$ the projective structure obtained by cutting $X$ along $\gamma$ and inserting, without twisting, a Euclidean cylinder of height $t$ and circumference equal to $\ell_X(\gamma)$. This construction is called \emph{grafting} along $\gamma$.
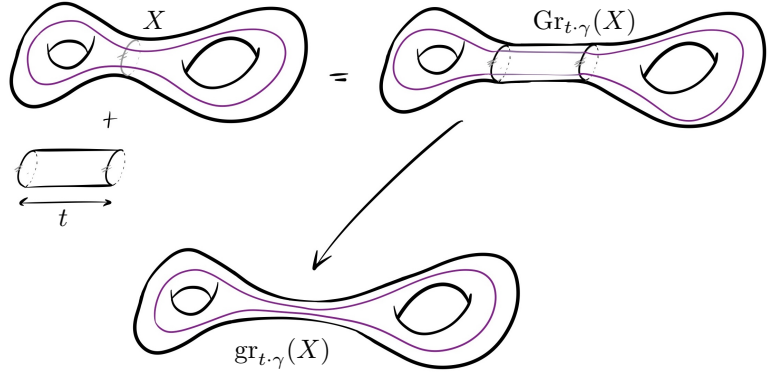
\begin{figure}
    \centering
    \begin{overpic}[width=0.8\linewidth,keepaspectratio]{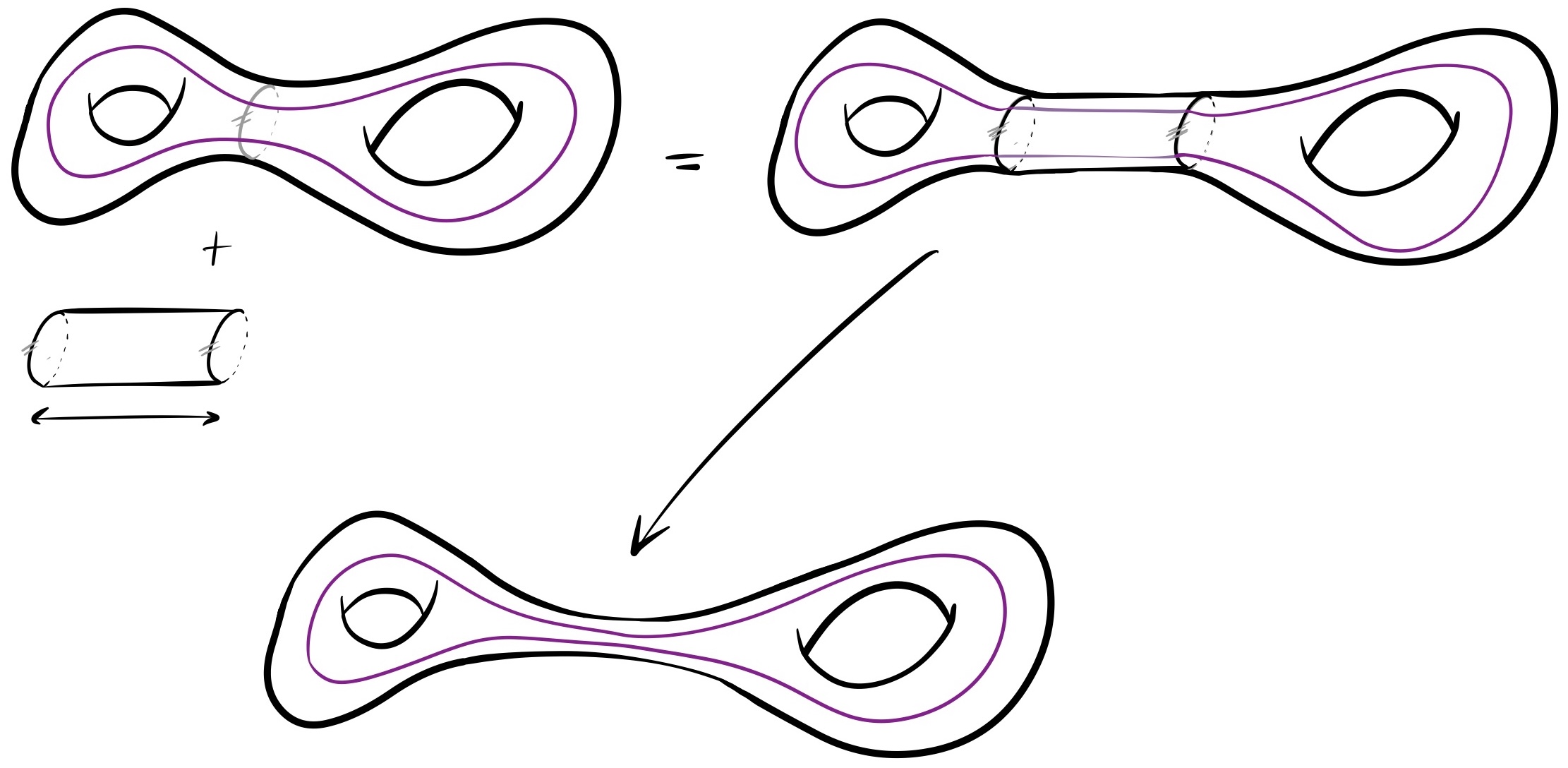}
    \put(18,45){$X$}
    \put(7,19.3){$t$}
    \put(30,2){$\gr_{t\cdot\gamma}(X)$}
    \put(69,45){$\Grr_{t\cdot\gamma}(X)$}
    \end{overpic}
    \vspace{-0.2cm}
    \caption{Grafting operation along a simple closed curve.}\label{Grafting}
\end{figure}
The Euclidean metric on the grafting cylinder and the hyperbolic metric on the hyperbolic parts glue together to define the Thurston metric $\rho_{\Th}$ on $\Grr_{t\cdot \gamma}(X)$. By uniformization, there exists a unique hyperbolic metric in the conformal class of $\rho_{\Th}$. We denote the resulting hyperbolic surface by \[\gr_{t\cdot \gamma}(X).\] The construction extends to simple multicurves. 

The following two lemmata, obtained by Diaz and Kim~\cite{DiazKim2012}, are length estimates on grafted surfaces that will be used throughout the next sections. 
\begin{lemma}[Proposition 3.4~\cite{DiazKim2012}]\label{DiazKimLengthGr}
    Let $X \in \teich_{g,n}$ and let $\Gamma = (\gamma_1, \ldots, \gamma_k)$ be a collection of pairwise disjoint simple closed curves on $X$. Let ${(X_t)}_{t\geq 0}$ denote the grafting ray defined by
    \[
        X_t=\gr_{t\Gamma}(X)
    \] 
    and starting at $X_0 = X$. Then, for all $i=1,\ldots,k$,
    \[
    \frac{2\theta(\ell_i)}{2\theta(\ell_i)+t} \cdot \ell_i \leq \ell_{X_t}(\gamma_i) \leq \frac{\pi}{\pi + t}\cdot \ell_i,
    \]
    where $\ell_i \coloneq \ell_X(\gamma_i)$ and $\theta(\ell)\coloneq\arctan\left(\frac{e^{2\wstd(\ell)}-1}{2e^{\wstd(\ell)}}\right)$
\end{lemma}

\begin{lemma}[Proof of Proposition 3.5~\cite{DiazKim2012}]\label{DiazKimLengthGr2}
    Let $X \in \teich_{g,n}$ and let $\gamma$ be a simple closed curve on $X$. Let ${(X_t)}_{t\geq0}$ denote the grafting ray defined by
    \[
        X_t=\gr_{t\gamma}(X)
    \] 
    and starting at $X_0 = X$. Let $L_{t}$ be a horizontal geodesic arc on the euclidean grafting cylinder. Decompose $\eta_t$ into three subarcs 
    \[
        \eta_t= \eta_t^1 \cup \eta_t^2 \cup \eta_t^3
    \] so that $\eta_t^1$ and $\eta_t^3$ both have euclidean length $\tau<t$, where $\tau$ is any fixed positive number. Then, 
    \[\ell_{X_{t}}(\eta_{t})\leq 2\tau + 2\log\left(\frac{\cos\left(\frac{\tau \pi}{2 t}\right)}{\sin\left(\frac{\tau \pi}{2 t}\right)}\right).\]
\end{lemma}

A \emph{grafting ray} is the path 
\[
    t \longmapsto \gr_{t\Gamma}(X)
\] in Teichmüller space. Any two points of Teichmüller space are connected by a grafting ray~\cite{DumasWolf2008}. Moreover, Hensel proved that a grafting ray converges geometrically to a noded surface obtained by pinching the curves of $\Gamma$ to nodes~\cite{hensel2008}. The resulting geometric limit is called the \emph{end of a grafting ray} and is denoted by: 
\[
    \gr_{\infty \Gamma}(X)=\lim_{t \to \infty} \gr_{t\Gamma}(X).
\] 
We now prove a distance estimate.

\begin{lemma}\label{WPgraftingPath}
    Let $g,n$ be integers such that $2g-2+n >0$, let $X \in \teich_{g,n}$ and let $\Gamma = (\gamma_1, \ldots, \gamma_k)$ be a collection of disjoint simple closed curves on $X$. Let $(X_t)_{t\geq 0}$ denote the grafting ray starting at $X_0 = X$ defined by $X_t=\gr_{t\Gamma}(X)$, and set $\ell_i \coloneq \ell_X(\gamma_i)$. Then for every $\tau>0$,
    \[
        \dwp\left(X_{\tau}, X_{\infty}\right) \;\leq\; \sqrt{2}\,\pi\, \sqrt{\sum_{i=1}^k \ell_i}\tau^{-1/2},
    \]
    where $X_\infty \coloneq \gr_{\infty \Gamma}(X)$.
\end{lemma}

\begin{proof} 
    The proof follows very closely the strategy of Wolpert to estimate the WP length of rays determined by a certain class of quadratic differentials~\cite{Wolpert1975}, and the work of Dumas and Wolf to compute the Beltrami differentials associated with a grafting ray~\cite{DumasWolf2008}. 
    
   For each $t>0$, grafting along $\Gamma$ on $X$ inserts, for every curve $\gamma_i \in \Gamma$, a Euclidean cylinder of height $t$ and circumference $\ell_i$ where $\ell_i \coloneq\ell_X(\gamma_i)$. This cylinder corresponds to an annulus 
    \[
        A_{t,i}=\left\{z_{t,i} \in \C \;|\; 1 \leq |z_{t,i}|\leq e^{\frac{2\pi t}{\ell_i}} \right\}
    \] 
    whose modulus is $\mod(A_{t,i})=t/\ell_i$. 
    Passing to normal coordinates 
    \[
        \log(z_{t,i})=\zeta_{t,i},
    \] 
    we identify the annulus $A_{t,i}$ with the rectangle 
    \[
        R_{t,i}=\left\{\zeta_{t,i}=u+iv\in \Hh\;\middle|\; 0 \leq u \leq \frac{2\pi t}{\ell_i},\; 0 \leq v \leq 2\pi \right\}
    \] 
    obtained as a quotient of the vertical strip $0\leq \Re(\zeta_{t,i})\leq \frac{2\pi t}{\ell_i}$ by the translation $\zeta \mapsto \zeta + 2\pi i$. The Poincaré metric on this strip is 
    \[
        \lambda_{t,i}~=~\frac{\ell_{i}}{2t}\csc\left(\frac{\ell_{i}}{2t} \Re(\zeta_{t,i})\right)|d\zeta_{t,i}|.
    \]
    
    For each $i=1,\ldots,k$ and $s>0$, let $f_{i}:R_{t,i} \longrightarrow R_{s,i}$  be the quasiconformal map between two grafting cylinders given by 
    \[
        \zeta_{t,i} \mapsto \zeta_{s,i} =\zeta_{t,i} + \left(\frac{s}{t}-1\right)\frac{\left(\overline{\zeta_{t,i}} + \zeta_{t,i}\right)}{2}.
    \] 
    The maps $f_{i}$ extended by the identity define a quasi-conformal map $f$ between $X_t$ and $X_s$ with Beltrami coefficient 
    \[ 
        \frac{\partial_{\overline \zeta_{t}}f}{\partial_{\zeta_t}f}=\frac{s-t}{t+s}
    \] 
    on the cylinders and $0$ elsewhere. Taking the derivative with respect to $s$ and setting $s=t$ gives $1/2t$. Thus the tangent to the grafting ray at $X_t$ is the Beltrami differential 
    \[
        \mu_f(t)=\frac{1}{2t}\frac{\overline{d\zeta_t}}{d\zeta_t},
    \] 
    where $\frac{\overline{d\zeta_t}}{d\zeta_t}=\frac{\overline{d\zeta_{t,i}}}{d\zeta_{t,i}}$ on the grafting cylinders and zero outside.
    
    Now we estimate the Weil--Petersson norm of $\mu_f(t)$. By definition,
    \begin{align}\label{quotientdefwp}
        \left\|\mu_f(t) \right\|_{\WP} =\sup_{\phi \neq 0} \frac{|\langle \mu_f(t), \phi \rangle|}{ \left\|\phi \right\| }
    \end{align}
    where the supremum is taken over all holomorphic quadratic differentials on $X_t$. To estimate this quotient, observe that the annuli $A_{t,i}$ are pairwise disjoint on  $X_t$ and by the Schwarz--Pick lemma applied to the inclusions $A_{t,i} \hookrightarrow X_t$, the pullback of the hyperbolic metric $\rho$ from $X_t$ restricted to $A_{t,i}$ is bounded by the hyperbolic metric $\lambda_{t,i}$ on the annulus and so
    \begin{align}\label{eq0}
        \left\|\phi \right\|^2 = \int_{X_t}|\phi|^2 \rho^{-2} \geq \sum_{i=1}^{k} \int_{R_{t,i}}|\phi|^2\lambda_{t,i}^{-2}.
    \end{align}
    On $R_{t,i}$ write $\phi=\phi_i(\zeta_{t,i})\,(d\zeta_{t,i})^2$; the function $\phi_i$ is holomorphic and $2\pi i$-periodic, hence admits a Laurent expansion 
    \[
        \phi_i(\zeta_{t,i})=\sum_{n\in\Z} a_{n,i}\, z_{t,i}^{n}=\sum_{n\in\Z} a_{n,i}\, e^{n\zeta_{t,i}}.
    \]
    Moreover the function $\phi_i$ satisfies the following equations:
   \begin{equation}\label{eq1}
        \int_0^{2\pi} \phi_i(u+iv)\, dv =\sum_{n\in\Z} a_{n,i} e^{nu} \int_0^{2\pi} e^{inv}\, dv = 2\pi\, a_{0,i},
    \end{equation}
    and 
    \begin{align}
        \int_0^{2\pi} |\phi_i(u+iv)|^2\, dv &= \sum_{n\in\Z}\sum_{m\in\Z} a_{n,i}\overline{a_{m,i}}e^{(n+m)u} \int_0^{2\pi} e^{i(n-m)v}\, dv \notag \\ 
        &=\sum_{n\in\Z} a_{n,i}\overline{a_{n,i}}e^{2nu}2\pi \notag =2\pi \sum_{n\in\Z} |a_{n,i}|^2 e^{2nu} \notag \\
        &\geq 2\pi\, |a_{0,i}|^2. \label{eq2}
    \end{align}
    Then we obtain a lower bound for the denominator in~(\ref{quotientdefwp}):
    \begin{align}
        \left\|\phi \right\|^2 &\underset{\text{by}~(\ref{eq0})}{\geq} \sum_{i=1}^{k} \int_{R_{t,i}}|\phi|^2 \lambda_{t,i}^{-2} 
        = \sum_{i=1}^{k} \int_{0}^{2\pi t/ \ell_i} \int_{0}^{2\pi} |\phi(u+iv)|^2  \lambda_{t,i}^{-2} \;dv\,du \notag \\
        &= \sum_{i=1}^{k} \int_{0}^{2\pi t/ \ell_i} \lambda_{t,i}^{-2}  \left(\int_{0}^{2\pi}|\phi_i(u+iv)|^2 \,dv\right)\,du \notag\\
        &\underset{\text{by}~(\ref{eq2})}{\geq} \sum_{i=1}^{k} \int_{0}^{2\pi t/ \ell_i} \lambda_{t,i}^{-2}2\pi\,|a_{0,i}|^2 \,du =\sum_{i=1}^{k} 2\pi\,|a_{0,i}|^2\,  \frac{4t^2}{\ell_i^2} \int_{0}^{\frac{2\pi t}{\ell_i}} \sin^2\left(\frac{\ell_i}{2t}u\right) du \notag \\
        &= \sum_{i=1}^{k} \frac{8\pi^2 t^3}{\ell_i^3}\,|a_{0,i}|^2\label{eq5},
    \end{align}
    the second step holds because $\lambda_{t,i}$ does not depend on the imaginary part of $\zeta_{t,i}$. Finally for the numerator, we compute
    \begin{align}
        \left\langle \mu_f(t), \phi \right\rangle
        &=\sum_{i=1}^{k}\int_{R_{t,i}}  \frac{1}{2t}\phi_i \;dv\,du
        =\sum_{i=1}^{k} \frac{1}{2t}\int_{0}^{2\pi t/ \ell_i} \int_{0}^{2\pi}\phi_i(u+iv) \;dv\,du \notag \\
        &\underset{\text{by}~(\ref{eq1})}{=}\sum_{i=1}^{k} \frac{1}{2t} \int_{0}^{2\pi t/ \ell_i} (2\pi a_{0,i}) \,du = 2\pi^2\sum_{i=1}^{k} \frac{a_{0,i}}{\ell_i},\label{eq3},
    \end{align}
    therefore
    \begin{align}
        |\left\langle \mu_f(t), \phi \right\rangle |
        &\underset{\text{by}~(\ref{eq3})}{\leq} \sum_{i=1}^{k} \frac{2\pi^2 |a_{0,i}|}{\ell_i}=\sum_{i=1}^{k} \sqrt{\frac{4\pi^4 \ell_i|a_{0,i}|^2}{\ell_i^3}\frac{t^3}{t^3}} \notag\\
        &\leq \sqrt{\sum_{i=1}^{k} \frac{8\pi^2t^3}{\ell_i^3}|a_{0,i}|^2}\sqrt{\sum_{i=1}^{k} \frac{\pi^2\ell_i}{2t^3}}\notag\\
        &\underset{\text{by}~(\ref{eq5})}{\leq}\left\|\phi \right\| \frac{\pi}{\sqrt{2}}\sqrt{\sum_{i=1}^{k}\ell_i}\;t^{-3/2}. \label{eq4}
    \end{align}
    Consequently, for $0<\tau \leq s,$
    \begin{align}
       \dwp(X_{\tau}, X_s) &\leq \int_{\tau}^s \left\|\mu_f(t)\right\|_{\WP} \, dt
        = \int_{\tau}^s \sup_{\phi} \frac{|\langle \mu_f(t), \phi \rangle|}{ \left\|\phi \right\|}  \, dt \notag\\
        &\underset{\text{by}~(\ref{eq4})}{\leq}\int_{\tau}^s \sup_{\phi} \frac{ \left\|\phi \right\| \frac{\pi}{\sqrt{2}}\sqrt{\sum_{i=1}^{k}\ell_i} t^{-3/2}}{\left\|\phi \right\|} \, dt \notag \\
        &\leq \sqrt{2}\,\pi\, \sqrt{\sum_{i=1}^k \ell_i}\tau^{-1/2}. \notag
    \end{align}
    The grafting ray $(X_t)_{t\geq0}$ converges to a noded surface $X_{\infty} \in S(\Gamma)$ when $t\rightarrow \infty$~\cite{hensel2008} hence letting $s\rightarrow \infty$ gives the distance estimate by continuity of the grafting map~\cite{Tanigawa1997, McMullen1998, ScannelWolf2002}.
\end{proof}

Next we give an estimate on the Weil--Petersson distance between the starting point of a grafting ray and any point along that ray.

\begin{lemma}\label{BoundStartGraftingRay}
    Let  $g,n$ be integers such that $2g-2+n >0$, let $X \in \teich_{g,n}$ and let $\Gamma = (\gamma_1, \ldots, \gamma_k)$ be a collection of disjoint simple closed curves on $X$. Let $(X_t)_{t\geq 0}$ denote the grafting ray starting at $X_0 = X$ defined by $X_t=\gr_{t\Gamma}(X)$, and set $\ell_i \coloneq \ell_X(\gamma_i)$. Then for every $\tau\geq 0$,
    \[
        \dwp(X,X_{\tau})\leq \sqrt{\tau\sum_{i=1}^k \ell_i\,}.
    \]
\end{lemma}

\begin{proof}
    Let $t>0$, by the proof of Lemma~\ref{WPgraftingPath}, the tangent to the grafting ray at $X_t$ is the Beltrami differential 
    \[
        \mu_f(t)=\hat{\mu}\frac{\overline{d\zeta_t}}{d\zeta_t},\qquad \text{ where } \hat{\mu}=\frac{1}{2t},
    \] 
    and with $\frac{\overline{d\zeta_t}}{d\zeta_t}=\frac{\overline{d\zeta_{t,i}}}{d\zeta_{t,i}}$ on the grafting cylinders $C_{t,i}$ and zero outside. The Weil--Petersson norm is bounded by the $L^2$- norm~\cite[Section 2]{BBB2023}
    \begin{equation}\label{WPbelowL2}
        \left\|\mu_f(t)\right\|_{\WP} \leq \left\|\mu_f(t)\right\|_{L^2}
        \coloneq \left(\int_{X_t} |\hat{\mu}|^2\,\rho_t^{2}\right)^{1/2}.
    \end{equation}
    Since $\mu_f(t)$ is supported on the pairwise disjoint grafted cylinders $C_{t,1},\dots,C_{t,k}$ we obtain
    \[
        \left\|\mu_f(t)\right\|^2_{L^2} = \frac{1}{4t^2}\sum_{i=1}^k \area_{X_t}(C_{t,i}),
    \]
    where $\area_{X_t}$ denotes the hyperbolic area on $X_t$. The hyperbolic metric on $X_t$ is bounded by the Thurston metric of the projective structure $\Grr_{t\Gamma}(X)$, see~\cite{Tanigawa1997,McMullen1998} which is the Euclidean metric on the grafted cylinder, so
    \[
        \area_{X_t}(C_{t,i}) \leq \area_{\Th}(C_{t,i}) = t\,\ell_i .
    \]
    Therefore,
    \[
        \left\|\mu_f(t)\right\|_{\WP}^2 \underset{\text{by}~(\ref{WPbelowL2})}{\leq} \frac{1}{4t^2} t\sum_{i=1}^k\ell_i \;=\; \frac{1}{4t}\sum_{i=1}^k\ell_i.
    \]
    Integrating, for $0<a<\tau$, yields
    \[
        \dwp(X_a,X_{\tau}) \leq \int_a^{\tau} \frac{1}{2}\sqrt{\sum_{i=1}^k \ell_i} t^{-1/2}\,dt
        = \sqrt{\sum_{i=1}^k \ell_i}\bigl(\sqrt{\tau}-\sqrt{a}\bigr).
    \]
    Since grafting is continuous~\cite{Tanigawa1997, McMullen1998, ScannelWolf2002}, one has $X_a\to X$ in $\teich_{g,n}$ as $a\to 0$, and so the distance estimate follows.
\end{proof}

By combining Lemma~\ref{WPgraftingPath} and Lemma~\ref{BoundStartGraftingRay}, we obtain the following estimate for the Weil--Petersson distance along a grafting ray:

\begin{proposition}\label{WPGraftingRay}
    Let  $g,n$ be integers such that $2g-2+n >0$, let $X \in \teich_{g,n}$, and let $\Gamma$ be a collection of disjoint simple closed curves on $X$. Let $(X_t)_{t\geq 0}$ denote the grafting ray starting at $X_0 = X$ defined by $X_t=\gr_{t\Gamma}(X)$, and let $X_\infty=\gr_{\infty\Gamma}(X)$. Then,
    \[
        \dwp(X, X_{\infty}) \leq 2^{5/4}\sqrt{\pi} \,\sqrt{\sum_{\gamma \in \Gamma} \ell_X(\gamma)} .
    \]
\end{proposition}

\begin{proof}
    Let $\tau>0$. By Lemma~\ref{BoundStartGraftingRay} and Lemma~\ref{WPgraftingPath}, then
    \[
        \dwp(X, X_{\infty}) \leq  \dwp(X, X_{\tau}) +  \dwp(X_{\tau}, X_{\infty}) 
        \leq \left(\sqrt{\tau} + \frac{\sqrt{2}\,\pi}{\sqrt{\tau}}\right)\sqrt{\sum_{\gamma \in \Gamma} \ell_X(\gamma)}.
    \]
    The choice $\tau=\sqrt{2}\,\pi$ minimizes the prefactor and gives $2(\sqrt{2}\pi)^{1/2}=2^{5/4}\sqrt{\pi}$.
\end{proof}

\subsection{Projection along a grafting ray while controlling shear parameters}\label{sec:graftshear2} Let $X$ in $ \overline{\teich}_{g,n}$ and let $(\Gamma, \mathcal{A})$ be a short hexagon decomposition of $X$. Recall that by Corollary~\ref{Bershear}, the shear parameters with respect to the spiralling triangulation $\mathcal{T}_{(\Gamma, \mathcal{A})}$ are bounded on $X$ by 
\[
    16\log(8\pi(2g-2+n))+23.
\] 
Our goal is to project $X$ to the stratum $S(\Gamma)$ obtained by pinching the curves of $\Gamma$, while maintaining uniform control on the shear parameters. The idea is to travel along iterated grafting rays. Although grafting rays are at bounded distance from Teichmüller geodesics~\cite{DiazKim2012}, grafting does not form a flow~\cite{hensel2008}; in particular $\gr_{t\Gamma}(\gr_{s\Gamma}(X))$ is not the same as $\gr_{(t+s)\Gamma}(X)$. 

First we graft along the curves of $\Gamma$ whose lengths are at least $\epsilon_T$. This yields a surface $X_{t_0}$ on which every curve of $\Gamma$ has length at most $\epsilon_T$. We show that $(\Gamma, \mathcal{A})$ remains uniformly ``short'' on $X_{t_0}$ and therefore by Theorem~\ref{BershearRmk2}, the shear parameters of $\mathcal{T}_{(\Gamma, \mathcal{A})}$ remains uniformly bounded. 
Finally we graft simultaneously along the entire multicurve $\Gamma$. Along this second grafting ray, the lengths of the curves in $\Gamma$ tend to zero while the truncated length of the orthogeodesic arcs remain uniformly controlled. This gives uniform upper bound for the shear parameters along the grafting ray. By continuity of the shearing coordinates under pinching~\cite{Roger2013} this upper bound extends to the limit noded surface $X_{\infty}$.  

Recall that the endpoint $X_{\infty}$ lies in the stratum $S(\Gamma)$. In this stratum, the spiralling triangulation $\mathcal{T}_{(\Gamma, \mathcal{A})}$ is the ideal spiralling triangulation $\mathcal{A}_\infty$ obtained from $(\Gamma, \mathcal{A})$ by pinching the curves of $\Gamma$. The shearing coordinates associated to $\mathcal{A}_\infty$ give a global parametrization of $S(\Gamma)$~\cite{Bonahon1996}.

\begin{figure}[h]
    \centering
    \begin{overpic}[width=1.0\linewidth,keepaspectratio]{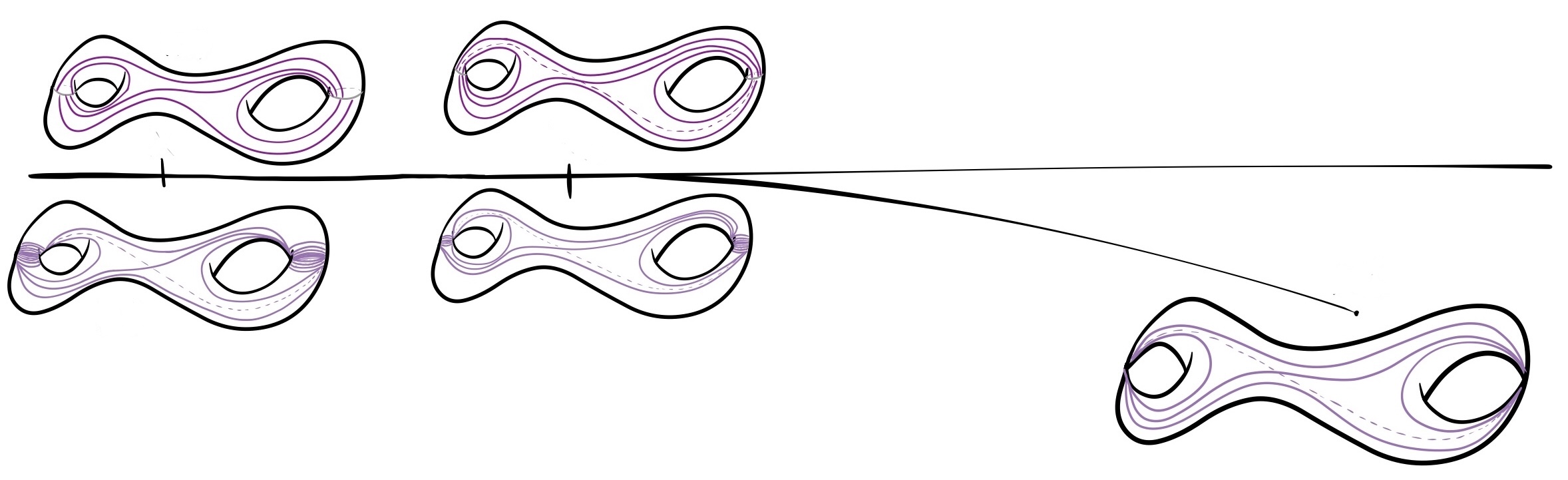}
    \put(9.3,20.6){\small$X$}
    \put(34,21){$X_{t_0}$}
    \put(86,12){$X_{\infty}$}
    \end{overpic}
    \vspace{-0.6cm}
    \caption{Iterating grafting rays to project a point onto a stratum.}\label{GraftingShearFig}
\end{figure}

\begin{proposition}\label{GraftingShear}
    Let $(\Gamma, \mathcal{A})$ be a hexagon decomposition and let $X\in V(\Gamma, \mathcal{A})$. Then there exists a noded surface 
    \[
        X_\infty \in S(\Gamma)
    \] 
    such that the shearing coordinates of $X_\infty$ with respect to the ideal spiralling triangulation $\mathcal{A}_{\infty}$ satisfy: 
    \[ 
        \shear_{\mathcal{A}_{\infty}}(X_\infty) \leq 11 \log(8\pi(2g-2+n))+80.
    \] 
\end{proposition}

\begin{proof}
    Let $(\Gamma,\mathcal{A})$ be a hexagon decomposition and let $X\in V(\Gamma,\mathcal{A})$.
    
    The strategy is as follows. First we graft along the curves of $\Gamma$ that are longer than $\epsilon_T$ in order to obtain a surface on which every curve of $\Gamma$ has length at most $\epsilon_T$. We then graft along the entire multicurve $\Gamma$ and control the shear parameters along the second grafting ray.
    
    Let
    \[
        \Gamma_0=\{\gamma\in\Gamma : \ell_X(\gamma)>\epsilon_T \}.
    \]
    If $\Gamma_0=\varnothing$, set $X_{t_0}=X$. Otherwise define
    \[
        X_{t_0}=\gr_{t_0\Gamma_0}(X),
        \qquad
        t_0=\pi\left(\frac{L}{\epsilon_T}-1\right).
    \]
    Next, we show that the hexagon decomposition $(\Gamma, \mathcal{A})$ remains $(\epsilon_T,L_A',\wtrunc)$-short on $X_{t_0}$. For every $\gamma\in\Gamma_0$, Lemma~\ref{DiazKimLengthGr} gives
    \[
        \ell_{X_{t_0}}(\gamma) \le \frac{\pi}{\pi+t_0}\ell_X(\gamma) \le \epsilon_T.
    \]
    Hence every curve of $\Gamma$ has length at most $\epsilon_T$ on $X_{t_0}$.
    Recall that $X_{t_0}$ is the uniformization of the Thurston metric obtained by inserting Euclidean cylinders along the curves of $\Gamma_0$, see~\cite{DumasWolf2008}. The hyperbolic metric on $X_{t_0}$ is smaller than the Thurston metric which coincides with the hyperbolic metric of $X$ outside the grafted parts, see for instance~\cite{Tanigawa1997,McMullen1998}. Hence if $a\in \mathcal{A}$ does not end on curves in $\Gamma_0$ we obtain 
    \[
        \ell_{X_{t_0},\Gamma}^{\wtrunc}(a)\leq \ell_{X,\Gamma}^{\wtrunc}(a) \leq L_A.
    \]  
    Suppose now that $a$ ends on a curve of $\Gamma_0$. Then it is prolonged, via the grafting operation, with an horizontal arc from the core curve of the grafting cylinder to its boundary that we denote by $\eta_{t_0}$. Since every curve of $\Gamma$ has length at most $\epsilon_T$ on $X_{t_0}$ and since the hyperbolic metric on $X_{t_0}$ is smaller than the Thurston metric we obtain:
    \begin{equation*}
        \ell_{X_{t_0},\Gamma}^{\wtrunc}(a)\leq \ell_{X,\Gamma}^{\wtrunc}(a) + \frac{1}{2}\sum_{i=1}^{k} \ell_{X_{t_0}}(\eta_{t_0,i}),
    \end{equation*}
    where $k\in\{1,2\}$ depends on whether $a$ terminates on one or two curves of $\Gamma_0$.
    It therefore remains to estimate the length of the horizontal segments $\eta_{t_0,i}$, on $X_{t_0}$. By Lemma~\ref{DiazKimLengthGr2}, 
    \begin{equation}\label{eq_ortho-geod-arc2}
        \frac{1}{2}\ell_{X_{t_0}}(\eta_{t_0,i})\leq \tau + \log\left(\frac{\cos\left(\frac{\tau \pi}{2 t_0}\right)}{\sin\left(\frac{\tau \pi}{2 t_0}\right)}\right)= \tau + \log\left(\cot\left(\frac{\tau \epsilon_T}{2(L-\epsilon_T)}\right)\right),
    \end{equation}
    where $\tau$ is a fixed positive number $<t_0$, take $\tau=1$. Note that $2g-2+n >0$ and so $L>2.8$. Hence the right-hand side of the inequality~(\ref{eq_ortho-geod-arc2}) is bounded by: 
    \[
        \frac{1}{2}\ell_{X_{t_0}}(\eta_{t_0,i})\leq \tau + \log(L) + 3,
    \] 
    and so we obtain the following bound: 
    \[
        \ell_{X_{t_0},\Gamma}^{\wtrunc}(a)\leq\ell_{X,\Gamma}^{\wtrunc}(a) + 2\log(L)+ 8 \leq L_A'.
    \]
    where $L_A'= L_A+2\log(L)+8$. Therefore the hexagon decomposition $(\Gamma, \mathcal{A})$ is $(\epsilon_T, L_A', \wtrunc)$-short on $X_{t_0}$.
    Consider the grafting ray $(X_t)_{t\geq 0}$ defined by
    \[
        X_{t}=\gr_{t\Gamma}(X_{t_0}).
    \]
    Let $t >0$, by Lemma~\ref{DiazKimLengthGr} every curve of $\Gamma$ has length at most $\epsilon_T$ on $X_{t}$. To bound the shear parameters with respect to $\mathcal{T}_{(\Gamma,\mathcal{A})}$ on $X_{t}$, by Theorem~\ref{BershearRmk2}, it remains to bound the truncated length of each orthogeodesic arc $ a \in \mathcal{A}$ on $X_{t}$. As before, if  $a\in \mathcal{A}$ does not end on curves in $\Gamma$ (that is, it ends in cusps) we obtain 
    \[
        \ell_{X_{t},\Gamma}^{\wtrunc}(a)\leq \ell_{X_{t_0},\Gamma}^{\wtrunc}(a)\leq L_A'.
    \]  
    Otherwise, the orthogeodesic arc $a$ decomposes into a portion outside the grafting cylinders and one or two horizontal segment(s) $\eta_{t,i}$ inside the grafted cylinders. Fix $0 < \tau < t$ then Lemma~\ref{DiazKimLengthGr2} implies that a horizontal segment $\eta_t$ satisfies

    \begin{equation}\label{length_horizontal_subarc}
        \frac{1}{2}\ell_{X_t}(\eta_{t,i}) \leq \log\left(\cot \left(\frac{\tau \pi}{2 t}\right)\right) + \tau. 
    \end{equation}
    Let $\gamma \in \Gamma$ and recall that by Lemma~\ref{DiazKimLengthGr},
    \[
        \ell_{X_{t}}(\gamma)\leq\frac{\pi}{\pi+t}\ell_{X_{t_0}}(\gamma)\leq \frac{\pi}{\pi+t} \epsilon_T.
    \] thus it follows that 
    \begin{equation}\label{width_collar_comparison}
        \wtrunc(\ell_{X_t}(\gamma)) \geq \wtrunc\left(\frac{\pi}{\pi+t}\epsilon_T \right)=\arccosh\left(\frac{\pi+t}{\pi}\right)
    \end{equation}
    and 
    \begin{equation}\label{width_std_collar_comparison}
        \wstd(\ell_{X_t}(\gamma)) \geq \arcsinh\left(\frac{1}{\sinh\left(\frac{\pi}{\pi +t} \epsilon_T \right)}\right).
    \end{equation}
    By comparing the functions on the right hand-side of inequalities (\ref{length_horizontal_subarc}), (\ref{width_collar_comparison}) and (\ref{width_std_collar_comparison}) we observe that they all grow asymptotically like $\log(t) + \log\left(\frac{2}{\tau \pi}\right)+ \mathcal{O}(1)$ for large $t$. In particular, we obtain that on $X_t$ the grafted subarc stays inside the standard collar neighborhood of the geodesic representative of its endcurve(s) in $\Gamma$. Consequently, the grafted part contributes to at most $2\Delta_{\wtrunc}$ to the $\wtrunc$-truncated length of the arc $a$ and so for large enough $t$,
    \[
        \ell_{X_t, \Gamma}^{\wtrunc}(a)\leq \ell_{X_{t_0}, \Gamma}^{\wtrunc}(a)+ 2\Delta_{\wtrunc} \leq L_A' + 2\Delta_{\wtrunc}.
    \]
    Thus for large enough $t$, the hexagon decomposition $(\Gamma, \mathcal{A})$ is $(\epsilon_T, L_A' + 2\Delta_{\wtrunc}, \wtrunc)$-short on $X_t$. 
    Hence, by Theorem~\ref{BershearRmk2}, the shear parameters on $X_{t}$ with respect to $\mathcal{T}_{(\Gamma,\mathcal{A})}$ are bounded above by $4(L_A’+2\Delta_{\wtrunc})+8\epsilon_T+8\varrho = 8\log(8\pi(2g-2+n))+8\log(L)+16\Delta_{\wtrunc}+32+20\varrho$, and so we obtain
    \[
        11 \log(8\pi(2g-2+n))+80.
    \]

    Let $t~\rightarrow~\infty$ then the grafting ray $(X_t)_{t\geq0}$ converges to a noded surface $X_{\infty} \in S(\Gamma)$~\cite{hensel2008}. Roger proved that shearing coordinates extend continuously to the augmented Teichmüller space under pinching~\cite[Proposition 6]{Roger2013}. In the stratum $S(\Gamma)$, the spiralling triangulation $\mathcal{T}_{(\Gamma, \mathcal{A})}$ is the ideal spiralling triangulation $\mathcal{A}_\infty$ obtained from $(\Gamma, \mathcal{A})$ by pinching the curves of $\Gamma$. Since the shearing coordinates with respect to $\mathcal{T}_{(\Gamma,\mathcal{A})}$ remain uniformly bounded along the ray, the same bound holds at the limit:
    \[
        \shear_{\mathcal{A}_{\infty}}(X_\infty)
        \leq
        11 \log(8\pi(2g-2+n))+80,
    \]
    this completes the proof.
\end{proof}

We now combine Proposition~\ref{WPGraftingRay}, which controls the Weil--Petersson length of a grafting ray, with Proposition~\ref{GraftingShear} which controls shear parameters along a grafting ray.

\begin{proposition}\label{TravellingToStratum}
    Let $(\Gamma, \mathcal{A})$ be a hexagon decomposition and let $X\in V(\Gamma, \mathcal{A})$. Then there exists a noded surface 
    \[ 
        X_\infty \in S(\Gamma)
    \]
    such that:
    \[ 
        \shear_{\mathcal{A}_{\infty}}(X_\infty) \leq   11 \log(8\pi(2g-2+n))+80.
    \]
    and 
    \[
        \dwp(X,X_\infty)\leq 5\sqrt{2\pi(2g-2+n)\log(8\pi(2g-2+n))}.
    \]
\end{proposition}

\begin{proof}
    Let  
    \[
        \Gamma_0=\{\gamma\in\Gamma : \ell_X(\gamma)>\epsilon_T \}.
    \] 
    By the proof of Proposition~\ref{GraftingShear}, there exists a surface 
    \[
        X_{t_0}=\gr_{t_0 \Gamma_0}(X)
    \] 
    where the curves in $\Gamma$ have lengths at most $\epsilon_T$ on $X_{t_0}$. 
    Moreover the grafting ray $(X_t)_{t\geq 0}$ where 
    \[
        X_t=\gr_{t \Gamma}(X_{t_0})
    \]
    converges to a noded surface $X_\infty \in S(\Gamma)$ whose shearing coordinates with respect to $\mathcal{A}_{\infty}$ satisfy
    \[ 
        \shear_{\mathcal{A}_{\infty}}(X_\infty) \leq 11\log(8\pi(2g-2+n))+80.
    \]
    
    It remains to estimate the Weil--Petersson distance from $X$ to $X_\infty$. By the triangle inequality,
    \[
        \dwp(X,X_\infty)\leq \dwp(X,X_{t_0}) + \dwp(X_{t_0}, X_\infty).
    \]
    The surface $X_{t_0}$ lies on the grafting ray $(\gr_{t\Gamma_0}(X))_{t\geq0}$ and every segment of
    the ray has Weil–Petersson length at most the length of the ray, so Proposition~\ref{WPGraftingRay} yields
    \[
        \dwp(X,X_{t_0})
        \leq 2^{5/4}\sqrt{\pi} \sqrt{\sum_{\gamma \in \Gamma_0} \ell_X(\gamma)}
        \leq 2^{5/4}\sqrt{\pi}\sqrt{(3g-3+n)\,L},
    \]
    since $\ell_X(\gamma)\leq L$ for every $\gamma\in\Gamma$ (the hexagon decomposition $(\Gamma,\mathcal{A})$ is short on $X$) and $|\Gamma_0|\leq 3g-3+n$. Similarly, applying Proposition~\ref{WPGraftingRay} to the ray $(\gr_{t\Gamma}(X_{t_0}))_{t\geq0}$ and using $\ell_{X_{t_0}}(\gamma) \leq \epsilon_T$ for every $\gamma\in\Gamma$,
    \[
        \dwp(X_{t_0}, X_\infty) \leq 2^{5/4}\sqrt{\pi}\, \sqrt{(3g-3+n)\,\epsilon_T}.
    \]
    Write $m=2g-2+n$, so that $3g-3+n\leq \frac{3}{2}m$ and $L=2\log(8\pi m)$. Since $\epsilon_T=\frac{3\log 3}{8}\leq 0.13\log(8\pi)\leq 0.13 \log(8\pi m)$, we obtain
    \begin{align*}
        \dwp(X,X_\infty)
        &\leq 2^{5/4}\sqrt{\pi}\sqrt{\tfrac{3}{2}m} \left(\sqrt{2\log(8\pi m)} + \sqrt{\epsilon_T}\right)\\
        &\leq 2^{5/4}\sqrt{\tfrac{3\pi}{2}} \left(\sqrt{2}+\sqrt{0.13}\right)\sqrt{m\log(8\pi m)}
        \leq 5\sqrt{2\pi m \log(8\pi m)}\\
        &\leq 5\sqrt{2\pi(2g-2+n)\log(8\pi(2g-2+n))}.
    \end{align*}

\end{proof}

\section{Projection to the thick part of a stratum}\label{sec:project}

Let $X$ in $ \overline{\teich}_{g,n}$, we show that there exists a hexagon decomposition $(\Gamma,\mathcal{A})$ such that $X$ projects to the thick part of the stratum $S(\Gamma)$. Moreover we control the Weil-Petersson distance from $X$ to the stratum while maintaining uniform control on the shear parameters of the arrival point.

The strategy is as follows. Start with a short hexagon decomposition $(\Gamma_0,\mathcal{A}_0)$ on $X$. Graft along the curves of $\Gamma_0$ whose lengths are larger than $\epsilon_T$. This yields a hyperbolic surface $X_{t_0}$ on which every curve of $\Gamma_0$ has length at most $\epsilon_T$. The decomposition $(\Gamma_0,\mathcal{A}_0)$ is $\wtrunc$-short on $X_{t_0}$. During this first grafting process, additional curves may become short. If a curve not contained in $\Gamma_0$ has length below $\epsilon_T$ on $X_{t_0}$, we add it to the pinching multicurve. This yields a new $\wtrunc$-short hexagon decomposition $(\Gamma,\mathcal{A})$ on $X_{t_0}$ such that every curve in $\Gamma$ has length at most $\epsilon_T$. Then consider the grafting ray obtained by grafting along the entire multicurve $\Gamma$ sarting from $X_{t_0}$. Since every curve in $\Gamma$ has length at most $\epsilon_T$ on $X_{t_0}$, their collars are sufficiently long to absorb the change in the metric induced by the grafting operation. Therefore, along this second grafting ray, curves lying in the thick part of the surface remain in the thick part. Consequently the endpoint $X_{\infty}$ lies in the thick part of the stratum $S(\Gamma)$. Moreover, the lengths of the curves in $\Gamma$ tend to zero, while the truncated lengths of the orthogeodesic arcs in $\mathcal{A}$ remain uniformly bounded. Hence the shear parameters of the spiralling triangulation $\mathcal{T}_{(\Gamma,\mathcal{A})}$ are bounded along this ray and the same bound hold at the limiting surface $X_{\infty}$.

\subsection{Grafting along very short simple closed curves}\label{sec:veryshort}
We prove the following auxiliary lemma, which shows that a short hexagon decomposition, whose curves all have length at most $\epsilon \leq \epsilon_0$, can be enlarged to include every simple closed curve with length at most $\epsilon$.

\begin{lemma}~\label{LemmaNiceHexDecomp}
    Let $\wabstract:\R_{>0} \longrightarrow \R_{\geq 0}$ be a truncation function. Given $\epsilon \leq \epsilon_0$ and $X\in \teich_{g,n}$, let $(\Gamma, \mathcal{A})$ be an $(\epsilon, L_A, \wabstract)$-short hexagon decomposition of $X$. Then there exists a hexagon decomposition $(\Gamma_1, \mathcal{A}_1)$ of $X$ such that
    \begin{enumerate}[topsep=1pt]
        \item $\Gamma \subseteq \Gamma_1$ and $\ell_{X}(\gamma) \leq \epsilon$ for all $\gamma \in\Gamma_1$,
        \item $\ell_{X,\Gamma_1}^{\wabstract}(a)\leq 4 L_A + \epsilon$ for all $a\in\As_1$,
        \item $\ell_{X}(\alpha) > \epsilon$ for all simple closed curve $\alpha$ not in $\Gamma_1$.
    \end{enumerate}
\end{lemma}

\begin{proof}
    Let $X\in \teich_{g,n}$ and let $(\Gamma, \mathcal{A})$ be a $(\epsilon, L_A, \wabstract)$-short hexagon decomposition of $X$ such that $\ell_X(\gamma) \leq \epsilon$ for all $\gamma \in \Gamma$. Let $\Gamma_1$ be the set of all (disjoint) simple closed curves on $X$ with length at most $\epsilon$:
    \[
        \Gamma_1=\left\{ \alpha \text{ a simple closed curve} \mid \ell_X(\alpha) \leq \epsilon \right\}.
    \]
    By construction, $\Gamma \subseteq \Gamma_1$ and $\Gamma_1$ satisfies conditions (1) and (3) of the Lemma. 
    
    Consider the $(\Gamma_1,\wabstract)$-truncation of $X$,
    \[
        X_{1}=X-\left(\bigcup_{i=1}^n N_{h_0(\wabstract)}(c_i) \cup \bigcup_{\gamma\in\Gamma_1} \mathcal{C}_{\wabstract(\ell_X(\gamma))}(\gamma)\right),
    \]
    and its $(\Gamma,\wabstract)$-truncation
    \[
        X_{0}=X-\left(\bigcup_{i=1}^n N_{h_0(\wabstract)}(c_i) \cup \bigcup_{\gamma\in\Gamma} \mathcal{C}_{\wabstract(\ell_X(\gamma))}(\gamma)\right).
    \]
    Let $x\in X_0$, then $x$ belongs to some hexagon $H \in (\Gamma, \mathcal{A})$. Since $(\Gamma, \mathcal{A})$ is short on $X$, and $\ell_X(\gamma) \leq \epsilon$ for all $\gamma \in \Gamma$, the distance from $x$ to $\partial H$ is bounded by half of sum of the length of three consecutive sides
    \[
        d(x, \partial H) \leq \frac{1}{2}\left( 2L_A + \epsilon\right),
    \]
    therefore 
    \[
        d(x, \partial X_0) \leq d(x, \partial H) + L_A \leq \mathcal{L}_A' \coloneq 2 L_A + \frac{\epsilon}{2}.
    \]
    
    Because the $\wabstract$-collars around the curves in $\Gamma_1$ are mutually disjoint, we obtain $\partial X_0 \subseteq \partial X_1$. Consequently, for any point $x \in X_1$ we have:
    \begin{equation}\label{bound_arc}
        d(x, \partial X_1) \leq d(x, \partial X_0) \leq \mathcal{L}_A'.
    \end{equation}

    Let $B_1$ denote the collection of boundary components of $X_1$. We consider the Voronoï cell decomposition of $X_1$ with respect to $B_1$ as defined in~\cite{BowditchEpstein1988},~\cite[Section 2.1]{Mondello2009} and~\cite[Section 2.2]{Budd2025}. We present briefly the construction, we proceed as in~\cite[Theorem 1.3]{Parlier2016} and~\cite[Proposition 2.8]{Bershear}. The valence of a point $x\in X_1$, denoted by $\nu(x)$ is the number of shortest geodesics from $x$ to $B_1$ realizing the distance $d(x,B_1)$. Note that $\nu(x) \geq 1$ and two distinct shortest geodesics at one point can end on the same boundary component. Define 
    \[ 
        V= \left\{x \in X_1 \mid \nu(x)\geq 3 \right\} \text{ and } N= \left\{x \in X_1 \mid \nu(x)=2 \right\}. 
    \]
    Then $V$ is a finite collection of points called \emph{vertices} while $N$ is a finite disjoint union of simple open geodesic arcs called \emph{edges}. Their union $V \cup N$ is a $1$-dimensional CW-complex called the cut locus (or spine) of $X_1$. For $e \in N$, let $\beta_e$ be an edge of the cut locus of $X_1$. Pick any point $x \in \beta_e$ and let $\gamma_1$ and $\gamma_2$ be the two shortest geodesics joining $x$ with $B_1$. Let $\alpha_e$ be the orthogeodesic arc in the isotopy class of $\gamma_1 \cup \gamma_2$, that is meeting $B_1$ perpendicularly. We say that $\alpha_e$ is dual to $\beta_e$. Recall that by construction of $X_1$, $\alpha_e$ has length at most $2 \mathcal{L}_A'$ since $x$ lies at distance at most $\mathcal{L}_A'$ from $\partial X_1$. If the cut locus of $X_1$ is trivalent, that is if all the vertices have valence 3, then the collection of dual arcs $(\alpha_e)_{e\in N}$ forms a valid hexagon decomposition of $X_1$ and is called the \emph{dual graph} of the spine. If $V$ contains a vertex $v$ of valence $\nu(v)=k > 3$, we add arcs to its dual graph as follows. Let $(a_i)_{i=1}^{k}$ be the family of shortest geodesic from $v$ to $B_1$. Consider the isotopy classes of the $k-3$ concatenations $a_1 \star a_3, \dots, a_1 \star a_{k-1}$. Then we add to the dual graph the orthogeodesic representatives from each isotopy classes. Repeating this process for all vertices of degree greater than $3$ yields an augmented dual graph that defines a valid hexagon decomposition of $X_1$. Since each $a_i$ has length at most $ \mathcal{L}_A'$, the new arcs have lengths at most $2 \mathcal{L}_A'$ on $X_1$. 
    
    Since $X_1$ was obtained by removing disjoint collars and cusp neighborhoods, we can continuously extend the arcs of this augmented dual graph from $X_1$ back across the collars to the full surface $X$. This construction yields an arc set $\mathcal{A}_1$ such that $(\Gamma_1, \mathcal{A}_1)$ forms a hexagon decomposition of $X$. By construction, any curve $\gamma \in \Gamma_1$ satisfies $\ell_X(\gamma) \leq \epsilon$, all simple closed curves shorter than $\epsilon$ are contained in $\Gamma_1$, and the truncated length of any arc $a \in \mathcal{A}_1$ satisfies $\ell_{X, \Gamma_1}^{\wabstract}(a) \leq 2 \mathcal{L}_A'$, completing the proof.

\end{proof}

For a closed hyperbolic surface $X$ and a point $x \in X$, the injectivity radius, denoted by $\inj_X(x)$, is the radius of the largest embedded geodesic ball centered at $x$. By~\cite[Lemma 4.1.5]{Buser} the injectivity radius is half the length of the shortest non-contractible closed loop passing through $x$.

Let $\epsilon \leq \epsilon_0$, the $\epsilon$-thick part of $X$, denoted by $X_{\geq \epsilon}$, is the subsurface where the injectivity radius is at least $\epsilon$: 
\[
    X_{\geq \epsilon}= \{ x \in X \mid \inj_X(x) \geq \epsilon\}.
\]
The complement of this region is the $\epsilon$- thin part of the surface, denoted by $X_{\leq \epsilon}$, which consists of disjoint collars around short geodesics and neighborhoods of cusps. If an essential simple closed curve $\gamma$ lies entirely inside the $\epsilon$-thick part $X_{\geq \epsilon}$, then its hyperbolic length satisfies: \[ \ell_X(\gamma)\geq 2\epsilon.\]

Grafting along very short curves produces long hyperbolic collars. As a result, the change in the metric is absorbed inside these long collars, and the effect of grafting decays exponentially as one moves away from the grafting locus. Consequently, curves lying in the thick part of the surface---away from the grafting locus---are unaffected by the deformation that is any curve initially in the thick part remains in the thick part. 

\begin{lemma}\label{LemmaStayThick}
    Given $\epsilon \leq \epsilon_0/2$ and $X \in \overline{\teich}_{g,n}$, let $(\Gamma, \mathcal{A})$ be a hexagon decomposition of $X$ satisfying the following conditions:
    \begin{itemize}
        \item $\ell_X(\gamma) \leq \epsilon$ for all $\gamma \in \Gamma$,
        \item $\ell_X(\beta) > \epsilon$ for every simple closed curve $\beta \notin \Gamma$.
    \end{itemize}
    For $t\geq 0$, let $X_t \coloneq \gr_{t\Gamma}(X)$. Then there exists a constant $\epsilon' > 0$, depending only on $\epsilon$, such that for all $t \geq 0$ and all simple closed curves $\alpha$ within $X_t \smallsetminus \Gamma$ we have,
    \[
        \ell_{X_t}(\alpha) \geq \epsilon',
    \]
    where $\epsilon’$ is given by
    \[
        \epsilon’ \;=\; \frac{1-\sinh\left(\frac{\epsilon}{2}\right)}{1+\sinh\left(\frac{\epsilon}{2}\right)}
        \frac{\epsilon}{\sqrt{1+3\coth^2\left( \frac{\epsilon}{4}\right)}}.
    \]
\end{lemma}

\begin{proof}
    Let $X\in \overline{\teich}_{g,n}$ and let $(\Gamma, \mathcal{A})$ be a hexagon decomposition of $X$ satisfying the hypotheses of the lemma. Fix a simple closed curve $\alpha$ disjoint from $\Gamma$. Let $t\geq 0$ and set 
    \[
        X_t = \gr_{t\Gamma}(X).
    \] 
    Denote by $\rho_{\Th}$ and $\rho_{X_t}$ the  line element of, respectively, the Thurston projective metric and the hyperbolic metric on $X_t$. The Thurston metric is conformally equivalent to the hyperbolic metric on $X_t$~\cite{ChoiDumasRafi2012} thus 
    \[
        \rho_{X_t}=e^{2u}\rho_{\Th},
    \] 
    where $u$ is a real-valued function. Since the hyperbolic metric is smaller than the Thurston projective metric on $X_t$, we have $ \rho_{X_t} \leq \rho_{\Th}$ and so
        \begin{align*}
             u \leq 0.
        \end{align*}  
    Let $v\coloneq -u $ which implies $v\geq 0$. The Gaussian curvature of $\rho_{\Th}$ is well-defined everywhere except at the boundary of the grafting cylinders: its value is 0 inside the grafting cylinder and -1 outside. Hence, outside the grafting cylinders, by the conformal change of curvature, the density function satisfies
    \[
        v\geq 0, \qquad \Delta_{Th}v = \frac{1}{2} \left(1-e^{-4v}\right) \geq 0
    \]  
    where $\Delta_{Th}$ is the Laplacian operator with respect to the Thurston metric. Thus $v$ is a subharmonic function. While inside the grafting cylinders the density function satisfies, 
    \[
        \Delta_{Th}v = -\frac{1}{2}e^{-4v}\leq 0
    \] 
    and so $v$ is a superharmonic function.
    
    Consider the truncated surface $\hat{X_t}$ obtained from $X_t$ by removing the standard collars around the simple closed curves in $\Gamma$:
    \[
        \hat{X_t}=X_t\smallsetminus \cup_{\gamma \in \Gamma} \mathcal{C}_{\wstd(\ell_{X_t}(\gamma))}(\gamma).
    \] 
    Note that the symmetric grafting cylinder might not be entirely contained in the standard collar after uniformization. However, inside the grafting cylinder, the function $v$ is strictly decreasing as we move away from the core curve. Furthermore, by the maximum principle, the maximum of the subharmonic function $v$ outside the grafting cylinders must be attained on the boundary. Therefore on $\hat{X_t}$, the extremal value of $v$ is achieved on the boundary closest to the grafting locus. 

    Let $\delta=\frac{\epsilon}{2}$ and define
    \[
        F(\ell)=\wstd(\ell)-\arccosh\left(\frac{\sinh\left(\delta\right)}{\sinh\left(\frac{\ell}{2}\right)}\right),
        \qquad
        R_{\delta}=\inf_{\ell\in(0,\epsilon]} F(\ell),
    \]
    and
    \begin{equation}\label{defintion_K_delta}
        K(\delta)\coloneq\frac{3}{2}\coth^2\left(\frac{\delta}{2}\right)=\frac{3}{2}\coth^2\left(\frac{\epsilon}{4}\right).
    \end{equation}
    \begin{claim}\label{claimC}
        For every $\gamma\in\Gamma$, the conformal factor satisfies 
        \[
            e^{-2\max_{z_0\in\partial \mathcal{C}_{\wstd(\ell_t(\gamma))}(\gamma)}v(z_0)}
            \;\geq\;
            \frac{1}{\sqrt{1+2K(\delta)}\,\coth\left(\frac{R_{\delta}}{2}\right)} .
        \]
    \end{claim}
    \begin{proofclaim}
        We follow closely~\cite[Proof of Lemma 5.20]{CremaschiGiovannini2025}. The function $F$ is increasing on $\ell\in(0,\epsilon]$, with $F(\epsilon)=\wstd(\epsilon)$, we obtain  
        \[
            \lim_{\ell\to 0}F(\ell)
            =\log\left(\frac{1}{\sinh(\delta)}\right),
        \]
        so that the infimum is not attained and
        \begin{equation}\label{bound_R_delta}
            R_{\delta}=\log\left(\frac{1}{\sinh\left(\frac{\epsilon}{2}\right)}\right)>0,
            \quad\text{equivalently}\qquad
            \coth\left(\frac{R_{\delta}}{2}\right)=\frac{1+\sinh\left(\frac{\epsilon}{2}\right)}{1-\sinh\left(\frac{\epsilon}{2}\right)}.
        \end{equation}
        Fix $\gamma \in \Gamma$, let $z_0\in \partial \mathcal{C}_{\wstd(\ell_X(\gamma))}(\gamma)$ and consider a point $p\in B(z_0, R_{\delta}) \cap \mathcal{C}_{\wstd(\ell_{X_t}(\gamma))}(\gamma)$. Let $d_p$ be the distance from $p$ to the boundary $\partial \mathcal{C}_{\wstd(\ell_X(\gamma))}(\gamma)$, let $\ell=\ell_{X_t}(\gamma)$. By~\cite[Lemma 4.1.5]{Buser}, the injectivity radius at $p$ satisfies:
        \begin{align*}
            \inj(p)&= \arcsinh\left(\sinh\left(\frac{\ell}{2}\right) \cosh\left(\wstd(\ell)-d_p \right)\right) \eqcolon G(\ell, d_p) 
        \end{align*}

        Observe that $G(\ell,x)$ is continuous and decreasing with respect to $x$. Since $p\in B(z_0, R_{\delta}) \cap \mathcal{C}_{\wstd(\ell_X(\gamma))}(\gamma)$, we have $d_p\leq R_{\delta}$. Therefore: 
            \begin{align*}
                \text{inj}(p) &\geq \text{arcsinh}\left( \sinh\left(\frac{\ell}{2}\right) \cosh(\wstd(\ell) - R_{\delta}) \right) \\
                &\geq \text{arcsinh}\left( \sinh\left(\frac{\ell}{2}\right) \cosh(\wstd(\ell) - F(\ell)) \right) \\
                &\geq \text{arcsinh}\left( \sinh\left(\frac{\ell}{2}\right) \cdot \frac{\sinh(\delta)}{\sinh\left(\frac{\ell}{2}\right)} \right) \\
                &= \text{arcsinh}(\sinh(\delta))= \delta.
            \end{align*} 
        
    Now let $f$ be the developing map of the complex projective structure $Gr_{t\Gamma}(X)$. Consider the restriction of $f$ to the long complex projective tube $f:\Delta\rightarrow \Xi $. By~\cite[Section 5.1]{CremaschiGiovannini2025}, its image $\Xi $ is not simply connected. Let $\Schw(f)$ be the Schwarzian derivative of $f$. By~\cite[Corollary 2.12]{BridgemanBrockBromberg2019}, for any $z\in \Xi$ we have
        \begin{align*}
            \|\Schw(f)(z)\| \leq \frac{3}{2}\coth^2\left(\frac{\inj(z)}{2}\right).
        \end{align*}
    In particular for $z_0\in \partial\mathcal{C}_{\wstd(\ell_X(\gamma))}(\gamma)$ we established that $\inj(p)\geq \delta$ for all $p\in B(z_0, R_{\delta}) \cap \mathcal{C}_{\wstd(\ell_X(\gamma))}(\gamma)$. Then take $p \in B(z_0, R_{\delta}) \smallsetminus \mathcal{C}_{\wstd(\ell_X(\gamma))}(\gamma)$. Since every simple closed curve outside $\Gamma$ has length $>\epsilon$, such $p$ lies outside the $\epsilon/2$-thin part and so $\inj(p)\geq \epsilon/2=\delta$. Consequently for all $p\in B(z_0, R_{\delta})$, 
    \begin{equation}~\label{defintion_K_delta}
        \|\Schw(f)(p)\| \leq \frac{3}{2}\coth^2\left(\frac{\delta}{2}\right) \eqcolon K(\delta).
    \end{equation}
    Applying~\cite[Theorem 2.8]{BridgemanBrockBromberg2019}, we obtain 
    \[
        \rho_{\Th}(z_0) \leq \rho_{X_t}(z_0)\sqrt{1+2K(\delta)}\coth\left(\frac{R_{\delta}}{2}\right).
    \] 
    Substituting in $\rho_{X_t}=e^{-2v}\rho_{\Th}$ yields the bound:
        \[ 
            e^{-2v(z_0)} \geq \frac{1}{\sqrt{1+2K(\delta)}\coth\left(\frac{R_{\delta}}{2}\right)},
        \]
        for any $z_0\in \partial\mathcal{C}_{\wstd(\ell_X(\gamma))}(\gamma)$, which proves the claim.
    \end{proofclaim}
    
    Finally, to finish the proof of the lemma, let $\alpha$ be a simple closed curve such that $\alpha \cap \Gamma= \emptyset$. By the hypotheses $\ell_X(\alpha) > \epsilon$. Let \begin{enumerate}
        \item[-] $\alpha^*$ be the geodesic representative of $\alpha$ with respect to $\rho_{\Th}$,
        \item[-] $\alpha^*_t$ be the geodesic representative of $\alpha$ with respect to $\rho_{X_t}$
    \end{enumerate}    
    Since $\alpha$ is disjoint from any curve $\gamma \in \Gamma$, by the Collar Theorem, $\alpha^*_t$ is disjoint from its standard collar $\mathcal{C}_{\wstd(\ell_{X_t}(\gamma))}(\gamma)$ on $X_t$. Hence $\alpha$ is entirely contained within $\hat{X_t}$. Integrating along $\alpha$ yields:
        \begin{align*}
            \ell_{X_t}(\alpha) &= \ell_{X_t}(\alpha^*_t) = \int_{\alpha^*_t}\rho_{X_t} = \int_{\alpha^*_t}e^{-2v}\rho_{\Th} \\ 
            &\geq e^{-2\max_{z \in \alpha^*_t} v(z)} \int_{\alpha^*_t} \rho_{\Th}(z)  \\
            &\geq e^{-2\max_{z \in \partial \mathcal{C}(\gamma')} v(z)} \int_{\alpha^*} \rho_{\Th}(z) =  e^{-2\max_{z \in \partial \mathcal{C}(\gamma)} v(z)} \ell_{X}(\alpha^*)\\
            &\geq \frac{ \ell_{X}(\alpha^*)}{\sqrt{1+2K(\delta)}\coth\left(\frac{R_{\delta}}{2}\right)} > \frac{\epsilon}{\sqrt{1+2K(\delta)}\coth\left(\frac{R_{\delta}}{2}\right)} \eqcolon \epsilon'
        \end{align*}    
    where the third inequality follows from Claim~\ref{claimC}. Finally substituting~(\ref{bound_R_delta}) and (\ref{defintion_K_delta}) gives:
    \[
        \epsilon’ = \frac{1-\sinh\left(\frac{\epsilon}{2}\right)}{1+\sinh\left(\frac{\epsilon}{2}\right)}\cdot
        \frac{\epsilon}{\sqrt{1+3\coth^2\left( \frac{\epsilon}{4}\right)}},
    \]
    for $\epsilon=\epsilon_T$ we obtain $\epsilon_T’ \approx 0.0159954\ldots\geq 0.0159$.
\end{proof}

\subsection{Controlled grafting path to the thick part}\label{sec:controlled} 
We use the results from the previous subsection to project any point in the augmented Teichmüller space into the thick part of a stratum, while controlling both the Weil--Petersson distance to the projected point and the shearing coordinates of the arrival point.

\begin{proposition}\label{PropProjectThickStartum}
    Let $X\in \overline{\teich}_{g,n}$. Then there exists a hexagon decomposition $(\Gamma_1, \mathcal{A}_1)$ of $X$ and a noded surface $X_\infty$
    such that\begin{enumerate}
        \item $\shear_{\mathcal{A}_{1\infty}}(X_\infty) \leq 44\log(8\pi(2g-2+n))+249,$
        \item  $X_\infty$ lies in the $\epsilon_T'$-thick part of the stratum $S(\Gamma_1)$, where $\epsilon_T'\approx 0.016$ is a universal constant, independent of the topology.
        \item $\dwp(X,X_{\infty})\leq 5 \sqrt{2\pi(2g-2+n)\log(8\pi(2g-2+n))}$.
    \end{enumerate}
\end{proposition}

\begin{proof}
    Let $X\in \overline{\teich}_{g,n}$, by Corollary~\ref{ShortHexagonDecompositionExistence} there exists a short hexagon decomposition $(\Gamma,\mathcal{A})$ of $X$. First, we graft along the curves of $\Gamma$ that are longer than $\epsilon_T$ in order to obtain a surface on which every curve of $\Gamma$ has length at most $ \epsilon_T$. 
    Let
    \[
        \Gamma_0=\{\gamma\in\Gamma : \ell_X(\gamma)>\epsilon_T \}.
    \]
    If $\Gamma_0=\varnothing$, set $X_{t_0}=X$. Otherwise define
    \[
        X_{t_0}=\gr_{t_0\Gamma_0}(X),
        \qquad
        t_0=\pi\Bigl(\frac{ L}{\epsilon_T}-1\Bigr).
    \]
    By proof Proposition~\ref{GraftingShear} the hexagon decomposition $(\Gamma,\mathcal{A})$ is $(\epsilon_T, L_A', \wtrunc)$-short on $X_{t_0}$ where $L_A'=L_A + 2\log(L) + 8$. By Lemma~\ref{LemmaNiceHexDecomp}, there exists an $(\epsilon_T, L_A'', \wtrunc)$-short hexagon decomposition $(\Gamma_1, \As_1)$ on $X_{t_0}$ where 
    \[ 
        L_A'' = 4L_A' + \epsilon_T = 4L_A + 8\log(L) + 32 + \epsilon_T. 
    \] 
    Moreover, any simple closed curve $\alpha$ not in $\Gamma_1$ satisfies 
    \[
        \ell_{X_{t_0}}(\alpha) > \epsilon_T.
    \]

    Consider the grafting ray $(X_t)_{t\geq0}$ defined by $X_{t}=\gr_{t\Gamma_1}(X_{t_0})$. By Lemma~\ref{LemmaStayThick}, there exists a constant $\epsilon_T'\approx  0.016$, depending only on $\epsilon_T$, such that for all $t\geq 0 $, any simple closed curve $\alpha$ in $ X_{t}\smallsetminus \Gamma_1$ satisfies 
    \[
        \ell_{X_t}(\alpha) \geq \epsilon_T'.
    \]
    The grafting ray $(X_t)_{t\geq0}$ converges to a noded surface $X_{\infty}$ in the stratum $S(\Gamma_1)$. By continuity of the length function on the augmented Teichmüller space we deduce that 
    \[
        \ell_{X_{\infty}}(\alpha) \geq \epsilon_T'.
    \]
    Therefore $X_{\infty}$ lies in the $\epsilon_T'$-thick part of the stratum $S(\Gamma_1)$.

    The next step is to bound, for large $t\geq 0$, the shear parameters with respect to $\mathcal{T}_{(\Gamma_1,\mathcal{A}_1)}$ on $X_{t}$. By Lemma~\ref{DiazKimLengthGr} for any $\gamma \in \Gamma_1$,
    \[
        \ell_{X_t}(\gamma) \leq \frac{\pi}{\pi+t}\ell_{X_{t_0}}(\gamma) \leq \epsilon_T.
    \]
    By Theorem~\ref{BershearRmk2}, it suffices to bound the truncated length of each orthogeodesic arc $a \in \mathcal{A}_1$ on $X_{t}$. As in proof Proposition~\ref{GraftingShear}, if  $a\in \mathcal{A}$ does not end on curves in $\Gamma$ (that is, it ends in cusps) we obtain 
    \[
        \ell_{X_{t},\Gamma_1}^{\wtrunc}(a)\leq \ell_{X_{t_0},\Gamma_1}^{\wtrunc}(a) \leq L_A''.
    \]  
    Otherwise, the orthogeodesic arc $a$ decomposes into a portion outside the grafting cylinders and one or two horizontal segment(s) $\eta_{t,i}$ inside the grafted cylinders. Fix $0 < \tau < t$ then Lemma~\ref{DiazKimLengthGr2} implies that a horizontal segment $\eta_t$ satisfies

    \begin{equation}\label{length_horizontal_subarc2}
        \frac{1}{2}\ell_{X_t}(\eta_{t,i}) \leq \log\left(\cot \left(\frac{\tau \pi}{2 t}\right)\right) + \tau. 
    \end{equation}
    Let $\gamma \in \Gamma_1$ and recall that by Lemma~\ref{DiazKimLengthGr},
    \[
        \ell_{X_{t}}(\gamma)\leq\frac{\pi}{\pi+t}\ell_{X_{t_0}}(\gamma)\leq \frac{\pi}{\pi+t} \epsilon_T.
    \] thus it follows that 
    \begin{equation}\label{width_collar_comparison2}
        \wtrunc(\ell_{X_t}(\gamma)) \geq \wtrunc\left(\frac{\pi}{\pi+t}\epsilon_T \right)=\arccosh\left(\frac{\pi+t}{\pi}\right)
    \end{equation}
    and 
    \begin{equation}\label{width_std_collar_comparison2}
        \wstd(\ell_{X_t}(\gamma)) \geq \arcsinh\left(\frac{1}{\sinh\left(\frac{\pi}{\pi +t} \epsilon_T \right)}\right).
    \end{equation}
    By comparing the functions on the right hand-side of inequalities (\ref{length_horizontal_subarc2}), (\ref{width_collar_comparison2}) and (\ref{width_std_collar_comparison2}) we obtain that, on $X_t$, the grafted subarc stays inside the standard collar neighborhood of the geodesic representative of its endcurve(s) in $\Gamma$. Consequently, the grafted part contributes to at most $2\Delta_{\wtrunc}$ to the $\wtrunc$-truncated length of the arc $a$ and so for large enough $t$,
    \begin{equation}\label{boundArcTruncatedCVStratum}
         \ell_{X_t, \Gamma_1}^{\wtrunc}(a)\leq \ell_{X_{t_0}, \Gamma_1}^{\wtrunc}(a)+ 2\Delta_{\wtrunc} \leq L_A'' + 2\Delta_{\wtrunc}.
    \end{equation}
   Thus, for large enough $t$, the hexagon decomposition $(\Gamma_1, \mathcal{A}_1)$ is $(\epsilon_T, L_A'' + 2\Delta_{\wtrunc}, \wtrunc)$-short on $X_t$. By Theorem~\ref{BershearRmk2} we conclude that the shear parameters on $X_{t}$ with respect to $\mathcal{T}_{(\Gamma_1,\mathcal{A}_1)}$ are bounded by 
    \[
        44\log(8\pi(2g-2+n))+249.
    \]        
    Roger proved that shearing coordinates extend continuously to the augmented Teichmüller space under pinching~\cite[Proposition 6]{Roger2013}. The spiralling triangulation $\mathcal{T}_{(\Gamma_1,\mathcal{A}_1)}$ is the ideal triangulation ${\mathcal{A}_1}_\infty$. Since the shearing coordinates with respect to $\mathcal{T}_{(\Gamma_1,\mathcal{A}_1)}$ are uniformly bounded on $X_t$, the same bound holds at the limit:
    \[
        \shear_{{\mathcal{A}_1}_\infty}(X_\infty)
        \leq
        44\log\left(8\pi(2g-2+n)\right)+249.
    \]

   Finally we apply Proposition~\ref{WPGraftingRay} to estimate the WP distance between $X$ and $X_{\infty}$:

   \begin{align*}
    \dwp(X,X_{\infty}) &\leq  \dwp(X,X_{t_0}) + \dwp(X_{t_0},X_{\infty})\\
    &\leq 5\sqrt{2\pi(2g-2+n)\log(8\pi(2g-2+n))}.
   \end{align*}

    This completes the proof.
\end{proof}


\section{A combinatorial Weil--Petersson distance from the hexagon graph}\label{sec:hexqi}
In this section we compare the combinatorial geometry of the hexagon graph with the Weil--Petersson geometry of the augmented Teichmüller space. The goal is to show that the following map is a quasi-isometry. Let 
\[
\begin{aligned}
    Q \colon \mathscr{H}_{g,n} &\longrightarrow \overline{\teich}_{g,n},\\
    (\Gamma,\mathcal{A}) &\longmapsto N_{\Gamma,\mathcal{A}},
\end{aligned}
\]
where $N_{\Gamma,\mathcal{A}}$ is the unique noded surface in the stratum $S(\Gamma)$ whose shearing coordinates with respect to the induced ideal triangulation $\mathcal{A}_{\infty}$ are all zero.
Recall that a map between metric spaces is a quasi-isometry if it is both coarsely surjective and if it does not distort distances by at most a multiplicative and an additive constant in both directions.
\begin{definition}
    Let $(X_1,d_1)$ and $(X_2,d_2)$ be metric spaces. A map
    $f:X_1\to X_2$ is a $(C_1,C_2)$-quasi-isometry if:
    \begin{itemize}
        \item[-] there exists $W\geq 0$ such that every point $x'\in X_2$ lies at distance at most $W$ from some point in the image of $f$;
        \[
            d_2(x',f(x))\leq W;
        \]
        \item[-] there exist constants $C_1\geq 1$ and $C_2\geq 0$ such that for all $x,y\in X_1$,
        \[
            \frac{1}{C_1}d_1(x,y)-C_2
            \leq d_2(f(x),f(y))
            \leq C_1d_1(x,y)+C_2 .
        \]
    \end{itemize}
\end{definition}

\subsection{Weil--Petersson width of the hexagon graph}\label{sec:hexwidth}
We first prove that $Q$ is coarsely surjective. Namely, we show that every point of the augmented Teichmüller space is within uniformly bounded Weil--Petersson distance from an image point $N_{\Gamma,\mathcal{A}}$.
More precisely, we estimate the \emph{width} of the hexagon graph model:
\[
    \sup_{X\in\overline{\teich}_{g,n}}
    \inf_{(\Gamma,\mathcal{A})\in\mathscr{H}_{g,n}}
    \dwp(X,Q(\Gamma,\mathcal{A})).
\]
In the previous sections we project a surface $X\in V(\Gamma,\mathcal{A})$ onto the stratum $S(\Gamma)$. This projection has two properties: the Weil--Petersson length of the path is controlled, and the shearing coordinates of the projection point are uniformly bounded, while the truncated lengths of the arcs of $\mathcal{A}_\infty$ on it are of order $\log(g+n)$. Then we show that the canonical noded surface $N_{\Gamma,\mathcal{A}}$ belongs to the region $V(\Gamma,\mathcal{A})$. Finally, since the projection point and $N_{\Gamma,\mathcal{A}}$ carry the same ideal triangulation $\mathcal{A}_\infty$ with truncated arc lengths of order $\log(g+n)$, Theorem~\ref{thm:WolpertShear}, applied on each component of $S(\Gamma)$, bounds the Weil--Petersson distance between them.

Combining these two estimates gives a uniform bound on the Weil--Petersson width of the hexagon graph.

\begin{lemma}\label{Shear0}
    Let $\Sigma$ be a hyperbolic surface of genus $g$ with $n$ punctures such that $2g-2+n>0$. Let $(\Gamma, \mathcal{A})$ be a hexagon decomposition of $\Sigma$. Then $N_{\Gamma, \mathcal{A}}$ lies in $V(\Gamma,\mathcal{A})$.
\end{lemma} 

\begin{proof}
    Since for all $\gamma\in \Gamma$ we have $\ell_{N_{\Gamma, \mathcal{A}}}(\gamma)=0$, to show that $N_{\Gamma, \mathcal{A}} \in V(\Gamma,\mathcal{A})$ it remains to bound the truncated lengths of the ideal arcs in $\mathcal{A}_\infty$. All the shear parameters of $N_{\Gamma, \mathcal{A}}$ with respect to $\mathcal{A}_\infty$ are trivial, thus Lemma~\ref{Shear0-HorocyclicSegments} implies that the $h$-lengths associated with a cusp $c$ are all equal to $\epsilon_T/v(c)$ where $v(c)$ denote the number of arc-ends incident to $c$. 
    Every cusp has valence at most $12g-12+6n$ thus 
    \[
        \frac{\epsilon_T}{12g-12+6n} \leq h(c) \leq h_0(\wtrunc)=\epsilon_T.
    \]
    Combining this estimates with the relations~(\ref{lambda_hyper_length}) and~(\ref{lambda_h_relation}) between $h$-lengths, lambda-lengths and hyperbolic lengths, it follows that for every arc $a \in \mathcal{A}_{\infty}$
    \[
        \ell_{N_{\Gamma, \mathcal{A}}, \Gamma}^{\wtrunc}(a) \leq 2\log\left(\frac{12g-12+6n}{\epsilon_T}\right) < L.
    \]
\end{proof}

\begin{proposition}\label{WidthHexagonModel}
    Let $g,n$ be integers such that $2g-2+n >0$, then for any $X$ in $\overline{\teich}_{g,n}$ there exists a hexagon decomposition $(\Gamma, \mathcal{A})$ such that 
    \[
        \dwp(X, Q(\Gamma, \mathcal{A})) \leq 47\sqrt{(2g-2+n)\log\bigl(8\pi(2g-2+n)\bigr)}.
    \]
\end{proposition}

\begin{proof}
    Let $X\in \overline{\teich}_{g,n}$. By Corollary~\ref{ShortHexagonDecompositionExistence}, there exists a hexagon decomposition $(\Gamma,\mathcal{A})$ such that $X\in V(\Gamma,\mathcal{A})$.
    Next, by Proposition~\ref{TravellingToStratum}, there exists a noded surface $X_\infty\in S(\Gamma)$ such that
    \[
        \shear_{\mathcal{A}_\infty}(X_\infty)
        \leq 11\log(8\pi(2g-2+n))+80
    \]
    and
    \[
        \dwp(X,X_\infty)
        \leq 5\sqrt{2\pi(2g-2+n)\log(8\pi(2g-2+n))}.
    \]
    It remains to bound the Weil-Petersson distance between $X_{\infty}$ and $N_{\Gamma,\mathcal{A}}$. Note that by construction,
    \begin{equation}\label{shear_0_local_center_N_H}
        \shear_{\mathcal{A}_\infty}(N_{\Gamma,\mathcal{A}})=0.
    \end{equation}
    Then we show that the truncated length of any arc in $\mathcal{A}_\infty$ is uniformly bounded on $X_\infty$. Recall that $X_\infty$ is obtained as the geometric limit of the grafting ray $(X_t)_{t\geq 0}$ constructed in the propositions~\ref{GraftingShear} and~\ref{TravellingToStratum}, that is
    \[
        X_\infty=\lim_{t\to\infty}X_t.
    \]
    For any arc $a \in \mathcal{A}$ the proof of Proposition~\ref{GraftingShear} gives: 
    \[
        \ell_{X_t, \Gamma}^{\wtrunc}(a)\leq G(g,n)
    \]
    where
    \begin{align*}
        G(g,n)&=L_A'+2\Delta_{\wtrunc}=L_A+2\log(L)+8+2\Delta_{\wtrunc}
    \end{align*}
    By Lemma~\ref{lem:truncated_length_continuous} the truncated length function is continuous under geometric convergence, therefore we obtain:
    \[
        \ell_{X_\infty, \Gamma}^{\wtrunc}(a)
        =
        \lim_{t\to\infty}\ell_{X_t, \Gamma}^{\wtrunc}(a)
        \leq G(g,n).
    \]
    
    Since $\mathcal{A}_{\infty}$ corresponds to $(\Gamma, \mathcal{A})$ on the stratum $S(\Gamma)$, for every arcs $a \in \mathcal{A}_\infty$ we have:
    \[
        \ell_{X_\infty, \Gamma}^{\wtrunc}(a)\leq G(g,n).
    \]
    We now bound $\dwp(X_\infty,N_{\Gamma,\mathcal{A}})$. Write $m=2g-2+n$ and $\ell_m=\log(8\pi m)\geq\log(8\pi)>3.22$. Both $X_\infty$ and $N_{\Gamma,\mathcal{A}}$ lie in $S(\Gamma)$ and carry the ideal triangulation $\mathcal{A}_\infty$ of $\Sigma\smallsetminus\Gamma$, whose $3m$ arcs correspond to the arcs of $\mathcal{A}$.

    We first compare the two truncations. Let $V\in S(\Gamma)$. The $(\Gamma,\wtrunc)$-truncation removes exactly the neighborhoods of the cusps and nodes of $V$ bounded by the horocycles of length $h_0(\wtrunc)=\epsilon_T$,  hence each end contributes exactly $\log(4\epsilon_T)$, so that every $a\in\mathcal{A}$ satisfies
    \begin{equation}\label{eq:trunc-quarter}
        \ell^{1/4}_V(a)=\ell^{\wtrunc}_{V,\Gamma}(a)+2\log(4\epsilon_T)<\ell^{\wtrunc}_{V,\Gamma}(a)+1.
    \end{equation}
    By the previous inequality, 
    \[
        \Lambda_{\mathcal{A}_\infty}(X_\infty)\leq 3m\bigl(G(g,n)+1\bigr).
    \] 
    By~(\ref{shear_0_local_center_N_H}), the proof of Lemma~\ref{Shear0} gives $\ell^{\wtrunc}_{N_{\Gamma,\mathcal{A}},\Gamma}(a)\leq2\log(6m/\epsilon_T)$ for every $a\in\mathcal{A}$. Hence $\ell^{1/4}_{N_{\Gamma,\mathcal{A}}}(a)\leq2\log(24m)$ by~(\ref{eq:trunc-quarter}), and 
    \[
        \Lambda_{\mathcal{A}_\infty}(N_{\Gamma,\mathcal{A}})\leq 6m\log(24m)\leq 6m\,\ell_m.
    \]

    We now apply Theorem~\ref{thm:WolpertShear}(\ref{thm:WolpertShear-trunc}) on each component $\Sigma_i$ of $\Sigma\smallsetminus\Gamma$, of type $(g_i,n_i)$, with the ideal triangulation induced by $\mathcal{A}_\infty$. By Remark~\ref{eq:WP-product-strata} and since $\sum_i n_i\leq 3m$,
    \begin{align*}
        \dwp(X_\infty,N_{\Gamma,\mathcal{A}})^2&\leq C_W^2\Bigl(3m+\Lambda_{\mathcal{A}_\infty}(X_\infty)+\Lambda_{\mathcal{A}_\infty}(N_{\Gamma,\mathcal{A}})\Bigr)\\
        &\leq 3C_W^2\,m\bigl(G(g,n)+2+2\ell_m\bigr).
    \end{align*}
    Recall that $L=2\ell_m$, $L_A=L+2\Delta_{\wtrunc}$ and $\Delta_{\wtrunc}<2.28$, so that
    \begin{align*}
        G(g,n)+2+2\ell_m&=4\ell_m+2\log(2\ell_m)+10+4\Delta_{\wtrunc}\\
        &\leq 4\ell_m+2\log(2\ell_m)+19.12\leq 11.1\,\ell_m,
    \end{align*}
    where the last inequality holds since $\ell\mapsto 7.1\,\ell-2\log(2\ell)$ is increasing on $[1,\infty)$ and is larger than $19.12$ at $\ell=\log(8\pi)$. Since $C_W^2<33.77$, we obtain
    \[
        \dwp(X_\infty,N_{\Gamma,\mathcal{A}})\leq C_W\sqrt{33.3\,m\,\ell_m}\leq 34\sqrt{m\,\ell_m}.
    \]
    By the triangle inequality, and since $5\sqrt{2\pi}<12.6$,
    \begin{align*}
        \dwp(X,N_{\Gamma,\mathcal{A}})
        &\leq\dwp(X,X_\infty)+\dwp(X_\infty,N_{\Gamma,\mathcal{A}})\leq (12.6+34)\sqrt{m\,\ell_m}\\
        &\leq 47\sqrt{(2g-2+n)\log\bigl(8\pi(2g-2+n)\bigr)}.
    \end{align*}
\end{proof}

\subsection{Proof of the hexagon graph quasi-isometry}\label{sec:hexproof}
In the previous section we proved that the map
\[
\begin{aligned}
    Q \colon \mathscr{H}_{g,n} &\longrightarrow \overline{\teich}_{g,n},\\
    (\Gamma,\mathcal{A}) &\longmapsto N_{\Gamma,\mathcal{A}},
\end{aligned}
\]
is coarsely surjective. It remains to show that $Q$ coarsely preserves distances. More precisely, we prove that adjacent vertices of the hexagon graph are mapped to points at uniformly bounded Weil--Petersson distance, which, together with coarse surjectivity, gives the quasi-isometry statement.

\subsubsection{\texorpdfstring{$Q$}{Q} is Lipschitz} We prove that two points at distance one in the hexagon graph are send by $Q$ to two points at bounded WP distance in the augmented Teichmüller space. First we need the following standard result, see~\cite[Proposition 2.3]{ChekhovPenner2007} which states that a flip move on one ideal quadrilateral only changes the shearing coordinates along the arcs of the quadrilateral.

\begin{lemma}\label{FlipMoveShear0}
Let $\texttt{Q}$ be an ideal quadrilateral contained in an ideal triangulation $\mathcal{T}$. Denote by
$(s_1,s_2,s_3,s_4,s)$ the shearing coordinates of the five edges involved in the quadrilateral, where $s$ corresponds to the diagonal being flipped and $s_1,s_2,s_3,s_4$ to the four boundary edges. After flipping the diagonal, the new shearing coordinates satisfy
\[
    s'=-s,
\]
and
\[
    \begin{aligned}
    s_1'&=s_1+\phi(s),&
    s_2'&=s_2-\phi(-s),\\
    s_3'&=s_3+\phi(s),&
    s_4'&=s_4-\phi(-s),
    \end{aligned}
\]
where
\[
    \phi(x)=\log(1+e^x).
\]
All other shearing coordinates remain unchanged.
\end{lemma}

\begin{proof}
We work in the upper half-plane model. Choose a lift of the ideal quadrilateral with vertices
\[
    (0,1,1+e^s,\infty),
\]
so that the diagonal with shear parameter $s$ is the geodesic joining $1$ and $\infty$. The flip replaces this diagonal by the geodesic joining $0$ and $1+e^s$. See Figure~\ref{Flipmove}.
\begin{figure}[h]
    \centering
    \begin{overpic}[width=0.9\linewidth,keepaspectratio]{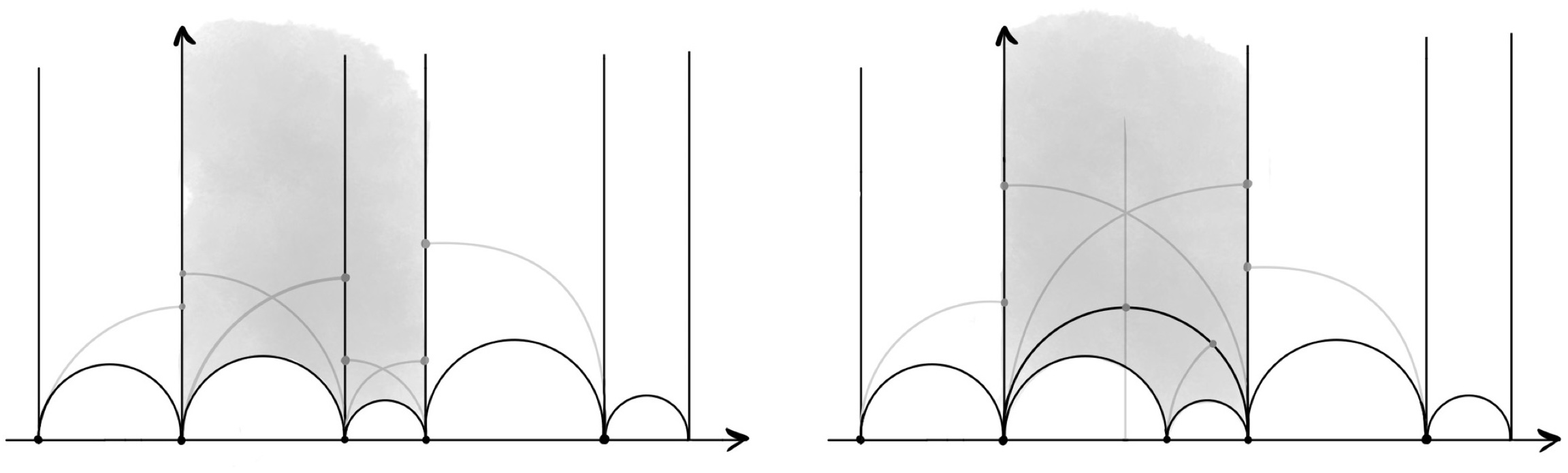}
    \put(0,-1){\tiny$-e^{-s_1}$}
    \put(10.8,-1){\tiny$0$}
    \put(21.3,-1){\tiny$1$}
    \put(24,-1){\tiny$1+e^s$}
    \put(35,-1){\tiny$1+e^s+e^{s+s_2}$}
    \put(18,25){\color{Gray}$Q$}
    \put(10.3,13){\tiny\color{Gray}$i$}
    \put(6,10){\tiny{\color{Gray}$ie^{-s_1}$}}
    \put(52,-1){\tiny$-e^{-s_1}$}
    \put(63.3,-1){\tiny$0$}
    \put(73.7,-1){\tiny$1$}
    \put(77,-1){\tiny$1+e^s$}
    \put(87,-1){\tiny$1+e^s+e^{s+s_2}$}
    \put(55.3,19){\tiny\color{Gray}$i(1+e^s)$}
    \put(58,11){\tiny{\color{Gray}$ie^{-s_1}$}}
    \put(70,11){\tiny$A$}
    \put(77.6,8){\tiny{$B$}}
    
    \end{overpic}
    \caption{A flip move in a hyperbolic ideal quadrilateral.}\label{Flipmove}
\end{figure}
We first compute the new shears on the four boundary edges. Direct computations gives
\[
    s_1'= \log\left(\frac{1+e^s}{e^{-s_1}}\right) =
    s_1+\log(1+e^s),
\]
and
\[
    s_2' = \log\left(\frac{e^{s+s_2}}{1+e^s}\right)
    = s_2-\log(1+e^{-s}).
\]
The same computation on the opposite sides gives
\[
    s_3'=s_3+\log(1+e^s), \qquad s_4'=s_4-\log(1+e^{-s}).
\]
It remains to compute the shear parameter of the new diagonal. Let $A$ and $B$ be the two shear points along the new diagonal. Hence $A$ is the intersection point of the geodesic line $x=(1+e^s)/2$ with the circle $\texttt{C}(c,r)$ of center $c=(1+e^s)/2$ and radius $r=(1+e^s)/2$, see Figure~\ref{Flipmove}, thus 
\[
    A=\frac{1+e^s}{2}+ i\frac{1+e^s}{2}.
\] 
Then, $B$ is the orthogonal intersection point of the circle $\texttt{C}(c,r)$ with the geodesic circle passing through $1$, we obtain 
\[
    B=\frac{1+e^s}{1+e^{2s}}+ i\frac{e^s+e^{2s}}{1+e^{2s}}.
\] 
The Möbius transformation 
\[
    z\mapsto \frac{z}{1+e^s-z}
\]
maps $\mathcal{C}(c,r)$ to the vertical line, $0\mapsto 0$ and $1+e^s\mapsto \infty$. Computing the signed distance between the images of $A$ and $B$ gives
\[
    d_{\text{hyp}}(A,B)=-s .
\]
Hence the new diagonal has shear parameter
\[
    s'=-s
\]
which proves the lemma.
\end{proof}

\begin{lemma}\label{FlipMoveWPDist}
    If $\mathcal{H},\mathcal{H}' \in \mathscr{H}_{g,n}$ are two hexagon decompositions related by a flip move, then 
    \[
        \dwp(Q(\mathcal{H}), Q(\mathcal{H}')) \leq 35\sqrt{(2g-2+n)\log\bigl(8\pi(2g-2+n)\bigr)}.
    \]
\end{lemma}

\begin{proof}
    Let $\mathcal{H}=(\Gamma, \mathcal{A})$ and $\mathcal{H}'=(\Gamma, \mathcal{A}')$ be two vertices in $\mathscr{H}_{g,n}$ that differ by a single flip move. 
    We want to estimate the Weil--Petersson distance between $N_{\Gamma, \mathcal{A}}$ and $N_{\Gamma, \mathcal{A}'}$. By construction, the shearing coordinates of $N_{\Gamma,\mathcal{A}'}$ with respect to the ideal triangulation $\mathcal{A}_{\infty}'$ are trivial: 
    \[
        \shear_{\mathcal{A}_{\infty}'}(N_{\Gamma,\mathcal{A}'})=0.
    \]
    Similarly, one has $\shear_{\mathcal{A}_{\infty}}(N_{\Gamma,\mathcal{A}})=0$, by Lemma~\ref{FlipMoveShear0}, the flip move changes only the shear parameters on the ideal quadrilateral supporting the flipped edge. Hence, the shearing coordinates of $N_{\Gamma,\mathcal{A}}$, with respect to $\mathcal{A}_{\infty}'$, remain trivial outside this quadrilateral and along the flipped arc, while on the boundary of this quadrilateral, the shearing coordinates are
    \[
    (\log 2,-\log 2,\log 2,-\log 2).
    \] 
    Therefore the oscillation norm $\mathcal{O}_{\mathcal{A}_{\infty}'}(N_{\Gamma, \mathcal{A}})$ is uniformly bounded independently of the topology. More precisely, each of the four sides of the quadrilateral contributes $\pm\log2$ to the shear of the corresponding arc of $\mathcal{A}'_\infty$, and all the other contributions vanish. Hence the sum of the absolute values of the shears of $N_{\Gamma,\mathcal{A}}$ over all the arc-ends of $\mathcal{A}'_\infty$ is at most $2\cdot4\log 2$, so that $\mathcal{O}_{\mathcal{A}_{\infty}'}(N_{\Gamma, \mathcal{A}})\leq 8\log2$, while $\mathcal{O}_{\mathcal{A}_{\infty}'}(N_{\Gamma, \mathcal{A}'})=0$. Write $m=2g-2+n$ and $\ell_m=\log(8\pi m)$, and let $\Sigma_i$ be the components of $\Sigma\smallsetminus\Gamma$, of types $(g_i,n_i)$, with $m_i=2g_i-2+n_i$, so that $\sum_i m_i=m$. Both surfaces lie in $S(\Gamma)$, so Remark~\ref{eq:WP-product-strata} and Theorem~\ref{thm:WolpertShear}(\ref{thm:WolpertShear-osc}) applied on each $\Sigma_i$ give
    \[
        \dwp(N_{\Gamma, \mathcal{A}}, N_{\Gamma, \mathcal{A}'})^2\leq 21^2\sum_i m_i\bigl(\log(24m_i)+8\log2\bigr)\leq 21^2\, m\bigl(\log(24m)+8\log2\bigr).
    \]
    Since $\log(24m)\leq\ell_m$ and $8\log2<1.73\,\ell_m$, we obtain
    \[
        \dwp(N_{\Gamma, \mathcal{A}}, N_{\Gamma, \mathcal{A}'})\leq 21\sqrt{2.73\,m\,\ell_m}\leq 35\sqrt{m\,\ell_m}.
    \]
\end{proof}

The edges of $\Hs_{g,n}$ are of two types, and Lemma~\ref{FlipMoveWPDist} covers the first. We now bound the Weil--Petersson cost of the second, an adding curve move, which turns out to be the more expensive of the two.

\begin{lemma}\label{LIpschADdingcurve}
    If $\mathcal{H},\mathcal{H}'\in\mathscr{H}_{g,n}$ are two hexagon decompositions related by a curve addition move, then
    \[
        \dwp(Q(\mathcal{H}),Q(\mathcal{H}'))\lesssim \sqrt{(2g-2+n)\log\bigl(8\pi(2g-2+n)\bigr)}.
    \]
\end{lemma}
    
\begin{proof}
    Let $\mathcal{H}=(\Gamma,\mathcal{A})$ and $\mathcal{H}'=(\Gamma',\mathcal{A}')$ be vertices in $\mathscr{H}_{g,n}$ such that $\mathcal{H}'$ is obtained from $\mathcal{H}$ by adding a compatible simple closed curve $\alpha$, that is
    \[
        \Gamma'=\Gamma\cup\{\alpha\} .
    \]
    Let $N_{\Gamma,\mathcal{A}}$ and $N_{\Gamma',\mathcal{A}'}$ denote the corresponding noded hyperbolic surfaces. First, we estimate the length of $\alpha$ on $N_{\Gamma,\mathcal{A}}$. The curve $\alpha$ has a representative $\alpha^*$ disjoint from the cusp $\wtrunc$-neighborhoods. In addition, $\alpha$ is compatible with $\mathcal{H}$, thus $\alpha$ intersects each arc $a \in \mathcal{A}$ at most once. Moreover, since $\mathcal{H}$ is short on $N_{\Gamma,\mathcal{A}}$, the $(\Gamma,\wtrunc)$-truncated length of $a$ is bounded from above by $L_A$. Consequently, $\alpha$ decomposes into at most $6g-6+3n$ segments, each supported on a hexagon in $\mathcal{H}$. Hence, each such segment has length at most $2L_A+2\epsilon_T$. Therefore,
    \[
        \ell_{N_{\Gamma,\mathcal{A}}}(\alpha) \leq \mathcal{L} \coloneq (6g-6+3n)(2L_A+2\epsilon_T).
    \]

    Set
    \[
        X_{t_0}=\gr_{t_0\alpha}(N_{\Gamma,\mathcal{A}}), \qquad t_0=\pi\left(\frac{\mathcal{L}}{\epsilon_T}-1\right).
    \]
    We show that $(\Gamma', \mathcal{A}')$ is $(\Gamma', \wtrunc)$-short on $X_{t_0}$. By Lemma~\ref{DiazKimLengthGr},
    \[
        \ell_{X_{t_0}}(\alpha) \leq \frac{\pi}{\pi+t_0}\ell_{N_{\Gamma,\mathcal{A}}}(\alpha) \leq \epsilon_T,
    \]
    while the remaining curves in $\Gamma'$ are already pinched. Next, by definition (see Section~\ref{definition_hexagon_graph}), the arcs in $\mathcal{A}'$ are obtained by cutting the arcs in $\mathcal{A}$ at their intersection points with $\alpha$. The arcs $a\in \mathcal{A}$ have bounded truncated length on $N_{\Gamma,\mathcal{A}}$, 
    \[ 
        \ell_{N_{\Gamma,\mathcal{A}}, \Gamma}^{\wtrunc}(a)\leq L_A. 
    \]
    The hyperbolic metric on $X_{t_0}$ is smaller than the Thurston metric on $\Grr_{t_0 \alpha}(N_{\Gamma,\mathcal{A}})$ which coincides outside the grafting cylinder with the hyperbolic metric on $N_{\Gamma,\mathcal{A}}$, see for instance~\cite{Tanigawa1997,McMullen1998}. Hence if an orthogeodesic arc $a \in \mathcal{A}'$ does not meet $\alpha$, its truncated length satisfies
    \[
        \ell_{X_{t_0}, \Gamma'}^{\wtrunc}(a)\leq \ell_{N_{\Gamma,\mathcal{A}}, \Gamma}^{\wtrunc}(a) \leq L_A.
    \]
    Otherwise, $a$ meets $\alpha$ and so $a$ is prolonged, via the grafting operation, with an horizontal arc from the core curve of the grafting cylinder to its boundary that we denote by $\eta_{t_0}$. Since the hyperbolic metric on $X_{t_0}$ is smaller than the Thurston metric we obtain:
    \[
        \ell_{X_{t_0}, \Gamma'}^{\wtrunc}(a)
        \leq \ell_{N_{\Gamma,\mathcal{A}}, \Gamma}^{\wtrunc}(a) +\frac{1}{2}\ell_{X_{t_0}}(\eta_{t_0}).
    \]
    By Lemma~\ref{DiazKimLengthGr2}, 
    \begin{equation}\label{ort-geo-arc-bound}
        \frac{1}{2}\ell_{X_{t_0}}(L_{{t_0}})\leq \tau + \log\left(\frac{\cos\left(\frac{\tau \pi}{2 t_0}\right)}{\sin\left(\frac{\tau \pi}{2 t_0}\right)}\right)= \tau + \log\left(\cot\left(\frac{\tau \epsilon_T}{2(\mathcal{L}-\epsilon_T)}\right)\right),
    \end{equation}
    where $\tau$ is a fixed positive number $<t_0$, take $\tau=1$. Note that $2g-2+n >0$ and so $\mathcal{L}>2.8$. Hence the right-hand side in~(\ref{ort-geo-arc-bound}) is bounded by: 
    \[
        \frac{1}{2}\ell_{X_{t_0}}(\eta_{t_0})\leq \tau + \log(\mathcal{L}) + 3,
    \] 
    and so we obtain the following bound: 
    \[
        \ell_{X_{t_0},\Gamma}^{\wtrunc}(a)\leq\ell_{X,\Gamma}^{\wtrunc}(a) + 2\log(\mathcal{L})+ 8 \leq L_A'.
    \]
    where $L_A'= L_A+2\log(\mathcal{L})+8$. Therefore the hexagon decomposition $(\Gamma, \mathcal{A})$ is $(\epsilon_T, L_A', \wtrunc)$-short on $X_{t_0}$.
    Consider the grafting ray $(X_t)_{t\geq 0}$ defined by
    \[
        X_{t}=\gr_{t\alpha}(X_{t_0}).
    \]
    The next step is to bound, for large $t\geq 0$, the shear parameters with respect to $\mathcal{T}_{\Gamma',\mathcal{A}'}$ on $X_{t}$. By Lemma~\ref{DiazKimLengthGr} we have $\ell_{X_t}(\alpha) \leq \epsilon_T$, moreover, the remaining curves in $\Gamma'$ are already pinched. By Theorem~\ref{BershearRmk2}, it remains to bound the truncated length of each orthogeodesic arc $a \in \mathcal{A}'$ on $X_{t}$. As in proof Proposition~\ref{GraftingShear}, if  $a\in \mathcal{A}$ does not end on curves in $\Gamma$ (that is, it ends in cusps) we obtain 
    \[
        \ell_{X_{t},\Gamma'}^{\wtrunc}(a)\leq \ell_{X_{t_0},\Gamma'}^{\wtrunc}(a) \leq L_A'.
    \]  
    Otherwise, the orthogodesic arc $a$ decomposes into a portion outside the grafting cylinders and one horizontal segment $L_{t}$ inside the grafted cylinder. Fix $0 < \tau < t$ then Lemma~\ref{DiazKimLengthGr2} implies that a horizontal segment $L_{t}$ satisfies

    \begin{equation}\label{length_horizontal_subarc3}
        \frac{1}{2}\ell_{X_t}(L_{t}) \leq \log\left(\cot \left(\frac{\tau \pi}{2 t}\right)\right) + \tau. 
    \end{equation}
  Recall that by Lemma~\ref{DiazKimLengthGr},
    \[
        \ell_{X_{t}}(\alpha)\leq\frac{\pi}{\pi+t}\ell_{X_{t_0}}(\alpha)\leq \frac{\pi}{\pi+t} \epsilon_T.
    \] thus it follows that 
    \begin{equation}\label{width_collar_comparison4}
        \wtrunc(\ell_{X_t}(\alpha)) \geq \wtrunc\left(\frac{\pi}{\pi+t}\epsilon_T \right)=\arccosh\left(\frac{\pi+t}{\pi}\right)
    \end{equation}
    and 
    \begin{equation}\label{width_std_collar_comparison5}
        \wstd(\ell_{X_t}(\alpha)) \geq \arcsinh\left(\frac{1}{\sinh\left(\frac{\pi}{\pi +t} \epsilon_T \right)}\right).
    \end{equation}
    By comparing the functions on the right hand-side of inequalities (\ref{length_horizontal_subarc3}), (\ref{width_collar_comparison4}) and (\ref{width_std_collar_comparison5}) we obtain that, the grafted subarc stays inside the standard collar neighborhood of the geodesic representative of $\alpha$ on $X_t$. Consequently, the grafted part contributes to at most $2\Delta_{\wtrunc}$ to the $\wtrunc$-truncated length of the arc $a$ and so for large enough $t$,
    \[
        \ell_{X_t, \Gamma}^{\wtrunc}(a)\leq \ell_{X_{t_0}, \Gamma}^{\wtrunc}(a) + 2\Delta_{\wtrunc} \leq L_A' + 2\Delta_{\wtrunc}.
    \]
   Thus, for large enough $t$, the hexagon decomposition $(\Gamma', \mathcal{A}')$ is $(\epsilon_T, L_A' + 2\Delta_{\wtrunc}, \wtrunc)$-short on $X_t$. By Theorem~\ref{BershearRmk2} we conclude that the shear parameters on $X_{t}$ with respect to $\mathcal{T}_{(\Gamma',\mathcal{A}')}$ satisfy
    \[
        \shear_{\mathcal{T}_{(\Gamma',\mathcal{A}')}}(X_{t}) \lesssim \log(8\pi(2g-2+n)).
    \]        
    Roger proved that shearing coordinates extend continuously to the augmented Teichmüller space under pinching~\cite[Proposition 6]{Roger2013}. The spiralling triangulation $\mathcal{T}_{(\Gamma',\mathcal{A}')}$ is the ideal triangulation ${\mathcal{A}'}_\infty$. Since the shearing coordinates with respect to $\mathcal{T}_{(\Gamma',\mathcal{A}')}$ are uniformly bounded on $X_t$, the same bound holds at the limit:
    \begin{equation}\label{bound_shear_proof_LIpsh}
        \shear_{{\mathcal{A}'}_\infty}(X_\infty)
        \lesssim \log(8\pi(2g-2+n)).
    \end{equation}
   Finally it remains to estimate the Weil--Petersson distance: 
    \[
        \dwp(N_{\Gamma, \mathcal{A}},N_{\Gamma', \mathcal{A}'}) \leq \dwp(N_{\Gamma, \mathcal{A}},X_{\infty}) + \dwp(X_{\infty},N_{\Gamma', \mathcal{A}'}).
    \]
    Recall that $ \ell_{N_{\Gamma, \mathcal{A}}}(\alpha)\leq\mathcal{L}$ therefore Proposition~\ref{WPGraftingRay} gives
    \begin{align*}
        \dwp(N_{\Gamma, \mathcal{A}},X_{\infty}) &\leq 2^{5/4}\sqrt{\pi \mathcal{L}}.
    \end{align*}
    Write $m=2g-2+n$. Since $L_A\lesssim\log(8\pi m)$, we have \[
        \mathcal{L}=3m(2L_A+2\epsilon_T)\lesssim m\log(8\pi m),
    \]
    so that $\dwp(N_{\Gamma, \mathcal{A}},X_{\infty})\lesssim\sqrt{m\log(8\pi m)}$ and $L_A'=L_A+2\log(\mathcal{L})+8\lesssim\log(8\pi m)$. We have shown that $(\Gamma', \mathcal{A}')$ is $(\epsilon_T, L_A' + 2\Delta_{\wtrunc}, \wtrunc)$-short on $X_t$ for large enough $t$. As $t\to\infty$, $X_t$ converges to $X_\infty\in S(\Gamma')$~\cite{hensel2008}, and the truncated length functions are continuous on $\bigcup_{\Gamma^0\subseteq\Gamma'}S(\Gamma^0)$ (Lemma~\ref{lem:truncated_length_continuous}). Hence every $a\in\mathcal{A}'$ satisfies $\ell_{X_\infty, \Gamma'}^{\wtrunc}(a)\leq L_A'+2\Delta_{\wtrunc}$. Both $X_\infty$ and $N_{\Gamma',\mathcal{A}'}$ lie in $S(\Gamma')$ and carry the ideal triangulation $\mathcal{A}'_\infty$, and $\shear_{\mathcal{A}_{\infty}'}(N_{\Gamma', \mathcal{A}'})=0$. As in the proof of Proposition~\ref{WidthHexagonModel}, inequality~(\ref{eq:trunc-quarter}) and the proof of Lemma~\ref{Shear0} give
    \[
        \Lambda_{\mathcal{A}'_\infty}(X_\infty)\leq 3m\bigl(L_A'+2\Delta_{\wtrunc}+1\bigr)
        \qquad\text{and}\qquad
        \Lambda_{\mathcal{A}'_\infty}(N_{\Gamma',\mathcal{A}'})\leq 6m\log(24m),
    \]
    so that Remark~\ref{eq:WP-product-strata} and Theorem~\ref{thm:WolpertShear}(\ref{thm:WolpertShear-trunc}), applied on each component of $\Sigma\smallsetminus\Gamma'$, give
    \begin{align*}
        \dwp(X_\infty,N_{\Gamma', \mathcal{A}'}) &\leq C_W\sqrt{3m+3m\bigl(L_A'+2\Delta_{\wtrunc}+1\bigr)+6m\log(24m)}\\
        &\lesssim \sqrt{m\log(8\pi m)}.
    \end{align*}
    Consequently
    \[
        \dwp(N_{\Gamma, \mathcal{A}},N_{\Gamma', \mathcal{A}'})
        \lesssim \sqrt{(2g-2+n)\log\bigl(8\pi(2g-2+n)\bigr)},
    \]
    which concludes the proof.
\end{proof}

From Lemma~\ref{FlipMoveWPDist} and Lemma~\ref{LIpschADdingcurve} we obtain:

\begin{proposition}\label{QLipschitz}
    The map $Q$ is $K(g,n)$-Lipschitz with 
    \[
        K(g,n)\lesssim \sqrt{(2g-2+n)\log\bigl(8\pi(2g-2+n)\bigr)}.
    \]
\end{proposition}

\subsubsection{\texorpdfstring{$Q$}{Q} is a quasi-isometry}
Here we follow Brock's strategy to prove that $Q$ is a quasi-isometry, we develop some tools to adapt the proof.

\begin{lemma}\label{TruncOrthoSpectrum}
    Let $X\in \teich_{g,n}$ with $2g-2+n>0$ and $L>0$. Let $\Gamma$ be a simple multicurve on $X$ and $\wabstract$ a truncation function. Then
    \[
        \#\left\{
        \begin{array}{c}
        a \text{ is simple orthogeodesic arc on }X\\
        \text{ and } \interior(a) \text{ is disjoint from } \Gamma
        \end{array}
        \;\middle|\;
        \ell_{X,\Gamma}^{\wabstract}(a)\le L
        \right\}
        \text{ is finite.}
        \]
    where $\sharp$ denotes cardinality. In other words, the truncated length spectrum is discrete.
\end{lemma}

\begin{proof}
    Recall that the length spectrum of simple geodesic is discrete, see for instance~\cite[Lemma 12.4]{PrimerBook}. The length spectrum of orthogeodesics on surfaces with geodesic boundary is discrete, see~\cite[Theorem 1.1]{Basmajian1993Spectrum}. We use similar arguments to prove the statement of the lemma.
    Let $X\in \teich_{g,n}$ with $(c_i)_{i=0}^n$ its collection of cusps. Then $X=\Hh^2/G$, where $G\subset \PSL_2(\R)$ is a discrete torsion-free group. Recall that the $(\Gamma, \wabstract)$-truncation of $a$ is defined as: 
    \[ 
        a_{\Gamma}^{\wabstract} = a \cap X_{\Gamma}^{\wabstract}.  
    \]

    \begin{claim}\label{claim:proof_truncated_length_spectrum_discrete}
        The portion of $a$ removed by the truncation, denoted by
        \[
            a^t \coloneq a \cap \left( \bigcup_{i=1}^n \interior\left( N_{h_0(\wabstract)}(c_i)\right) \cup \bigcup_{\gamma\in\Gamma} \interior\left( \Cs_{\wabstract(\ell_X(\gamma))}(\gamma) \right) \right),
        \] 
        consists of exactly two connected components. Each component of $a^t$ is either a single terminal ray inside a cusp neighborhood $N_{h_0(\wabstract)}(c_i)$ or a single geodesic segment inside a collar neighborhood $\Cs_{\wabstract(\ell_X(\gamma))}(\gamma)$.
    \end{claim}
    \begin{proofclaim}
        First assume that $a\cap \interior(N_{h_0(\wabstract)}(c_i)) \neq \emptyset$ for some cusp $c_i$. Consider a lift $\tilde{a}$ of $a$ to  $\Hh^2$ such that the cusp $c_i$ lifts to $\infty$ and its cusp neighborhood $N_{h_0(\wabstract)}(c_i)$ corresponds to $\{z \in \Hh^2 : \Im(z) \geq 1/h_0 \}$. We can assume the stabilizer of this cusp is generated by the parabolic element $T:z \mapsto z+1$. Recall that by the definition of a truncation function, we have $h_0(\wabstract) \leq 2$. \begin{itemize}
            \item[-] If $\tilde{a}$ has $\infty$ as an endpoint, then it is a vertical line. 
            In this case, $\tilde{a} \cap \{z \in \Hh^2 : \Im(z) > 1/h_0 \}$ is a single geodesic ray, and thus its projection $a \cap \interior(N_{h_0(\wabstract)}(c_i))$ is a single terminal ray. 
            \item[-] Otherwise, $\tilde{a}$ does not have $\infty$ as an endpoint, meaning it is a Euclidean semi-circle of radius $R$. 
            Because $\tilde{a} \cap \{z \in \Hh^2: \Im(z) > 1/h_0 \} \neq \emptyset$, its maximum height is strictly greater than $1/h_0$. 
            We must therefore have $R > 1/h_0 \geq 1/2$, and thus $2R > 1$. 
            The translated geodesic $T\tilde{a}$ is another lift of $a$ that is distinct from $\tilde{a}$. 
            Hence, $\tilde{a}$ and $T\tilde{a}$ are both semi-circles with radius $R$ and centers lying on the real axis $\partial \Hh^2$ at distance $1$. 
            Since the distance between their centers is strictly less than $2R$, they must intersect in $\Hh^2$. 
            This implies that the projection $a$ has a self-intersection on $X$, which is a contradiction since $a$ is simple. 
            Thus, any intersection $a\cap \interior(N_{h_0(\wabstract)}(c_i))$ must be a terminal ray. 
        \end{itemize} 
        
        Now assume that there exists some $\gamma \in \Gamma$ such that $a\cap \interior\left( \Cs_{\wabstract(\ell_X(\gamma))}(\gamma) \right) \neq \emptyset$. 
        Consider a lift $\tilde{\gamma}$ of $\gamma$ to the imaginary axis $i\R_{>0}$ in $\Hh^2$.  
        The collar neighborhood $ \Cs_{\wabstract(\ell_X(\gamma))}(\gamma)$ lifts to an infinite circle sector bounded by two straight euclidean rays through the origin.
        \[ 
            S = \{ z \in \Hh^2 : \dhyp(z,i\R_{>0}) \leq  \wabstract(\ell) \}. 
        \] 
        The stabilizer of $\tilde{\gamma}$ is generated by $T: z \mapsto e^\ell z$, where $\ell = \ell_X(\gamma)$.
        Recall that by the definition of a truncation function, $\wabstract(\ell) \leq w^{std}(\ell)$. 
        Furthermore, the standard collar of width $w^{std}(\ell)$ is an embedded cylinder.
        \begin{itemize}
            \item[-] If $\tilde{a}$ has an endpoint on $\tilde{\gamma}$, then because $a$ is an orthogeodesic arc, $\tilde{a}$ is orthogonal to $\tilde{\gamma}$. 
            In this upper half-plane model, $\tilde{a}$ is a Euclidean semi-circle centered at the origin. 
            Its intersection with the sector $S$ is a single connected geodesic segment, which projects to a single geodesic segment inside the collar neighborhood. 
            \item[-]If $\tilde{a}$ does not intersect $\tilde{\gamma}$, then $\tilde{a}$ is a Euclidean semi-circle contained entirely in the right (or left) half-plane. 
            Let $u$ and $v$ be the endpoints of $\tilde{a}$ on the real axis, with $0 < u < v$. 
            The hyperbolic distance $d$ from $\tilde{a}$ to the imaginary axis is achieved for $\cosh(d) = \frac{v+u}{v-u}$. 
            Because $\tilde{a}$ intersects the interior of the collar, we must have $d < \wabstract(\ell) \le w^{std}(\ell)$. 
            Recall that $w^{std}(\ell) = \operatorname{arcsinh}\left(\frac{1}{\sinh(\ell/2)}\right)$. 
            This implies that $\cosh(w^{std}(\ell)) = \coth(\ell/2) = \frac{e^\ell+1}{e^\ell-1}$. 
            From $\cosh(d) < \cosh(w^{std}(\ell))$, we obtain:
            \[
                \frac{v+u}{v-u} < \frac{e^\ell+1}{e^\ell-1},
            \]
            which simplifies to $v > e^\ell u$. 
            The translated lift $T\tilde{a}$ is a semi-circle with endpoints $e^\ell u$ and $e^\ell v$. 
            Since $0 < u < e^\ell u < v < e^\ell v$, the intervals $(u,v)$ and $(e^\ell u, e^\ell v)$ overlap but are not contained in one another. 
            This means that the semi-circles $\tilde{a}$ and $T\tilde{a}$ intersect in $\Hh^2$, which again implies that $a$ self-intersects on $X$, contradicting that $a$ is simple. 
        \end{itemize}
        Therefore, $a$ can only enter a $\wabstract$-cusp or a $\wabstract$-collar neighborhood around a curve in $\Gamma$ if it terminates there. 
        Recall that by the Collar Lemma (see Theorem~\ref{CollarThm}), the $\wabstract$-collar and $\wabstract$-cusp neighbourhoods are disjoint. Because $a$ has exactly two endpoints (which lie on $\Gamma$ or in the cusps) and $\interior(a)$ does no intersects $\Gamma$, the removed portion $a^t$ consists of exactly two connected components, each being a terminal ray in a cusp neighborhood or a geodesic segment in a collar neighborhood.
    \end{proofclaim}
    To unify notation, let $(b_k)$ be the collection of all cusps $c_i$ and curves $\gamma_j \in \Gamma$. Choose one lift for each $b_k$, denoted by $\tilde{b}_k$. Specifically, let $\xi_i \in \partial \Hh^2$ be the lift of $c_i$, and $\tilde{\gamma}_j \subset \Hh^2$ be the lift of $\gamma_j$. 
    
    Let $R_k$ be the lifted truncation region corresponding to $\tilde{b}_k$. This region is either a horoball $B_i$ bounded by the lift of $\partial N_{h_0(\wabstract)}(c_i)$ at $\xi_i$, or an equidistant region $W_j$ bounded by the lift of $\partial \Cs_{\wabstract(\ell_X(\gamma_j))}(\gamma_j)$ around $\tilde{\gamma}_j$. Let $e_k \in G$ be the generator of the stabilizer of $\tilde{b}_k$. This generator is a parabolic element $g_i$ for $\xi_i$, or a hyperbolic element $f_j$ for $\tilde{\gamma}_j$. Let $F_k \subset \partial R_k$ be a compact fundamental domain for the action of $\langle e_k \rangle$ on $\partial R_k$.

    Fix a pair of boundary elements $b_k$ and $b_l$. Consider the set of all oriented simple orthogeodesic arcs $a$ from $b_k$ to $b_l$ whose interiors are disjoint from $\Gamma$ and whose truncated lengths satisfy $\ell_{X,\Gamma}^{\wabstract}(a) \le L$. Lift such an arc $a$ to a geodesic $a'$ in $\Hh^2$ so that it originates at $\tilde{b}_k$ and its intersection with $\partial R_k$ lies inside $F_k$. Such a lift shall be refered to as the \emph{normalized lift} of $a$. The terminal endpoint of this lift $a'$ lies on a translated region $g R_l$ for some element $g \in G$. By applying an element of the stabilizer $\langle g e_l g^{-1} \rangle$, we can further assume that the intersection of $a'$ with $g \partial R_l$ lies in the translated fundamental domain $g F_l$. This process associates to each such arc $a$ a unique element $g_a\in G$. 
    \[
        \begin{aligned}
            \left\{
                \begin{array}{c}
                    a \text{ is an oriented simple orthogeodesic (o.s.o.) arc}\\
                    \text{on } X,\;\interior(a)\cap\Gamma=\varnothing,\;
                    a \text{ from }b_k \text{ to } b_l
                \end{array}
            \right\}
            &\longrightarrow G,\\[4pt]
            a&\longmapsto g_a .
        \end{aligned}
    \]
    This map is injective because an oriented orthogeodesic arc on $X$ is uniquely determined by the endpoints of its normalized lift in $\Hh^2$. By Claim \ref{claim:proof_truncated_length_spectrum_discrete}, $a_{\Gamma}^{\wabstract}$ is a single compact geodesic segment, whose normalized lift has one endpoint $\tilde{x}$ on $F_k$ and another endpoint $\tilde{y}$ on $F_l$. Its truncated length $\ell_{X,\Gamma}^{\wabstract}(a)$ then equals the hyperbolic distance $d_{\Hh^2}(\tilde{x},\tilde{y})$. Hence the following map is injective:
    \[
        \left\{
            \begin{array}{c}
                a \text{ is an o.s.o. arc on } X,\\
                \interior(a)\cap\Gamma=\varnothing,\;
                a \text{ from }b_k \text{ to } b_l
            \end{array}
            \;\middle|\;
            \ell_{X,\Gamma}^{\wabstract}(a) \leq L
        \right\}
        \longrightarrow
        \left\{
            g_a \in G
            \;\middle|\;
            \dhyp(F_k,g_a F_l)\leq L
        \right\}.
    \]
    Since the group $G \subset \Isom^+(\Hh^2)$ is discrete and because $F_k$ and $F_l$ are compact, the set 
    \[
        \left\{ g \in G \;\middle|\; d_{\Hh^2}(F_k, g_a F_l) \le L \right\}
    \]
    is finite. Since the map is injective, the number of such arcs from $b_k$ to $b_l$ is finite. Summing over the finitely many pairs $(k,l)$ (with $k=l$ allowed), and up to a factor of two (as orthogeodesic arcs are unoriented), yields a finite number, which completes the proof.
\end{proof}

Let $L'>L$ and $L_A'>L_A$. Let $\mathcal{H}=(\Gamma, \As)$. Define \[V_{L',L_A'}(H)=\left\{ X \in \overline{\teich}_{g,n} \mid \ell_X(\gamma) \leq L',\; \ell_{X,\Gamma}^{\wtrunc}(a)\leq L_A' \right\}.\]

\begin{lemma}\label{CountSHortHExagonDecomp}
    Let $L'\geq L$, $L_A'\geq L_A$ and let $\Sigma$ be a surface of type $(g,n)$ with $2g-2+n>0$. There is a constant $N=N(L',L_A',g,n)$ such that the following property holds: for any $X\in\overline{\teich}_{g,n}$, there are at most $N$-many $(L',L_A',\wtrunc)$-short hexagon decompositions on $X$.
\end{lemma}

\begin{proof}
    First, we define a continuous bump function that will allow to count hexagon decompositions in a continuous manner. For $L>0$, let
    \[
        \begin{aligned}
           f_L\colon \R_{\geq 0}\cup\{\infty\} &\longrightarrow \R_{\geq 0} \\
            \ell &\longmapsto \begin{cases}
                1 & \text{if $\ell \leq L$},\\
                -\ell+L+1 & \text{if $L \leq \ell \leq L+1$},\\
                0 & \text{if $\ell \geq L+1$}.
            \end{cases}
        \end{aligned}
    \]
    In particular, $f_L(\infty)=0$.

    Let $X\in \overline{\teich}_{g,n}$, for any hexagon decomposition $(\Gamma, \As)$ of $X$, define  
    \[
        f_L(\Gamma) \coloneq \prod_{\gamma\in \Gamma} f_L(\ell_X(\gamma)) \quad \text{and} \quad f_L(\mathcal{A}) \coloneq \prod_{a\in \mathcal{A}} f_L\left(\ell_{X,\Gamma}^{\wtrunc}(a)\right).
    \]
    Consider the following function defined over the set of all hexagon decompositions on $X$:
    \[
        \begin{aligned}
           F\colon \overline{\teich}_{g,n} &\longrightarrow \R_{\geq 0} \\
            X &\longmapsto \sum_{(\Gamma, \As)} f_{L'}(\Gamma)f_{L_A'}(\As).
        \end{aligned}
    \]
    Since the length spectrum is discrete, see for instance~\cite[Lemma 12.4]{PrimerBook}, there are only finitely many multicurves $\Gamma$ such that $f_{L'}(\Gamma)>0$. For each such multicurve $X\smallsetminus \Gamma$ is a collection of subsurfaces with geodesic boundaries and/or cusps. By Lemma~\ref{TruncOrthoSpectrum}, the truncated length spectrum of simple orthogeodesics is discrete as well. Therefore, there are only finitely many maximal collections of orthogeodesic arcs $\As$ such that $f_{L_A'}(\As)>0$. Consequently $F$ is well-defined.

    The truncated length functions are continuous by Lemma~\ref{lem:truncated_length_continuous} and so are length functions on the augmented Teichmüller space. Thus $F$ is continuous.
    
    Moreover, $F$ is invariant under the action of the mapping class group so it descends to a continuous function on the compactified moduli space. Hence $F$ is bounded from above by some constant $N$ depending only on the topological type $(g,n)$.
    
    In particular if a hexagon decomposition $(\Gamma, \As)$ is $(L',L_A',\wtrunc)$-short on X, then $X\in V(\Gamma, \As)$. By definition, this guarantees $\ell_X(\gamma)\leq L'$ and $\ell_{X,\Gamma}^{\wtrunc}(a)\leq L_A'$
    for all $\gamma \in \Gamma$ and $a\in \As$. For every such decomposition $f_{L'}(\Gamma)f_{L_A'}(\As)=1$, hence the total number of short hexagon decompositions on X is bounded above by $F(X)$ which is itself bounded by $N$:
    \[ \sharp \left\{ (\Gamma, \As) \mid X \in V_{L',L_A'}(\Gamma, \As) \right\} \leq N, \]
    where $\sharp$ denotes cardinality. 
\end{proof}

\begin{lemma}\label{QILemma1}
    Let $L'\geq L$, $L_A'\geq L_A$. Let $\mathcal{H},\mathcal{H}'$ be two hexagon decompositions in $\mathscr{H}_{g,n}$ for which 
    \[
        V_{L',L_A'}(\mathcal{H})\cap V_{L',L_A'}(\mathcal{H}') \neq \emptyset.
    \]
    Then, there exists a constant $D(g,n)$ such that 
    \[
        \dhex(\mathcal{H}, \mathcal{H}') \leq D(g,n).
    \]
\end{lemma}
 
\begin{proof}
    Fix $X \in V_{L',L_A'}(\mathcal{H}) \cap V_{L',L_A'}(\mathcal{H}')$, by Lemma~\ref{CountSHortHExagonDecomp} there exists only finitely many hexagon decompositions $(\mathcal{H}_i)_{i=1}^{N}$ so that $\mathcal{H}_i=(\Gamma_{\mathcal{H}_i}, \mathcal{A}_{\mathcal{H}_i})$ is short on $X$. 
    Let 
    \[
        D(g,n)=\max_{i=1,\dots,n}\dhex( \mathcal{H}, \mathcal{H}_i),
    \] 
    thus $\dhex(\mathcal{H}, \mathcal{H}') \leq D(g,n)$.
\end{proof}

\begin{lemma}\label{QILemma2}
        Given $L'> L$ and $L_A'> L_A$, there exists an integer $J>0$ such that if $(X_t)_{t\in [0,1]}$ is a unit length Weil–Petersson geodesic in $\overline{\teich}_{g,n}$, then there exist hexagon decompositions $\mathcal{H}_1,\dots \mathcal{H}_J$, such that $(X_t)_{t\in [0,1]}$ lies in the union \[V_{L',L_A'}(\mathcal{H}_1) \cup \dots \cup V_{L',L_A'}(\mathcal{H}_J).\]
\end{lemma}

\begin{proof}
    Let $\Sigma$ be a hyperbolic surface of type $(g,n)$ with $2g-2+n>0$. Let $\mathcal{H}=(\Gamma, \mathcal{A})$ be a hexagon decomposition of $\Sigma$. Let $\Twist(\Gamma)$ denote the subgroup of the mapping class group $\Mod(\Sigma)$ generated by Dehn twists about the curves in $\Gamma$. 
    
    \begin{claim*}
        The quotient space $\faktor{V_{L',L_A'}(\mathcal{H})}{\Twist(\Gamma)}$ is compact.
    \end{claim*}
    
    \begin{proofclaim}
        Let $(Y_m)_{m\in \N}$ be a sequence in $V_{L',L_A'}(\mathcal{H})$. For every $a\in\As$, we have $\ell_{Y_m,\Gamma}^{\wtrunc}(a) \in [\epsilon, L_A']$. Here, $\epsilon > 0$ since standard collars around simple closed geodesics are disjoint. By Lemma~\ref{lem:truncated_length_continuous}, the truncated length function is continuous. By passing from one subsequence to another successively for each arc $a\in \As$, we obtain a subsequence where all truncated lengths converge simultaneously:
        \[ \ell_{Y_m,\Gamma}^{\wtrunc}(a) \longrightarrow \ell^{\wtrunc}(a) \in [\epsilon, L_A']. \]
        
        Similarly, there exists a subset of curves $\Gamma^0 \subseteq \Gamma$ (which may be empty) and a constant $\delta > 0$ such that up to subsequence, for all $\gamma \in \Gamma^0$, $\ell_{Y_m}(\gamma) \longrightarrow 0$, and for all $\gamma \in \Gamma \smallsetminus \Gamma^0$, $\ell_{Y_m}(\gamma) \longrightarrow \ell(\gamma) \in [\delta, L']$.

        Thus, up to a subsequence, there exists a sequence of twist maps $(D_m)_{m\in \N}$ in $\Twist(\Gamma)$ such that for all $\gamma \in \Gamma \smallsetminus \Gamma^0$, the twist coordinates converge: \[\tau_{D_m Y_m}(\gamma) \longrightarrow \tau(\gamma) \in [0,\ell(\gamma)).\]

        By Proposition~\ref{HexagonalCoordinates}, the parameters $\ell^{\wtrunc}(a)$, $\ell(\gamma)$, and $\tau(\gamma)$ determine a unique surface $Y \in \overline{\teich}_{g,n}$ and such that $Y \in V_{L',L_A'}(\mathcal{H})$. Consequently, the quotient $\faktor{V_{L',L_A'}(\mathcal{H})}{\Twist(\Gamma)}$ is compact.
    \end{proofclaim}

    Next, we consider the boundaries $\partial V_{L',L_A'}(\mathcal{H})$ and $\partial V(\mathcal{H})$. Because $L' > L$ and $L_A' > L_A$, the two sets $ \faktor{\partial V_{L',L_A'}(\mathcal{H})}{\Twist(\Gamma)}$ and $ \faktor{\partial V(\mathcal{H})}{\Twist(\Gamma)}$ are closed subsets of the compact set $\faktor{V_{L',L_A'}(\mathcal{H})}{\Twist(\Gamma)}$ and thus they are compact by the claim above.
    
    Define the distance function:
    \[ I_{L',L_A'}(\mathcal{H}) \coloneq \inf_{X\in \partial V(\mathcal{H})} \dwp\bigl(X, \partial V_{L',L_A'}(\mathcal{H})\bigr).\]
    
    The two subsets $ \faktor{\partial V_{L',L_A'}(\mathcal{H})}{\Twist(\Gamma)}$ and $ \faktor{\partial V(\mathcal{H})}{\Twist(\Gamma)}$ are disjoint and compact in $\faktor{V_{L',L_A'}(\mathcal{H})}{\Twist(\Gamma)}$, this implies $I_{L',L_A'}(\mathcal{H}) > 0$. 
    
    Furthermore, the function $\mathcal{H} \longmapsto I_{L',L_A'}(\mathcal{H})$ is $\Mod(\Sigma)$-invariant. Therefore, it descends to a function: 
    \[ \tilde{I}_{L',L_A'} \colon \mathscr{H}_{g,n} / \Mod(\Sigma) \longrightarrow \R_{> 0}. \]
    By~\cite[Lemma 3.2]{Parlier2016}, the quotient $\mathscr{H}_{g,n} / \Mod(\Sigma)$ is finite. Hence, we can define:
    \[ \delta_0 \coloneq \min_{[\mathcal{H}] \in \mathscr{H}_{g,n} / \Mod(\Sigma)} \tilde{I}_{L',L_A'}([\mathcal{H}]) > 0. \]

    By Corollary~\ref{ShortHexagonDecompositionExistence} the sets $V(\mathcal{H})$ for $\mathcal{H}\in \mathscr{H}_{g,n}$ cover $\overline{\teich}_{g,n}$. Let $t_0 \in [0,1]$, There exists some hexagon decomposition $\mathcal{H}_i$ such that $X_{t_0} \in V(\mathcal{H}_i)$. By construction, for all $t \in (t_0 -\delta_0, t_0 +\delta_0)$, the surface $X_t$ lies in $V_{L',L_A'}(\mathcal{H}_i)$. 
    
    Since the unit interval $[0,1]$ is compact, it can be covered by a finite number of such $\delta_0$-intervals of the form $(t - \delta_0, t + \delta_0)$. Let $J = \lceil 1/\delta_0 \rceil$. We can select a finite number of points $t_1, \dots, t_J \in [0,1]$ such that their corresponding $\delta_0$-neighborhoods cover $[0,1]$ then we can select $\mathcal{H}_1,\dots, \mathcal{H}_J$ so that \[X_t \in V_{L',L_A'}(\mathcal{H}_1) \cup \dots \cup V_{L',L_A'}(\mathcal{H}_J).\] for each $t\in [0,1]$.
\end{proof}

\begin{theorem}\label{MainTheorem3}
    Let $g,n$ be non-negative integers such that $2g-2+n >0$. The map $Q$ is a quasi-isometry from the hexagon graph $\mathscr{H}_{g,n}$, endowed with the combinatorial metric, to $\left(\overline{\teich}_{g,n}, \dwp\right)$.
\end{theorem}
\begin{proof}
    By Proposition~\ref{WidthHexagonModel} the map $Q$ is coarsely surjective. By Proposition~\ref{QLipschitz} it is also $K(g,n)$- Lipschitz, thus given $\mathcal{H},\mathcal{H}'\in \mathscr{H}_{g,n}$ we have
    \[
        \dwp(Q(\mathcal{H}),Q(\mathcal{H}'))\leq K(g,n)\dhex(\mathcal{H},\mathcal{H}').
    \]
    To establish the lower bound, let $L'=2L$ and $L_A'=2L_A$, and let $X,Y\in \overline{\teich}_{g,n}$ such that $X$ lies in $V(\mathcal{H})$ and $Y$ lies in $V(\mathcal{H}')$. By Lemma~\ref{QILemma2} we can cover the Weil–Petersson geodesic $(X_t)_{t\in [0,d]}$ of length $d=\dwp(X,Y)$ joining $X$ and $Y$ with a collection 
    \[V_{L',L_A'}(\mathcal{H}_1) \cup \dots \cup V_{L',L_A'}(\mathcal{H}_l),\]
    where \[l=\left\lceil\frac{d}{\delta_0}\right\rceil \leq \frac{d}{\delta_0} +1 \leq J(d+1),\]
    $\mathcal{H}_1=\mathcal{H}$, $\mathcal{H}_l=\mathcal{H}'$ and $J = \lceil 1/\delta_0 \rceil$ is the constant from Lemma~\ref{QILemma2}. Moreover for successive hexagon decompositions we have $V_{L',L_A'}(\mathcal{H}_i)\cap V_{L',L_A'}(\mathcal{H}_{i+1}) \neq \emptyset$ for all $i=1,\dots ,l-1$. Hence by Lemma~\ref{QILemma1} we obtain:
    \[
        \dhex(\mathcal{H}_i, \mathcal{H}_{i+1}) \leq E(g,n).
    \]
    thus
    \[
        \dhex(\mathcal{H}, \mathcal{H}') \leq l \cdot E(g,n) \leq J(d+1) E(g,n).
    \]
    Consequently, we conclude that 
    \[
       \frac{\dhex(\mathcal{H},\mathcal{H}')}{E(g,n)J} - 1 \leq \dwp(Q(\mathcal{H}), Q(\mathcal{H}')).
    \]
\end{proof}

\section{Ideal triangulations and the thick-part of Teichmüller space}\label{sec:flipqi}

In this section we consider hyperbolic surfaces with genus $g$ and at least one cusp $n\geq 1$. Fix $\epsilon<\epsilon_0$, we study the map
\[
    Q:\Fs_{g,n}\longrightarrow \teich^{\epsilon}_{g,n}.
\]
Recall that $Q$ assigns to each ideal triangulation $\mathcal{T}$ in the flip graph of triangulations the hyperbolic surface $X_{\mathcal{T}}$ whose shearing coordinates with respect to $\mathcal{T}$ are trivial. 

First we prove that $X_{\mathcal{T}}$ lies in the thick part of Teichmüller space.

\begin{lemma}\label{SysShearZero} 
    Let $\Sigma$ be a hyperbolic surface of type $(g,n)$ with $2g-2+n>0$, $n\geq 1$ and let $\mathcal{T}$ be an ideal triangulation of $\Sigma$, then 
    \[ \sys(X_{\mathcal{T}}) \geq \log\left(\sqrt{3}\right). \]
\end{lemma}

\begin{proof}
   Let $a \in \mathcal{T}$, by the assumption, the shear parameter along $a$ on $X_{\mathcal{T}}$ is trivial. Hence by~\cite[Lemma 5.1]{ManmanJiang2021} we obtain 
    \begin{equation*}
        0=S_{X_\mathcal{T}}^{\mathcal{T}}(a) \geq 2\log\left(\frac{1}{\sys(X_{\mathcal{T}})}\right)-2\log\left(\frac{1}{\log\left(\sqrt{3}\right)}\right),
    \end{equation*}
    therefore, 
    \[ \sys(X_{\mathcal{T}}) \geq \log\left(\sqrt{3}\right). \] 
   
\end{proof}

\begin{lemma}\label{TeichshearingIDEAL}
    Let $\epsilon \leq \epsilon_0$ and let $g,n$ be integers such that $2g-2+n>0$ and $n\geq 1$. Then for any point $X$ in $\teich^{\epsilon}_{g,n}$ there exists an ideal triangulation $\mathcal{T}$ such that 
    \[ 
       \dteich(X,Q(\mathcal{T})) \lesssim
       \frac{g+n}{\epsilon}.
    \]
\end{lemma}

\begin{proof}
    Let $X \in \teich^{\epsilon}_{g,n}$, by Remark~\ref{RmkBersIdealTriangulation} and~\cite[Section 3.2]{Bershear} there exists an ideal triangulation $\mathcal{T}$ of $X$ such that  for all ideal arc $a\in \mathcal{T}$,
    \[\ell_X(\hat{a}) \leq \mathcal{L}\coloneq \frac{2\pi(2g-2+n)}{\min(\epsilon, 0.27)} + \log\left(\frac{2}{0.27}\right)\]
    where $\hat{a}$ denotes the subarc from $a$ obtained by removing the portions contained in standard cusp neighborhood. 
    Thus, by Proposition~\ref{PropTeichShear} we obtain: 
    \[ 
       \dteich(X,Q(\mathcal{T})) \lesssim
       \frac{g+n}{\epsilon}. 
    \]
\end{proof}

\begin{lemma}\label{Flip-Lipschitz}
    Let $\mathcal{T}$ and $\mathcal{T}'$ be two ideal triangulations related by one flip move in the flip graph of triangulations $\Fs_{g,n}$, then 
    \[ \dteich(Q(\mathcal{T}), Q(\mathcal{T}')) \lesssim \log(g+n)\]
\end{lemma}
\begin{proof}
    Let $\mathcal{T}$ and $\mathcal{T}'$ differ by one flip move in $\Fs_{g,n}$. We want to estimate the Teichmüller distance between $X_{\mathcal{T}}$ and $X_{\mathcal{T}'}$. Denote by $e$ the flipped arc in $\mathcal{T}$ and by $e'$ the corresponding arc in $\mathcal{T}'$. The shearing coordinates of $X_{\mathcal{T}}$ with respect to $\mathcal{T}$ are all trivial, hence by the proof of Lemma~\ref{Shear0}, for any $a \in \mathcal{T}$ we have
    \[
        \ell_{X_\mathcal{T}}^{\wtrunc}(a) \leq \log(6g-3+3n).
    \]
    Moreover, on $X_{\mathcal{T}}$, the truncated length of $e'$ is bounded by the perimeter of the truncated ideal quadrilateral and so 
    \[
        \ell_{X_\mathcal{T}}(\hat{e}')
        \leq (2\epsilon_T + 2\log(6g-3+3n))\lesssim \log(g+n).
    \]
    Consequently, the truncated length of any  ideal arc $ a \in \mathcal{T}'$ is bounded on $X_\mathcal{T}$ by \[ \ell_{X_\mathcal{T}}(\hat{a})\lesssim \log(g+n),\] where $\hat{a}$ denotes the subarc of $a$ leaving the standard cusp neighborhoods. In addition, by definition 
    \[
        \shear_{\mathcal{T}'}(X_{\mathcal{T}'})=0.
    \]
    Therefore by Proposition~\ref{PropTeichShear} we obtain:
    \[ 
        \dteich(X_\mathcal{T}, X_{\mathcal{T}'}) \lesssim \log(g+n).
    \]
\end{proof}

The thick part of Teichmüller space is not always convex with respect to the Teichmüller metric~\cite{LiuShigaSun2014}, however, it is path-connected for $\epsilon$ sufficiently small. Therefore, we consider the induced path metric; that is, for $X,Y \in \teich^{\epsilon}_{g,n}$, let
\[
    \dteichpath(X,Y) \coloneq \inf_{c} \sup \left\{ \sum_{i=1}^{m} \dteich\bigl(c(t_{i-1}),c(t_{i})\bigr) \;\middle|\; \substack{0=t_0\le \dots \le t_m=1, \\ m\in\mathbb{N}} \right\},
\]
where the infimum is taken over all continuous paths $c\colon[0,1]\to \teich^{\epsilon}_{g,n}$ with $c(0)=X$ and $c(1)=Y$. In particular, one has $\dteich \leq \dteichpath$.

\begin{proposition}\label{MainTheorem4}
    Let $0<\epsilon\leq\log\sqrt3$, and let $g,n$ be integers with $2g-2+n>0$ and $n\geq1$. Then the map $Q\colon\Fs_{g,n}\to\teich^{\epsilon}_{g,n}$ satisfies:
    \begin{enumerate}[label=(\roman*)]
        \item\label{flip-cost} For all $\mathcal{T},\mathcal{T}'\in\Fs_{g,n}$,
        \[
            \dteich\left(Q(\mathcal{T}),Q(\mathcal{T}')\right)\;\lesssim\log(g+n)\,d_{\Fs}(\mathcal{T},\mathcal{T}').
        \]
        \item\label{flip-width} The width of the flip graph model satisfies
        \[
            \sup_{\teich^{\epsilon}_{g,n}}
            \inf_{\mathcal{T}\in\Fs_{g,n}}
            \dteich(X,Q(\mathcal{T}))\leq \frac{g+n}{\epsilon}.
        \]
        \item\label{flip-qi} The map $Q$ is a quasi-isometry from $\left(\Fs_{g,n},d_{\Fs}\right)$ to $\left(\teich^{\epsilon}_{g,n},\dteichpath\right)$.
    \end{enumerate}
\end{proposition}

\begin{remark}\label{rem1}
    Statements~\ref{flip-cost} and~\ref{flip-width} are the content of Lemma~\ref{Flip-Lipschitz} and Lemma~\ref{TeichshearingIDEAL}, respectively. Then $\Mod(\Sigma)$ acts on $\Fs_{g,n}$ and on $\left(\teich^{\epsilon}_{g,n},\dteichpath\right)$, properly and cocompactly by isometries, hence it follows from the Schwarz--Milnor lemma that they are both quasi-isometric to $\Mod(\Sigma)$~\cite{DisarloParlier2019,LackenbyPurcell2024}. 

    The map $Q$ is $\Mod(\Sigma)$-equivariant. Indeed, for every $\varphi\in\Mod(\Sigma)$ and every arc $a\in\mathcal{T}$, we have $S^{\varphi\mathcal{T}}_{\varphi a}(\varphi\cdot X)=S^{\mathcal{T}}_{a}(X)$. Hence the surface $\varphi\cdot Q(\mathcal{T})$ has vanishing shear parameters along $\varphi\cdot\mathcal{T}$. That is,
    \[
        Q(\varphi\cdot\mathcal{T})=\varphi\cdot Q(\mathcal{T}) \qquad \text{for all } \varphi\in\Mod(\Sigma).
    \]
    Fix a vertex $\mathcal{T}\in\Fs_{g,n}$, write $X_{\mathcal{T}}=Q(\mathcal{T})$, and let $T_\gamma\in\Mod(\Sigma)$ be the Dehn twist along an essential simple closed curve $\gamma$. By equivariance of $Q$, we have $Q(T_\gamma^N\cdot\mathcal{T})=T_\gamma^N\cdot X_{\mathcal{T}}$ for all $N\in\Z$.

    On one hand, cyclic subgroups generated by Dehn twists are undistorted in the mapping class group~\cite{FarbLubotzkyMinsky2001}, so the quasi-isometry between the flip graph and $\Mod(\Sigma)$ endowed with the word metric yields
    \[
        d_{\Fs}\left(\mathcal{T},\,T_\gamma^N\cdot\mathcal{T}\right)\;\geq\; c_1 |N| - c_2
    \]
    for some constants $c_1, c_2 >0$ independent of $N$.

    On the other hand, the twist $T_\gamma$ acts by isometry on $\left(\overline{\teich}_{g,n},\dwp\right)$. Since it fixes the stratum $S(\gamma)$ pointwise, we have
    \[
        \dwp\left(X_{\mathcal{T}},\,T_\gamma^N\cdot X_{\mathcal{T}}\right)
        \leq 2\,\dwp\left(X_{\mathcal{T}},\,S(\gamma)\right),
    \]
    a bound independent of $N$. Hence, by Theorem~\ref{MainTheorem3}, $\dhex(\mathcal{T},\,T_\gamma^N\cdot\mathcal{T})$ is uniformly bounded, independently of $N$, while $d_{\Fs}(\mathcal{T},\,T_\gamma^N\cdot\mathcal{T})$ grows linearly. Therefore, we cannot deduce Theorem~\ref{MainTheorem4} from Theorem~\ref{MainTheorem3}: even though $\Fs_{g,n}\subset\mathscr{H}_{g,n}$ is the subgraph on the vertices with $\Gamma=\emptyset$, this inclusion is not a quasi-isometric embedding.

    The same example explains why Theorem~\ref{MainTheorem4} requires the path metric $\dteichpath$. Let $A\subset X_{\mathcal{T}}$ be an embedded annulus with core curve $\gamma$ and modulus $m$, identified with $(\R/\Z)\times[0,m]$. Identify $A$ as the cylinder $A\cong(\R/\Z)\times[0,m]$ with coordinates $z=x+iy$. Let $K \geq 1$, we construct a path connecting $X_{\mathcal{T}}$ to $T_\gamma^N \cdot X_{\mathcal{T}}$ in three steps:
    \begin{itemize}
        \item[-] Define $g_K(x+iy) = x + iKy$ on $A$ and the identity outside $A$. This quasiconformal map has maximal dilatation $K$ and yields a surface $X_K$ where the annulus $A$ is stretched to a modulus of $Km$. The Teichmüller distance is $\dteich(X_{\mathcal{T}}, X_K) \leq \frac{1}{2}\log K$.
        \item[-] On $X_K$, define the affine twist $h_N(x+iy) = \left(x + \frac{N y}{Km}\right) + iy$ on $A$ and the identity outside. This map is isotopic to $T_\gamma^N$. Its maximal dilatation is $e^{2\arcsinh(|N|/2Km)}$, so $\dteich(X_K, T_\gamma^N \cdot X_K) \leq \arcsinh\left(\frac{|N|}{2Km}\right)$.
        \item[-] Apply $g_K^{-1}$ to return the conformal structure outside the annulus to that of $X_{\mathcal{T}}$. The distance is again bounded by $\frac{1}{2}\log K$.
    \end{itemize}
    The total Teichmüller distance is bounded by the length of this path:
    \[
        \dteich\bigl(X_{\mathcal{T}},\,T_\gamma^N\cdot X_{\mathcal{T}}\bigr) \leq \log K + \arcsinh\left(\frac{|N|}{2Km}\right).
    \]
    Choosing $K = \max\bigl(1, \frac{|N|}{2m}\bigr)$, the twist term $\arcsinh(1)$ becomes a constant, yielding
    \[
        \dteich\bigl(X_{\mathcal{T}},\,T_\gamma^N\cdot X_{\mathcal{T}}\bigr) \leq \log |N| + C,
    \]
    where $C$ depends only on $m$ (and thus on $X_{\mathcal{T}}$ and $\gamma$), but not on $N$. However, for $|N|\geq 2m$ the intermediate surface $X_K$ contains an annulus of modulus $Km=|N|/2$ with core $\gamma$, so $\ell_{X_K}(\gamma)\leq \pi/(Km) \leq 2\pi/|N|$: this short path leaves the thick part $\teich^{\epsilon}_{g,n}$ for $|N|$ large. Within the thick part, $\dteichpath\bigl(X_{\mathcal{T}},\,T_\gamma^N\cdot X_{\mathcal{T}}\bigr)$ is forced to grow linearly in $N$.
\end{remark}
\color{black}
\begin{theorem}\label{bigCOR}
    Let $X\in \overline{\teich}_{g,n}$. Then there exists a hexagon decomposition $(\Gamma, \mathcal{A})$ of $X$ and a noded surface $X_\infty$ such that  
    \begin{enumerate}
        \item\label{end_cor_one} $X_\infty$ lies in the $\epsilon_T'$-thick part of the stratum $S(\Gamma)$, where $\epsilon_T'\approx 0.016$ is a universal constant, independent of the topology.
        \item\label{end_cor_one_bis} the shearing coordinates of $X_\infty$ along $\As_\infty$ satisfy
        \[
            \shear_{\As_\infty}(X_\infty)\;\leq\;44\log\bigl(8\pi(2g-2+n)\bigr)+249;
        \]
        \item\label{end_cor_two} $\dwp(X,X_{\infty})\leq 5\sqrt{2\pi(2g-2+n)\log(8\pi(2g-2+n))}$.
        \item $\dteich(X_{\infty}, Q(\Gamma, \mathcal{A})) \lesssim \log\left(g+n\right).$
    \end{enumerate}
\end{theorem}

\begin{proof}
    The properties~(\ref{end_cor_one}),~(\ref{end_cor_one_bis}) and~(\ref{end_cor_two}) follow directly from Proposition~\ref{PropProjectThickStartum}. Next by inequality~(\ref{boundArcTruncatedCVStratum}) from the proof of Proposition~\ref{PropProjectThickStartum} the truncated lengths of the ideal arcs in $\As_{\infty}$ are uniformly bounded on $X_\infty$ by $\lesssim \log(g+n)$. Since by definition \[ \shear_{\As_{\infty}}(Q(\Gamma, \mathcal{A}))=0\] we conclude by applying Proposition~\ref{PropTeichShear}.
\end{proof}

\section{Interpolation with the Thurston metric}\label{sec:thurston}

The Teichmüller metric is comparable to another metric on $\teich_{g,n}$ called the Thurston asymmetric metric~\cite{Thurston1998}. The asymmetric Thurston metric, measures the minimal amount by which lengths of curves must be stretched to deform one hyperbolic surface into another. It is defined by
\[
d_{\mathrm{Th}}(X,Y) = \log \inf_{f \sim \mathrm{id}} L(f),
\]
where the infimum is taken over all Lipschitz homeomorphisms $f : X \to Y$ isotopic to the identity and $L(f)$ is the Lipschitz constant of $f$. 
Thurston proved that this metric can equivalently be expressed as
\[
d_{\mathrm{Th}}(X,Y) = \log \sup_{\alpha \in \mathcal{S}} 
\frac{\ell_Y(\alpha)}{\ell_X(\alpha)},
\]
where $\mathcal{S}$ denotes the set of simple closed curves on $X$~\cite{Thurston1998}. 
The symmetrized version,
\[
d_{\mathrm{L}}(X,Y) = \max\{ d_{\mathrm{Th}}(X,Y),\, d_{\mathrm{Th}}(Y,X) \},
\]
is called the \emph{Lipschitz metric}.

Choi and Rafi~\cite{ChoiRafi2007} showed that on the $\epsilon$-thick part of the Teichmüller space, the Teichmüller and Lipschitz metrics are comparable up to an additive constant which depends only on $\epsilon$ and on the topological type of the surface.

A \emph{stretch path} is a geodesic for the Thurston metric constructed by stretching along the edges of an ideal triangulation. While arbitrary points in $\teich_{g,n}$ generally cannot be connected by a single stretch segment, the following section shows that we can always find a stratum $S(\Gamma)$ where such a move is possible.

\addtocontents{toc}{\SkipTocEntry}
\subsection*{Thurston stretch and Weil--Petersson projection}
We now briefly recall the notion of a \emph{stretch path}. Given a maximal geodesic lamination $\lambda$ on a hyperbolic surface $X$, the stretch path $(S_t)_{t \ge 0}$ along $\lambda$ is a path in $\overline{\teich}_{g,n}$ defined by 
\[t \in \R_{>0} \mapsto \stret(X,\lambda, t) \quad \text{where} \quad \shear(\stret(X,\lambda, t))=e^t \shear(X).\] 
Thurston showed that stretch paths are geodesics for the Thurston metric and satisfy\[d_{\mathrm{Th}}( \stret(X,\lambda, s),  \stret(X,\lambda, t)) = t-s \quad \forall s<t.\]
The following result provides a way of combining Weil--Petersson projections with Thurston stretch geometry, and may be of independent interest in understanding the coarse interaction between the two metrics. 
\begin{proposition}\label{MainTheorem}
  Let $X \in \overline{\teich}_{g,n}$, then there exists a hexagon decomposition $(\Gamma, \mathcal{A})$ and two noded surfaces $X_r$ and $X_{r'}$ in the stratum $S(\Gamma)$ such that
  \begin{enumerate}
    \item\label{MainTheorem-step1} their shearing coordinates satify   \[ 
        \shear_{\mathcal{A}_\infty}(X_r) \leq 11\log(8\pi(2g-2+n))+ 80
    \]
    and
    \[
        \shear_{\mathcal{A}_\infty}(X_{r'}) \leq \frac{1}{\log(8\pi(2g-2+n))};
    \]
    \item\label{MainTheorem-step2} the surfaces $X_r$ and $X_{r'}$ are close, respectively, to $X$ and the noded surface $N_{\Gamma, \mathcal{A}}$ with respect to the Weil--Petersson metric:
    \[
        \dwp(X, X_r) \leq 5 \sqrt{2\pi(2g-2+n)\log(8\pi(2g-2+n))},
    \]
    \[    
        \dwp(X_{r'}, N_{\Gamma, \mathcal{A}}) \leq 22\sqrt{(2g-2+n)\log\bigl(8\pi(2g-2+n)\bigr)};
    \]
    \item $X_r$ and $X_{r'}$ lie on one geodesic stretch path, therefore the Thurston distance between them is given by: 
    \[
        d_{\mathrm{Th}}(X_{r'}, X_r)= \log\left(8\pi(2g-2+n)\right) + \log\left(11\log(8\pi(2g-2+n))+80\right).
    \]
  \end{enumerate}
\end{proposition}

\begin{proof}
By Corollary~\ref{ShortHexagonDecompositionExistence} choose $(\Gamma, \mathcal{A})$ such that $X \in V(\Gamma, \mathcal{A})$ and denote by $X_r$ the noded surface obtained by applying Proposition~\ref{TravellingToStratum}. Let $r=(r_i)_{i=1}^{6g-6+3n}$ be the tuple of shearing coordinates of $X_r$ with respect to the ideal triangulation $\mathcal{A}_{\infty}$ thus by construction for each $i=1,\dots,6g-6+3n$
\[
    |r_i| \leq 11\log(8\pi(2g-2+n))+80.
\] 
Let $r'=(r_i')_{i=1}^{6g-6+3n}$ be the tuple defined by 
\[
    r'_i = \frac{r_i}{8\pi(2g-2+n)(11\log(8\pi(2g-2+n))+80)},
\]
hence $|r'_i|\leq \sfrac{1}{8\pi(2g-2+n)} <1$. Let $X_{r'}$ be the noded surface with shearing coordinates $r'$ with respect to $\mathcal{A}_{\infty}$, and so by construction we have~(\ref{MainTheorem-step1}). 

Write $m=2g-2+n$. Each puncture or node carries at most $6m$ arc-ends of $\mathcal{A}_\infty$, and $|r'_i|\leq 1/(8\pi m)$, so every sum of shears of $X_{r'}$ over consecutive arc-ends at a puncture or node is at most $6m/(8\pi m)<1/4$ in absolute value, that is $\mathcal{O}_{\mathcal{A}_{\infty}}(X_{r'})<1/4$, while $\mathcal{O}_{\mathcal{A}_{\infty}}(N_{\Gamma,\mathcal{A}})=0$. Let $\Sigma_i$ be the components of $\Sigma\smallsetminus\Gamma$, of types $(g_i,n_i)$, with $m_i=2g_i-2+n_i$, so that $\sum_i m_i=m$. By Remark~\ref{eq:WP-product-strata} and Theorem~\ref{thm:WolpertShear}(\ref{thm:WolpertShear-osc}) applied on each $\Sigma_i$,
\[
    \dwp(N_{\Gamma,\mathcal{A}},X_{r'})^2\leq 21^2\sum_i m_i\bigl(\log(24m_i)+\tfrac14\bigr)\leq 21^2\,m\bigl(\log(24m)+\tfrac14\bigr).
\]
Since $\log(24m)\leq\log(8\pi m)$ and $\frac14<0.08\log(8\pi m)$, we get 
\[
    \dwp(N_{\Gamma,\mathcal{A}},X_{r'})\leq 21\sqrt{1.08\,m\log(8\pi m)}\leq 22\sqrt{m\log(8\pi m)}.
\] 
Together with the distance estimate from Proposition~\ref{TravellingToStratum} this gives~(\ref{MainTheorem-step2}) 

Observe that the shearing coordinates of $X_r$  and $X_{r'}$ satisfy 
\[ 
    \shear(X_r) = e^{\log\left(\;8\pi(2g-2+n)(11\log(8\pi(2g-2+n))+80)\;\right)}\shear(X_{r'}).
\] 
Since the shearing coordinates along $\mathcal{A}_{\infty}$ parametrize $S(\Gamma)$ we identify 
\[ 
    X_r = \stret\left(X_{r'}, \mathcal{A}_{\infty}, t\right), 
\] where $t=\log\left(8\pi(2g-2+n)(11\log(8\pi(2g-2+n))+80)\;\right)$. Therefore there exists a stretch path from $X_{r'}$ to $X_r$ satisfying
\[
    d_{\mathrm{Th}}(X_{r'}, X_r)= \log\left(8\pi(2g-2+n)\right) + \log\left(11\log(8\pi(2g-2+n))+80\right).
\]
\end{proof}

\bibliographystyle{alpha}
\bibliography{bib}

@article{Parlier2016,
    AUTHOR = {Parlier, Hugo},
     TITLE = {Interrogating surface length spectra and quantifying
              isospectrality},
   JOURNAL = {Math. Ann.},
  FJOURNAL = {Mathematische Annalen},
    VOLUME = {370},
      YEAR = {2018},
    NUMBER = {3-4},
     PAGES = {1759--1787},
      ISSN = {0025-5831,1432-1807},
   MRCLASS = {32G15 (11F72 30F10 30F60 53C22 58J53)},
  MRNUMBER = {3770180},
MRREVIEWER = {Benjamin\ Linowitz},
       DOI = {10.1007/s00208-017-1571-x},
       URL = {https://doi.org/10.1007/s00208-017-1571-x},
}

@book{Buser,
    AUTHOR = {Buser, Peter},
     TITLE = {Geometry and spectra of compact {R}iemann surfaces},
    SERIES = {Progress in Mathematics},
    VOLUME = {106},
 PUBLISHER = {Birkh\"auser Boston, Inc., Boston, MA},
      YEAR = {1992},
     PAGES = {xiv+454},
      ISBN = {0-8176-3406-1},
   MRCLASS = {58G25 (30F99)},
  MRNUMBER = {1183224},
MRREVIEWER = {Robert\ Brooks},
}

@book{Penner2012,
    AUTHOR = {Penner, Robert C.},
     TITLE = {Decorated {T}eichm\"uller theory},
    SERIES = {QGM Master Class Series},
      NOTE = {With a foreword by Yuri I. Manin},
 PUBLISHER = {European Mathematical Society (EMS), Z\"urich},
      YEAR = {2012},
     PAGES = {xviii+360},
      ISBN = {978-3-03719-075-3},
   MRCLASS = {30-02 (30F60 32G15)},
  MRNUMBER = {3052157},
MRREVIEWER = {Anna\ Wienhard},
       DOI = {10.4171/075},
       URL = {https://doi.org/10.4171/075},
}

@article{Wolpert1975,
    AUTHOR = {Wolpert, Scott},
     TITLE = {Noncompleteness of the {W}eil-{P}etersson metric for
              {T}eichm\"uller space},
   JOURNAL = {Pacific J. Math.},
  FJOURNAL = {Pacific Journal of Mathematics},
    VOLUME = {61},
      YEAR = {1975},
    NUMBER = {2},
     PAGES = {573--577},
      ISSN = {0030-8730,1945-5844},
   MRCLASS = {32G15 (30A58)},
  MRNUMBER = {422692},
MRREVIEWER = {K.\ Strebel},
       URL = {http://projecteuclid.org/euclid.pjm/1102868050},
}

@article {hensel2008,
    AUTHOR = {Hensel, Sebastian W.},
     TITLE = {Iterated grafting and holonomy lifts of {T}eichm\"uller space},
   JOURNAL = {Geom. Dedicata},
  FJOURNAL = {Geometriae Dedicata},
    VOLUME = {155},
      YEAR = {2011},
     PAGES = {31--67},
      ISSN = {0046-5755,1572-9168},
   MRCLASS = {57M50 (30F60 32G15 37F30)},
  MRNUMBER = {2863893},
MRREVIEWER = {Fr\'ed\'eric\ Palesi},
       DOI = {10.1007/s10711-011-9577-0},
       URL = {https://doi.org/10.1007/s10711-011-9577-0},
}

@article {DumasWolf2008,
    AUTHOR = {Dumas, Emily and Wolf, Michael},
     TITLE = {Projective structures, grafting and measured laminations},
   JOURNAL = {Geom. Topol.},
  FJOURNAL = {Geometry \& Topology},
    VOLUME = {12},
      YEAR = {2008},
    NUMBER = {1},
     PAGES = {351--386},
      ISSN = {1465-3060,1364-0380},
   MRCLASS = {30F60 (30F10 30F40 32G15 57M50)},
  MRNUMBER = {2390348},
MRREVIEWER = {Athanase\ Papadopoulos},
       DOI = {10.2140/gt.2008.12.351},
       URL = {https://doi.org/10.2140/gt.2008.12.351},
}

@article {DiazKim2012,
    AUTHOR = {D\'iaz, Raquel and Kim, Inkang},
     TITLE = {Asymptotic behavior of grafting rays},
   JOURNAL = {Geom. Dedicata},
  FJOURNAL = {Geometriae Dedicata},
    VOLUME = {158},
      YEAR = {2012},
     PAGES = {267--281},
      ISSN = {0046-5755,1572-9168},
   MRCLASS = {57M50 (30F60 32G15 51M10)},
  MRNUMBER = {2922715},
MRREVIEWER = {Joan\ Porti},
       DOI = {10.1007/s10711-011-9632-x},
       URL = {https://doi.org/10.1007/s10711-011-9632-x},
}

@misc{KahnMarkovic2008,
      title={Random ideal triangulations and the {W}eil-{P}etersson distance between finite degree covers of punctured {R}iemann surfaces}, 
      author={Jeremy Kahn and Vladimir Markovic},
      year={2008},
      eprint={0806.2304},
      archivePrefix={arXiv}
}

@article {Brock2003,
    AUTHOR = {Brock, Jeffrey F.},
     TITLE = {The {W}eil-{P}etersson metric and volumes of 3-dimensional
              hyperbolic convex cores},
   JOURNAL = {J. Amer. Math. Soc.},
  FJOURNAL = {Journal of the American Mathematical Society},
    VOLUME = {16},
      YEAR = {2003},
    NUMBER = {3},
     PAGES = {495--535},
      ISSN = {0894-0347,1088-6834},
   MRCLASS = {32G15 (30F40 30F60 37F30)},
  MRNUMBER = {1969203},
MRREVIEWER = {Edward\ C.\ Taylor},
       DOI = {10.1090/S0894-0347-03-00424-7},
       URL = {https://doi.org/10.1090/S0894-0347-03-00424-7},
}

@article {CavendishParlier2012,
    AUTHOR = {Cavendish, William and Parlier, Hugo},
     TITLE = {Growth of the {W}eil-{P}etersson diameter of moduli space},
   JOURNAL = {Duke Math. J.},
  FJOURNAL = {Duke Mathematical Journal},
    VOLUME = {161},
      YEAR = {2012},
    NUMBER = {1},
     PAGES = {139--171},
      ISSN = {0012-7094,1547-7398},
   MRCLASS = {32G15},
  MRNUMBER = {2872556},
MRREVIEWER = {Ioannis\ D.\ Platis},
       DOI = {10.1215/00127094-1507312},
       URL = {https://doi.org/10.1215/00127094-1507312},
}

@misc{Bershear,
      title={Bounding shears of spiralling triangulations on hyperbolic surfaces}, 
      author={Marie Abadie},
      year={2025},
      eprint={2512.09818},
      archivePrefix={arXiv}
}

@article{Roger2013,
    AUTHOR = {Roger, Julien},
     TITLE = {Factorization rules in quantum {T}eichmuller theory},
      NOTE = {Thesis (Ph.D.)--University of Southern California},
 PUBLISHER = {ProQuest LLC, Ann Arbor, MI},
      YEAR = {2010},
     PAGES = {60},
      ISBN = {978-1124-16198-3},
   MRCLASS = {99-05},
  MRNUMBER = {2782317},
       URL =
              {http://gateway.proquest.com/openurl?url_ver=Z39.88-2004&rft_val_fmt=info:ofi/fmt:kev:mtx:dissertation&res_dat=xri:pqdiss&rft_dat=xri:pqdiss:3418158},
}

@article{Abikoff1977,
    AUTHOR = {Abikoff, William},
     TITLE = {Degenerating families of {R}iemann surfaces},
   JOURNAL = {Ann. of Math. (2)},
  FJOURNAL = {Annals of Mathematics. Second Series},
    VOLUME = {105},
      YEAR = {1977},
    NUMBER = {1},
     PAGES = {29--44},
      ISSN = {0003-486X},
   MRCLASS = {32G15 (14H15)},
  MRNUMBER = {442293},
MRREVIEWER = {C.\ Earle},
       DOI = {10.2307/1971024},
       URL = {https://doi.org/10.2307/1971024},
}

@book{Abikoff1980,
    AUTHOR = {Abikoff, William},
     TITLE = {The real analytic theory of {T}eichm\"uller space},
    SERIES = {Lecture Notes in Mathematics},
    VOLUME = {820},
 PUBLISHER = {Springer, Berlin},
      YEAR = {1980},
     PAGES = {vii+144},
      ISBN = {3-540-10237-X},
   MRCLASS = {32G15 (30F99 57N05)},
  MRNUMBER = {590044},
MRREVIEWER = {L.\ Keen},
}

@incollection{Bers1974,
    AUTHOR = {Bers, Lipman},
     TITLE = {Spaces of degenerating {R}iemann surfaces},
 BOOKTITLE = {Discontinuous groups and {R}iemann surfaces ({P}roc. {C}onf.,
              {U}niv. {M}aryland, {C}ollege {P}ark, {M}d., 1973)},
    SERIES = {Ann. of Math. Stud.},
    VOLUME = {No. 79},
     PAGES = {43--55},
 PUBLISHER = {Princeton Univ. Press, Princeton, NJ},
      YEAR = {1974},
   MRCLASS = {30A46},
  MRNUMBER = {361051},
MRREVIEWER = {C.\ Earle},
}

@incollection{EarleMarden2012,
    AUTHOR = {Earle, Clifford J. and Marden, Albert},
     TITLE = {Holomorphic plumbing coordinates},
 BOOKTITLE = {Quasiconformal mappings, {R}iemann surfaces, and
              {T}eichm\"uller spaces},
    SERIES = {Contemp. Math.},
    VOLUME = {575},
     PAGES = {41--52},
 PUBLISHER = {Amer. Math. Soc., Providence, RI},
      YEAR = {2012},
      ISBN = {978-0-8218-5340-5},
   MRCLASS = {30F60 (32G20)},
  MRNUMBER = {2933892},
MRREVIEWER = {Zongliang\ Sun},
       DOI = {10.1090/conm/575/11411},
       URL = {https://doi.org/10.1090/conm/575/11411},
}

@article{McMullen1998,
    AUTHOR = {McMullen, Curtis T.},
     TITLE = {Complex earthquakes and {T}eichm\"uller theory},
   JOURNAL = {J. Amer. Math. Soc.},
  FJOURNAL = {Journal of the American Mathematical Society},
    VOLUME = {11},
      YEAR = {1998},
    NUMBER = {2},
     PAGES = {283--320},
      ISSN = {0894-0347,1088-6834},
   MRCLASS = {32G15 (20H10 30F60 57N05)},
  MRNUMBER = {1478844},
MRREVIEWER = {Darryl\ McCullough},
       DOI = {10.1090/S0894-0347-98-00259-8},
       URL = {https://doi.org/10.1090/S0894-0347-98-00259-8},
}

@incollection{ChekhovPenner2007,
    AUTHOR = {Chekhov, Leonid O. and Penner, Robert C.},
     TITLE = {On quantizing {T}eichm\"uller and {T}hurston theories},
 BOOKTITLE = {Handbook of {T}eichm\"uller theory. {V}ol. {I}},
    SERIES = {IRMA Lect. Math. Theor. Phys.},
    VOLUME = {11},
     PAGES = {579--645},
 PUBLISHER = {Eur. Math. Soc., Z\"urich},
      YEAR = {2007},
      ISBN = {978-3-03719-029-6},
   MRCLASS = {32G81 (32G15)},
  MRNUMBER = {2349681},
MRREVIEWER = {Lee-Peng\ Teo},
       DOI = {10.4171/029-1/15},
       URL = {https://doi.org/10.4171/029-1/15},
}

@incollection {Wolpert2003,
    AUTHOR = {Wolpert, Scott A.},
     TITLE = {Geometry of the {W}eil-{P}etersson completion of
              {T}eichm{\"u}ller space},
 BOOKTITLE = {Surveys in differential geometry, {V}ol.\ {VIII} ({B}oston,
              {MA}, 2002)},
    SERIES = {Surv. Differ. Geom.},
    VOLUME = {8},
     PAGES = {357--393},
 PUBLISHER = {Int. Press, Somerville, MA},
      YEAR = {2003}
}

@article {Linch1974,
      TITLE = {A comparison of metrics on {T}eichm{\"u}ller space},
      AUTHOR = {Linch, Michele},
      JOURNAL =  {Proceedings of the American Mathematical Society},
      VOLUME = {43},
      YEAR = {1974},
      PAGES = {349--352}
}

@article {ChoiRafi2007,
    AUTHOR = {Choi, Young-Eun and Rafi, Kasra},
     TITLE = {Comparison between {T}eichm{\"u}ller and {L}ipschitz metrics},
   JOURNAL = {Journal of the London Mathematical Society. Second Series},
    VOLUME = {76},
      YEAR = {2007},
    NUMBER = {3},
     PAGES = {739--756}
}

@incollection {Thurston1998,
    AUTHOR = {Thurston, William P.},
     TITLE = {Minimal stretch maps between hyperbolic surfaces},
 BOOKTITLE = {Collected works of {W}illiam {P}. {T}hurston with commentary.
              {V}ol. {I}. {F}oliations, surfaces and differential geometry},
     PAGES = {533--585},
      NOTE = {1998 preprint},
 PUBLISHER = {Amer. Math. Soc., Providence, RI},
      YEAR = {2022}
}

@article{Parlier2023,
    AUTHOR = {Parlier, Hugo},
     TITLE = {A shorter note on shorter pants},
   JOURNAL = {Bull. Lond. Math. Soc.},
  FJOURNAL = {Bulletin of the London Mathematical Society},
    VOLUME = {56},
      YEAR = {2024},
    NUMBER = {4},
     PAGES = {1483--1487},
      ISSN = {0024-6093,1469-2120},
   MRCLASS = {57K20 (30F60 32G15)},
  MRNUMBER = {4743819},
MRREVIEWER = {Athanase\ Papadopoulos},
       DOI = {10.1112/blms.13007},
       URL = {https://doi.org/10.1112/blms.13007},
}

@article{Bonahon1996,
    AUTHOR = {Bonahon, Francis},
     TITLE = {Shearing hyperbolic surfaces, bending pleated surfaces and
              {T}hurston's symplectic form},
   JOURNAL = {Ann. Fac. Sci. Toulouse Math. (6)},
  FJOURNAL = {Toulouse. Facult\'e{} des Sciences. Annales. Math\'ematiques.
              S\'erie 6},
    VOLUME = {5},
      YEAR = {1996},
    NUMBER = {2},
     PAGES = {233--297},
      ISSN = {0240-2955},
   MRCLASS = {57M50 (53C15 57N05 57N10)},
  MRNUMBER = {1413855},
MRREVIEWER = {Athanase\ Papadopoulos},
       URL = {http://www.numdam.org/item?id=AFST_1996_6_5_2_233_0},
}

@misc{fock1998,
      title={Dual {T}eichm{\"u}ller spaces}, 
      author={Vladimir Fock},
      year={1998},
      eprint={dg-ga/9702018},
      archivePrefix={arXiv}
}

@article{GultepeParlier2025,
    AUTHOR = {G\"{u}ltepe, Funda and Parlier, Hugo},
     TITLE = {The coarse geometry of hexagon decomposition graphs},
   JOURNAL = {Canad. J. Math.},
  FJOURNAL = {Canadian Journal of Mathematics. Journal Canadien de
              Math\'ematiques},
    VOLUME = {78},
      YEAR = {2026},
    NUMBER = {3},
     PAGES = {867--885},
      ISSN = {0008-414X,1496-4279},
   MRCLASS = {57K20 (32G15 57M15)},
  MRNUMBER = {5078261},
       DOI = {10.4153/S0008414X24000853},
       URL = {https://doi.org/10.4153/S0008414X24000853},
}

@incollection{Keen1974,
    AUTHOR = {Keen, Linda},
     TITLE = {Collars on {R}iemann surfaces},
 BOOKTITLE = {Discontinuous groups and {R}iemann surfaces ({P}roc. {C}onf.,
              {U}niv. {M}aryland, {C}ollege {P}ark, {M}d., 1973)},
    SERIES = {Ann. of Math. Stud.},
    VOLUME = {No. 79},
     PAGES = {263--268},
 PUBLISHER = {Princeton Univ. Press, Princeton, NJ},
      YEAR = {1974}
}

@article {Tanigawa1997,
    AUTHOR = {Tanigawa, Harumi},
     TITLE = {Grafting, harmonic maps and projective structures on surfaces},
   JOURNAL = {J. Differential Geom.},
  FJOURNAL = {Journal of Differential Geometry},
    VOLUME = {47},
      YEAR = {1997},
    NUMBER = {3},
     PAGES = {399--419}
}

@article {ChoiDumasRafi2012,
    AUTHOR = {Choi, Young-Eun and Dumas, Emily and Rafi, Kasra},
     TITLE = {Grafting rays fellow travel {T}eichm{\"u}ller geodesics},
   JOURNAL = {Int. Math. Res. Not. IMRN},
  FJOURNAL = {International Mathematics Research Notices. IMRN},
      YEAR = {2012},
    NUMBER = {11},
     PAGES = {2445--2492}
}

@misc{CremaschiGiovannini2025,
    title={Behaviour of the {S}chwarzian derivative on long complex projective tubes}, 
    author={Tommaso Cremaschi and Viola Giovannini},
    year={2025},
    eprint={2502.10071},
    archivePrefix={arXiv}
}

@article {BridgemanBrockBromberg2019,
    AUTHOR = {Bridgeman, Martin and Brock, Jeffrey and Bromberg, Kenneth},
     TITLE = {Schwarzian derivatives, projective structures, and the
              {W}eil-{P}etersson gradient flow for renormalized volume},
   JOURNAL = {Duke Math. J.},
  FJOURNAL = {Duke Mathematical Journal},
    VOLUME = {168},
      YEAR = {2019},
    NUMBER = {5},
     PAGES = {867--896}
}

@article {RafiTao2013,
    AUTHOR = {Rafi, Kasra and Tao, Jing},
     TITLE = {The diameter of the thick part of moduli space and
              simultaneous {W}hitehead moves},
   JOURNAL = {Duke Math. J.},
  FJOURNAL = {Duke Mathematical Journal},
    VOLUME = {162},
      YEAR = {2013},
    NUMBER = {10},
     PAGES = {1833--1876}
}

@article {Ahlfors1961,
    AUTHOR = {Ahlfors, Lars V.},
     TITLE = {Some remarks on {T}eichm{\"u}ller's space of {R}iemann surfaces},
   JOURNAL = {Ann. of Math. (2)},
  FJOURNAL = {Annals of Mathematics. Second Series},
    VOLUME = {74},
      YEAR = {1961},
     PAGES = {171--191}
}

@article {Chu1976,
    AUTHOR = {Chu, Tienchen},
     TITLE = {The {W}eil-{P}etersson metric in the moduli space},
   JOURNAL = {Chinese J. Math.},
  FJOURNAL = {Chinese Journal of Mathematics},
    VOLUME = {4},
      YEAR = {1976},
    NUMBER = {2},
     PAGES = {29--51}
}

@article {Masur1976,
    AUTHOR = {Masur, Howard},
     TITLE = {Extension of the {W}eil-{P}etersson metric to the boundary of
              {T}eichmuller space},
   JOURNAL = {Duke Math. J.},
  FJOURNAL = {Duke Mathematical Journal},
    VOLUME = {43},
      YEAR = {1976},
    NUMBER = {3},
     PAGES = {623--635}
}

@article {Wolpert1987,
    AUTHOR = {Wolpert, Scott A.},
     TITLE = {Geodesic length functions and the {N}ielsen problem},
   JOURNAL = {J. Differential Geom.},
  FJOURNAL = {Journal of Differential Geometry},
    VOLUME = {25},
      YEAR = {1987},
    NUMBER = {2},
     PAGES = {275--296}
}

@article {Tromba1986,
    AUTHOR = {Tromba, A. J.},
     TITLE = {On a natural algebraic affine connection on the space of
              almost complex structures and the curvature of {T}eichm{\"u}ller
              space with respect to its {W}eil-{P}etersson metric},
   JOURNAL = {Manuscripta Math.},
  FJOURNAL = {Manuscripta Mathematica},
    VOLUME = {56},
      YEAR = {1986},
    NUMBER = {4},
     PAGES = {475--497}
}

@article {Wolpert1986a,
    AUTHOR = {Wolpert, Scott A.},
     TITLE = {Chern forms and the {R}iemann tensor for the moduli space of
              curves},
   JOURNAL = {Invent. Math.},
  FJOURNAL = {Inventiones Mathematicae},
    VOLUME = {85},
      YEAR = {1986},
    NUMBER = {1},
     PAGES = {119--145}
}

@inproceedings {Royden1975,
    AUTHOR = {Royden, H. L.},
     TITLE = {Intrinsic metrics on {T}eichm{\"u}ller space},
 BOOKTITLE = {Proceedings of the {I}nternational {C}ongress of
              {M}athematicians ({V}ancouver, {B}.{C}., 1974), {V}ol. 2},
     PAGES = {217--221},
 PUBLISHER = {Canad. Math. Congr., Montreal, QC},
      YEAR = {1975}
}

@article {DisarloParlier2019,
    AUTHOR = {Disarlo, Valentina and Parlier, Hugo},
     TITLE = {The geometry of flip graphs and mapping class groups},
   JOURNAL = {Trans. Amer. Math. Soc.},
  FJOURNAL = {Transactions of the American Mathematical Society},
    VOLUME = {372},
      YEAR = {2019},
    NUMBER = {6},
     PAGES = {3809--3844}
}

@article {DisarloParlier2018,
    AUTHOR = {Disarlo, Valentina and Parlier, Hugo},
     TITLE = {Simultaneous flips on triangulated surfaces},
   JOURNAL = {Michigan Math. J.},
  FJOURNAL = {Michigan Mathematical Journal},
    VOLUME = {67},
      YEAR = {2018},
    NUMBER = {3},
     PAGES = {451--464}
}

@article {LackenbyPurcell2024,
    AUTHOR = {Lackenby, Marc and Purcell, Jessica S.},
     TITLE = {The triangulation complexity of fibred 3-manifolds},
   JOURNAL = {Geom. Topol.},
  FJOURNAL = {Geometry \& Topology},
    VOLUME = {28},
      YEAR = {2024},
    NUMBER = {4},
     PAGES = {1727--1828}
}

@article {MasurMinsky1999,
    AUTHOR = {Masur, Howard A. and Minsky, Yair N.},
     TITLE = {Geometry of the complex of curves. {I}. {H}yperbolicity},
   JOURNAL = {Invent. Math.},
  FJOURNAL = {Inventiones Mathematicae},
    VOLUME = {138},
      YEAR = {1999},
    NUMBER = {1},
     PAGES = {103--149}
}

@article {SleatorTarjanThurston1988,
    AUTHOR = {Sleator, Daniel D. and Tarjan, Robert E. and Thurston, William
              P.},
     TITLE = {Rotation distance, triangulations, and hyperbolic geometry},
   JOURNAL = {J. Amer. Math. Soc.},
  FJOURNAL = {Journal of the American Mathematical Society},
    VOLUME = {1},
      YEAR = {1988},
    NUMBER = {3},
     PAGES = {647--681}
}

@article {ParlierPetri2018,
    AUTHOR = {Parlier, Hugo and Petri, Bram},
     TITLE = {The genus of curve, pants and flip graphs},
   JOURNAL = {Discrete Comput. Geom.},
  FJOURNAL = {Discrete \& Computational Geometry. An International Journal
              of Mathematics and Computer Science},
    VOLUME = {59},
      YEAR = {2018},
    NUMBER = {1},
     PAGES = {1--30}
}

@article {Pournin2014,
    AUTHOR = {Pournin, Lionel},
     TITLE = {The diameter of associahedra},
   JOURNAL = {Adv. Math.},
  FJOURNAL = {Advances in Mathematics},
    VOLUME = {259},
      YEAR = {2014},
     PAGES = {13--42}
}

@article {BrockFarb2006,
    AUTHOR = {Brock, Jeffrey and Farb, Benson},
     TITLE = {Curvature and rank of {T}eichm\"uller space},
   JOURNAL = {Amer. J. Math.},
  FJOURNAL = {American Journal of Mathematics},
    VOLUME = {128},
      YEAR = {2006},
    NUMBER = {1},
     PAGES = {1--22}
}

@article {HatcherThurston1980,
    AUTHOR = {Hatcher, A. and Thurston, W.},
     TITLE = {A presentation for the mapping class group of a closed
              orientable surface},
   JOURNAL = {Topology},
  FJOURNAL = {Topology. An International Journal of Mathematics},
    VOLUME = {19},
      YEAR = {1980},
    NUMBER = {3},
     PAGES = {221--237}
}

@article {Margalit2004,
    AUTHOR = {Margalit, Dan},
     TITLE = {Automorphisms of the pants complex},
   JOURNAL = {Duke Math. J.},
  FJOURNAL = {Duke Mathematical Journal},
    VOLUME = {121},
      YEAR = {2004},
    NUMBER = {3},
     PAGES = {457--479}
}

@article {Hatcher1991,
    AUTHOR = {Hatcher, Allen},
     TITLE = {On triangulations of surfaces},
   JOURNAL = {Topology Appl.},
  FJOURNAL = {Topology and its Applications},
    VOLUME = {40},
      YEAR = {1991},
    NUMBER = {2},
     PAGES = {189--194}
}

@book {PrimerBook,
    AUTHOR = {Farb, Benson and Margalit, Dan},
     TITLE = {A primer on mapping class groups},
    SERIES = {Princeton Mathematical Series},
    VOLUME = {49},
 PUBLISHER = {Princeton University Press, Princeton, NJ},
      YEAR = {2012}
}

@article {Basmajian1993Spectrum,
    AUTHOR = {Basmajian, Ara},
     TITLE = {The orthogonal spectrum of a hyperbolic manifold},
   JOURNAL = {Amer. J. Math.},
  FJOURNAL = {American Journal of Mathematics},
    VOLUME = {115},
      YEAR = {1993},
    NUMBER = {5},
     PAGES = {1139--1159}
}

@article {ManmanJiang2021,
    AUTHOR = {Jiang, Manman},
     TITLE = {Ideal triangulation and minimal shearing of punctured
              hyperbolic surfaces},
   JOURNAL = {J. Math. Anal. Appl.},
  FJOURNAL = {Journal of Mathematical Analysis and Applications},
    VOLUME = {495},
      YEAR = {2021},
    NUMBER = {2}
}

@article {Mondello2009,
    AUTHOR = {Mondello, Gabriele},
     TITLE = {Triangulated {R}iemann surfaces with boundary and the
              {W}eil-{P}etersson {P}oisson structure},
   JOURNAL = {J. Differential Geom.},
  FJOURNAL = {Journal of Differential Geometry},
    VOLUME = {81},
      YEAR = {2009},
    NUMBER = {2},
     PAGES = {391--436},
      ISSN = {0022-040X,1945-743X},
   MRCLASS = {32G15 (30F60 53C22 53D17 57M50)},
  MRNUMBER = {2472178},
MRREVIEWER = {Athanase\ Papadopoulos},
       URL = {http://projecteuclid.org/euclid.jdg/1231856265},
}

@article {BowditchEpstein1988,
    AUTHOR = {Bowditch, B. H. and Epstein, D. B. A.},
     TITLE = {Natural triangulations associated to a surface},
   JOURNAL = {Topology},
  FJOURNAL = {Topology. An International Journal of Mathematics},
    VOLUME = {27},
      YEAR = {1988},
    NUMBER = {1},
     PAGES = {91--117},
      ISSN = {0040-9383},
   MRCLASS = {57M99 (32G15)},
  MRNUMBER = {935529},
MRREVIEWER = {N.\ V.\ Ivanov},
       DOI = {10.1016/0040-9383(88)90008-0},
       URL = {https://doi.org/10.1016/0040-9383(88)90008-0},
}

@misc{Budd2025,
      title={A tree bijection for the moduli space of genus-0 hyperbolic surfaces with boundaries}, 
      author={Timothy Budd and Thomas Meeusen and Bart Zonneveld},
      year={2025},
      eprint={2512.09722},
      archivePrefix={arXiv}
}

@article {KorkmazPapadopoulos2012,
    AUTHOR = {Korkmaz, Mustafa and Papadopoulos, Athanase},
     TITLE = {On the ideal triangulation graph of a punctured surface},
   JOURNAL = {Ann. Inst. Fourier (Grenoble)},
  FJOURNAL = {Université de Grenoble. Annales de l'Institut Fourier},
    VOLUME = {62},
      YEAR = {2012},
    NUMBER = {4},
     PAGES = {1367--1382},
      ISSN = {0373-0956,1777-5310},
   MRCLASS = {30F10 (20F38 32G15)},
  MRNUMBER = {3025746},
MRREVIEWER = {Milagros\ Izquierdo},
       DOI = {10.5802/aif.2725},
       URL = {https://doi.org/10.5802/aif.2725},
}

@article {Mosher1995,
    AUTHOR = {Mosher, Lee},
     TITLE = {Mapping class groups are automatic},
   JOURNAL = {Ann. of Math. (2)},
  FJOURNAL = {Annals of Mathematics. Second Series},
    VOLUME = {142},
      YEAR = {1995},
    NUMBER = {2},
     PAGES = {303--384},
      ISSN = {0003-486X,1939-8980},
   MRCLASS = {57M07 (20F10 20F32)},
  MRNUMBER = {1343324},
MRREVIEWER = {Athanase\ Papadopoulos},
       DOI = {10.2307/2118637},
       URL = {https://doi.org/10.2307/2118637},
}

@article {Yamada2004,
    AUTHOR = {Yamada, Sumio},
     TITLE = {On the geometry of {W}eil-{P}etersson completion of
              {T}eichm\"uller spaces},
   JOURNAL = {Math. Res. Lett.},
  FJOURNAL = {Mathematical Research Letters},
    VOLUME = {11},
      YEAR = {2004},
    NUMBER = {2-3},
     PAGES = {327--344},
      ISSN = {1073-2780},
   MRCLASS = {32G15 (30F60)},
  MRNUMBER = {2067477},
MRREVIEWER = {Samuel\ Grushevsky},
       DOI = {10.4310/MRL.2004.v11.n3.a5},
       URL = {https://doi.org/10.4310/MRL.2004.v11.n3.a5},
}

@article {DaskalopoulosWentworth2003,
    AUTHOR = {Daskalopoulos, Georgios and Wentworth, Richard},
     TITLE = {Classification of {W}eil-{P}etersson isometries},
   JOURNAL = {Amer. J. Math.},
  FJOURNAL = {American Journal of Mathematics},
    VOLUME = {125},
      YEAR = {2003},
    NUMBER = {4},
     PAGES = {941--975},
      ISSN = {0002-9327,1080-6377},
   MRCLASS = {32G15 (30F60 37D20 37F30 57M50 57N05)},
  MRNUMBER = {1993745},
MRREVIEWER = {Athanase\ Papadopoulos},
       URL =
              {http://muse.jhu.edu/journals/american_journal_of_mathematics/v125/125.4daskalopoulos.pdf},
}

@article {BBB2023,
    AUTHOR = {Bridgeman, Martin and Brock, Jeffrey and Bromberg, Kenneth},
     TITLE = {The {W}eil-{P}etersson gradient flow of renormalized volume
              and 3-dimensional convex cores},
   JOURNAL = {Geom. Topol.},
  FJOURNAL = {Geometry \& Topology},
    VOLUME = {27},
      YEAR = {2023},
    NUMBER = {8},
     PAGES = {3183--3228},
      ISSN = {1465-3060,1364-0380},
   MRCLASS = {32G15 (30F40 30F60 32Q45 51P05)},
  MRNUMBER = {4668096},
       DOI = {10.2140/gt.2023.27.3183},
       URL = {https://doi.org/10.2140/gt.2023.27.3183},
}

@article {ScannelWolf2002,
    AUTHOR = {Scannell, Kevin P. and Wolf, Michael},
     TITLE = {The grafting map of {T}eichm\"uller space},
   JOURNAL = {J. Amer. Math. Soc.},
  FJOURNAL = {Journal of the American Mathematical Society},
    VOLUME = {15},
      YEAR = {2002},
    NUMBER = {4},
     PAGES = {893--927},
      ISSN = {0894-0347,1088-6834},
   MRCLASS = {32G15 (30F10 30F40 30F60 57M50)},
  MRNUMBER = {1915822},
MRREVIEWER = {James\ W.\ Anderson},
       DOI = {10.1090/S0894-0347-02-00395-8},
       URL = {https://doi.org/10.1090/S0894-0347-02-00395-8},
}

@article {FarbLubotzkyMinsky2001,
    AUTHOR = {Farb, Benson and Lubotzky, Alexander and Minsky, Yair},
     TITLE = {Rank-1 phenomena for mapping class groups},
   JOURNAL = {Duke Math. J.},
  FJOURNAL = {Duke Mathematical Journal},
    VOLUME = {106},
      YEAR = {2001},
    NUMBER = {3},
     PAGES = {581--597},
      ISSN = {0012-7094,1547-7398},
   MRCLASS = {20F34 (57M07)},
  MRNUMBER = {1813237},
MRREVIEWER = {Richard\ Kenyon},
       DOI = {10.1215/S0012-7094-01-10636-4},
       URL = {https://doi.org/10.1215/S0012-7094-01-10636-4},
}

@article {LiuShigaSun2014,
    AUTHOR = {Liu, LiXin and Shiga, Hiroshige and Sun, ZongLiang},
     TITLE = {Convex hull of set in thick part of {T}eichm\"uller space},
   JOURNAL = {Sci. China Math.},
  FJOURNAL = {Science China. Mathematics},
    VOLUME = {57},
      YEAR = {2014},
    NUMBER = {9},
     PAGES = {1799--1810},
      ISSN = {1674-7283,1869-1862},
   MRCLASS = {32G15 (30F60)},
  MRNUMBER = {3249392},
MRREVIEWER = {Athanase\ Papadopoulos},
       DOI = {10.1007/s11425-014-4871-6},
       URL = {https://doi.org/10.1007/s11425-014-4871-6},
}

@misc{wu2026,
      title={The {W}eil-{P}etersson diameter of moduli space for large genus},
      author={Yunhui Wu},
      year={2026},
      eprint={2610.19707},
      archivePrefix={arXiv},
      primaryClass={math.DG}
}

@misc{AbadieWolpertShear,
      title={A {W}olpert theorem for shearing coordinates}, 
      author={Marie Abadie},
      year={2026},
      eprint={2611.05045},
      archivePrefix={arXiv}
}

\end{document}